\documentclass[10pt,a4paper,reqno]{amsart}

\makeatletter
\g@addto@macro{\endabstract}{\@setabstract}
\newcommand{\authorfootnotes}{\renewcommand\thefootnote{\@fnsymbol\c@footnote}}%
\makeatother

\usepackage{amssymb}
\usepackage{mathtools}
\usepackage[utf8]{inputenc}
\usepackage{bm}
\usepackage{bbm}
\usepackage{amsmath}
\usepackage{amsfonts}
\makeatletter
\def\amsbb{\use@mathgroup \M@U \symAMSb}
\makeatother

\usepackage{comment}
\usepackage{graphicx}
\usepackage{caption}
\usepackage{subcaption}

\usepackage{bbold}

\newcommand{\bga}{\begin{aligned}}
	\newcommand{\ena}{\end{aligned}}

\newcommand{\bge}{\begin{enumerate}}
	\newcommand{\ene}{\end{enumerate}}

\usepackage{dutchcal}

\usepackage{fancyhdr}
\usepackage{subcaption}

\usepackage{pgfplots}
\pgfplotsset{compat=1.15}

\usepackage{booktabs}
\usepackage{xcolor}

\definecolor{webgreen}{rgb}{0,.5,0}
\definecolor{webbrown}{rgb}{.6,0,0}
\definecolor{RoyalBlue}{cmyk}{1, 0.50, 0, 0}

\newcommand{\red}[1]{{\color{red} #1}}
\newcommand{\blue}[1]{{\color{blue} #1}}
\newcommand{\magenta}[1]{{\color{magenta} #1}}
\newcommand{\gr}[1]{{\color{webgreen} #1}}
\newcommand{\hide}[1]{}

\usepackage{tikz}
\usepackage{pict2e}
\usepackage{todonotes}
\usetikzlibrary{decorations.markings}

\DeclareSymbolFont{bbold}{U}{bbold}{m}{n}
\DeclareSymbolFontAlphabet{\mathbbold}{bbold}

\newcommand{\D}[2]{D_{#1}^{(#2)}}
\newcommand{\cD}[2]{\overset{\circ}{D}_{#1}^{(#2)}}
\newcommand{\de}{\delta}

\newcommand{\C}{{\mathbb C}}

\newcommand{\Z}{{\mathbb Z}}

\newcommand{\N}{{\mathbb N}}
\newcommand{\T}{{\mathbb T}}

\newcommand{\al}{\alpha}
\newcommand{\be}{\beta}

\newcommand{\ze}{\zeta}

\newcommand{\di}{\displaystyle}

\newcommand{\ic}{\textrm{i}}
\newcommand{\dd}{\textrm{d}}

\newcommand{\qasq}{\quad \text{as} \quad}
\newcommand{\qandq}{\quad \text{and} \quad}

\usepackage{etoolbox}

\makeatletter
\pretocmd{\section}{\addtocontents{toc}{\protect\addvspace{1\p@}}}{}{}
\pretocmd{\subsection}{\addtocontents{toc}{\protect\addvspace{1\p@}}}{}{}
\pretocmd{\subsubsection}{\addtocontents{toc}{\protect\addvspace{1\p@}}}{}{}
\makeatother

\usepackage{comment}

\usepackage{amsmath,amssymb,amsthm}

\newtheorem{theorem}{Theorem}
\newtheorem{prop}{Proposition}
\newtheorem{lemma}{Lemma}

\theoremstyle{remark}
\newtheorem{remark}{Remark}

\newtheorem{definition}{Definition}

\usepackage[T1]{fontenc}

\usepackage{mathtools}

\makeatletter
\DeclareRobustCommand\widecheck[1]{{\mathpalette\@widecheck{#1}}}
\def\@widecheck#1#2{%
	\setbox\z@\hbox{\m@th$#1#2$}%
	\setbox\tw@\hbox{\m@th$#1%
		\widehat{%
			\vrule\@width\z@\@height\ht\z@
			\vrule\@height\z@\@width\wd\z@}$}%
	\dp\tw@-\ht\z@
	\@tempdima\ht\z@ \advance\@tempdima2\ht\tw@ \divide\@tempdima\thr@@
	\setbox\tw@\hbox{%
		\raise\@tempdima\hbox{\scalebox{1}[-1]{\lower\@tempdima\box
				\tw@}}}%
	{\ooalign{\box\tw@ \cr \box\z@}}}
\makeatother

\usepackage[mathscr]{euscript}

\usepackage{amsmath}

\ExplSyntaxOn

\NewDocumentCommand{\pFq}{O{}mmmmm}
{
	\group_begin:
	\keys_set:nn { hypergeometric } { #1 }
	\hypergeometric_print:nnnnn { #2 } { #3 } { #4 } { #5 } { #6 }
	\group_end:
}
\NewDocumentCommand{\hypergeometricsetup}{m}
{
	\keys_set:nn { hypergeometric } { #1 }
}

\tl_new:N \l_hypergeometric_divider_tl
\tl_new:N \l_hypergeometric_left_tl
\tl_new:N \l_hypergeometric_right_tl

\keys_define:nn { hypergeometric }
{
	symbol .tl_set:N = \l_hypergeometric_symbol_tl,
	symbol .initial:n = F,
	separator .tl_set:N = \l_hypergeometric_separator_tl,
	separator .initial:n = {},
	skip .tl_set:N = \l_hypergeometric_skip_tl,
	skip .initial:n = 8,
	divider .choice:,
	divider/semicolon .code:n = \tl_set:Nn \l_hypergeometric_divider_tl { \;; },
	divider/bar .code:n = \tl_set:Nn \l_hypergeometric_divider_tl { \;\middle|\; },
	divider .initial:n = semicolon,
	fences .choice:,
	fences/brack .code:n = 
	\tl_set:Nn \l_hypergeometric_left_tl {[}
	\tl_set:Nn \l_hypergeometric_right_tl {]},
	fences/parens .code:n = 
	\tl_set:Nn \l_hypergeometric_left_tl {(}
	\tl_set:Nn \l_hypergeometric_right_tl {)},
	fences .initial:n = brack,
}

\cs_new_protected:Nn \hypergeometric_print:nnnnn
{
	{} \sb {#1} \l_hypergeometric_symbol_tl \sb { #2 }
	\left\l_hypergeometric_left_tl
	\genfrac .. 
	{0pt} 
	{} 
	{ \__hypergeometric_process:n { #3 } } 
	{ \__hypergeometric_process:n { #4 } } 
	\l_hypergeometric_divider_tl
	#5
	\right\l_hypergeometric_right_tl
}

\cs_new_protected:Nn \__hypergeometric_process:n
{
	\clist_use:nn { #1 }
	{
		{\l_hypergeometric_separator_tl}
		\mspace { \l_hypergeometric_skip_tl mu }
	}
}

\ExplSyntaxOff

\usepackage{yfonts}

\usepackage{xcolor}
\definecolor{webgreen}{rgb}{0,.5,0}
\definecolor{webbrown}{rgb}{.6,0,0}
\definecolor{RoyalBlue}{cmyk}{1, 0.50, 0, 0}
\usepackage[colorlinks=true, breaklinks=true, urlcolor=webbrown, linkcolor=RoyalBlue, citecolor=webgreen,backref=page]{hyperref}

\usepackage{amssymb,amsthm,amsxtra}
\usepackage{epic,eepic}
\usepackage{epsfig}

\newcommand\psymmU{%
	\begin{picture}(1,1)(0,0)%
		\allinethickness{0.5pt}%
		\path(0,0)(0,1)(1,1)(1,0)(0,0)%
\end{picture}}
\newcommand\psymmUU{%
	\begin{picture}(1,1)(0,0)%
		\allinethickness{0.5pt}%
		\path(0,0)(0,1)(1,1)(1,0)(0,0)%
		\put(0.5,0.5){\makebox(0,0){$\cdot$}}%
\end{picture}}
\newcommand\psymmO{%
	\begin{picture}(1,1)(0,0)%
		\allinethickness{0.5pt}%
		\path(0,0)(0,1)(1,1)(1,0)(0,0)%
		\path(0,0)(1,1)%
\end{picture}}
\newcommand\psymmS{%
	\begin{picture}(1,1)(0,0)%
		\allinethickness{0.5pt}%
		\path(0,0)(0,1)(1,1)(1,0)(0,0)%
		\path(1,0)(0,1)%
\end{picture}}
\newcommand\psymmu{%
	\begin{picture}(1,1)(0,0)%
		\allinethickness{0.5pt}%
		\path(0,0)(0,1)(1,1)(1,0)(0,0)%
		\path(0,0)(1,1)%
		\path(0,1)(1,0)%
\end{picture}}

\newbox\tsymmUbox
\newbox\tsymmUUbox
\newbox\tsymmObox
\newbox\tsymmSbox
\newbox\tsymmubox
\setbox\tsymmUbox =\hbox{\kern0.75pt\setlength{\unitlength}{6pt}\psymmU \kern0.75pt}
\setbox\tsymmUUbox=\hbox{\kern0.75pt\setlength{\unitlength}{6pt}\psymmUU\kern0.75pt}
\setbox\tsymmObox =\hbox{\kern0.75pt\setlength{\unitlength}{6pt}\psymmO \kern0.75pt}
\setbox\tsymmSbox =\hbox{\kern0.75pt\setlength{\unitlength}{6pt}\psymmS \kern0.75pt}
\setbox\tsymmubox =\hbox{\kern0.75pt\setlength{\unitlength}{6pt}\psymmu \kern0.75pt}

\newbox\symmUbox
\newbox\symmUUbox
\newbox\symmObox
\newbox\symmSbox
\newbox\symmubox
\setbox\symmUbox =\hbox{\kern0.75pt\setlength{\unitlength}{4.5pt}\psymmU \kern0.75pt}
\setbox\symmUUbox=\hbox{\kern0.75pt\setlength{\unitlength}{4.5pt}\psymmUU\kern0.75pt}
\setbox\symmObox =\hbox{\kern0.75pt\setlength{\unitlength}{4.5pt}\psymmO \kern0.75pt}
\setbox\symmSbox =\hbox{\kern0.75pt\setlength{\unitlength}{4.5pt}\psymmS \kern0.75pt}
\setbox\symmubox =\hbox{\kern0.75pt\setlength{\unitlength}{4.5pt}\psymmu \kern0.75pt}

\makeatletter
\newcommand\incircbin
{%
	\mathpalette\@incircbin
}
\newcommand\@incircbin[2]
{%
	\mathbin%
	{%
		\ooalign{\hidewidth$#1#2$\hidewidth\crcr$#1\bigcirc$}%
	}%
}

\newcommand{\oeq}{\incircbin{=}}
\makeatother

\usepackage[normalem]{ulem} 
\DeclareMathOperator{\rank}{rank}

\usepackage{longtable}
\usepackage{array}

\begin{document}
	
	\tikzset{->-/.style={decoration={
				markings,
				mark=at position #1 with {\arrow{latex}}},postaction={decorate}}}
	
	\tikzset{-<-/.style={decoration={
				markings,
				mark=at position #1 with {\arrowreversed{latex}}},postaction={decorate}}}

\title[The $2j-k$ and $j-2k$ Bi-orthogonal Polynomial Systems on the Unit Circle: Further Properties and Riemann-Hilbert Characterizations]{The $2j-k$ and $j-2k$ Bi-orthogonal Polynomials on the Unit Circle: Further Properties and Riemann-Hilbert Characterizations}

\maketitle

\begin{center}
	\authorfootnotes
	Roozbeh Gharakhloo\footnote{Mathematics Department,	University of California Santa Cruz, Santa Cruz, CA, USA.	e-mail: roozbeh@ucsc.edu},
	Nicholas S. Witte\footnote{School of Mathematics and Statistics, Victoria University of Wellington, NZ. e-mail: n.s.witte@protonmail.com}  \par \bigskip
\end{center}

\begin{abstract}
In previous work \cite{GW}, we developed a theory of modulated \(2j-k\)
bi-orthogonal polynomial systems \(\{P_n(z;r),Q_n(z;r)\}\) and
\(j-2k\) bi-orthogonal polynomial systems \(\{R_n(z;r),S_n(z;r)\}\),
which generalize the classical \(j-k\) Toeplitz systems. In the present
paper, we further develop this theory in several directions. We derive
simplified and unified recurrence relations for both families of
polynomials, prove a more transparent Christoffel--Darboux formula, and
give Riemann--Hilbert characterizations of the \(2j-k\) and \(j-2k\)
systems.
\vspace{.05cm}
\end{abstract}

\tableofcontents

\section{Introduction}

This paper is a sequel to \cite{GW}. In that work, we introduced the
$2j-k$ and $j-2k$ bi-orthogonal polynomial systems on the unit circle $\T$ and
developed the basic theory: bordered determinant formulae, recurrence
relations, multiple integral representations, reproducing kernels,
Christoffel--Darboux identities, and associated functions.  
In \cite[Eq.~(2.12)]{GW}, the $2j-k$ multiple-integral functional \(\mathcal{D}_n[\cdot]\) was defined with the slanted Vandermonde-type interaction
\begin{equation}
	\prod_{1\leq j<k\leq n}
	(\zeta_k-\zeta_j)(\zeta_k^{-2}-\zeta_j^{-2}), \quad \zeta_j \in \T,
\end{equation}
where the eigenvalues $\zeta_j$ can be interpreted as mobile charges and are distributed on the unit circle under mutual electrostatic pair-wise interactions and subject to external fixed charges or potentials. 
Averages or matrix integrals take the general form
\begin{equation}
    \mathcal{D}_n[f] := \frac{1}{n!} \int_{\T} \frac{\dd \ze_1}{2 \pi \ic \ze_1} \cdots \int_{\T} \frac{\dd \ze_n}{2 \pi \ic \ze_n}
    f(\ze_1,\ldots,\zeta_n) \prod_{1\leq j<k\leq n} (\ze_k-\ze_j)(\ze^{-2}_k-\ze^{-2}_j), 
\end{equation} 	 
for suitable integrable class functions $f(\{\zeta_j\})$.
The goal of the present paper is to extend that theory further in a direction that is more
directly suited to asymptotic analysis. More precisely, we simplify and
unify several recurrence relations, prove auxiliary determinant identities
needed for later applications, and formulate Riemann--Hilbert problems for
the $2j-k$ and $j-2k$ systems.

The usual Toeplitz theory corresponds to the moment pattern $j-k$ and leads
to the standard bi-orthogonal polynomial systems on the unit circle. The
determinants studied here replace the Toeplitz moment pattern by the
slanted patterns $2j-k$ and $j-2k$. One of the main points of this paper is that many familiar structures from
the Toeplitz/OPUC theory survive in this slanted setting, but in modified
form. In particular, the recurrence relations are no longer the usual
Szeg\H{o}-type three-term recurrences, and the natural Riemann--Hilbert
problems are $3\times 3$ rather than $2\times 2$.

A central motivation comes from random matrix models for moments of
derivatives of characteristic polynomials. Altu\u{g}, Bettin, Petrow,
Rishikesh, and Whitehead studied averages over the symmetry types
USp$(2N)$, SO$(2N)$, and O$^{-}(2N)$ and expressed the leading constants, in the large $N$ asymptotic expansion, in
terms of determinants whose entries are hypergeometric moments
\cite{AltugBettinPetrowRishikeshWhitehead2014}. These determinants have
precisely the kind of shifted/slanted moment structure that motivates the
$2j-k$ and $j-2k$ systems. More recently, Assiotis, Gunes, Keating, and Wei
studied general joint moments of characteristic polynomials and higher
derivatives over USp$(2N)$ and SO$(2N)$, connecting the problem to Laguerre
ensemble averages and Painlev\'e equations \cite{AssiotisGunesKeatingWei2026}.
In the forthcoming application-oriented paper \cite{GW3}, we will use the
theory developed here and in \cite{GW} to give a slant-Toeplitz approach to
this collection of problems.

Closely related slanted determinant structures also appear in other
integrable settings, for instance in Bertola and Bothner's work on
Vorob'ev--Yablonski polynomials and rational solutions of Painlev\'e~II
\cite[Equation (1.4)]{BertolaBothner2015}, and in Barhoumi, Lisovyy, Miller, and Prokhorov's
Riemann--Hilbert analysis of Painlev\'e~III through \(2j-k\) determinantal
representations of Umemura polynomials
\cite[Section~3.5]{BarhoumiLisovyyMillerProkhorov2024}. These connections are not
developed here, but they further motivate the study of slanted determinant
structures.

The Riemann--Hilbert characterizations of the present work should also be viewed in the
larger context of Riemann--Hilbert characterizations of orthogonal and
multiple orthogonal systems. For ordinary orthogonal polynomials on the
real line and ordinary bi-orthogonal polynomials on the unit circle, the
basic Riemann--Hilbert problems are $2\times2$. Once one moves beyond the
standard Hankel and Toeplitz settings, larger matrix problems naturally
appear. This is the case for multiple orthogonal polynomials on the real
line\cite{BK04a}, for multiple orthogonal polynomials on the unit circle
\cite{MinguezCenicerosVanAssche2008}, and
for skew-orthogonal polynomials \cite{Pierce2008}. It is also the case for
Toeplitz+Hankel systems \cite{GI,GI2}. The $3\times3$ Riemann--Hilbert
problems in this paper are another example of this same phenomenon.

Let $\boldsymbol{D}^{(r)}_{n}$ and $\boldsymbol{E}^{(s)}_{n}$ respectively denote the $n\times n$ matrices of $2j-k$ and $j-2k$ structure and by $D^{(r)}_{n}$ and $E^{(s)}_{n}$ denote their determinants:
\begin{equation}\label{Det}
D_{n}^{(r)} := \det \begin{pmatrix}
w_{r} & w_{r-1}   & \cdots & w_{r-n+1} \\
w_{r+2}  & w_{r+1}  & \cdots & w_{r-n+3} \\
\vdots & \vdots &  \vdots & \vdots \\
w_{r+2n-2} & w_{r+2n-3} &  \cdots & w_{r+n-1}
\end{pmatrix} \equiv \underset{0\leq j,k \leq n-1}{\det}\left( w_{2j-k+r} \right), 
\end{equation}
\begin{equation}\label{Det E}
E_{n}^{(s)} := \det \begin{pmatrix}
w_{s} & w_{s-2}  &  \cdots & w_{s-2n+2} \\
w_{s+1}  & w_{s-1} &  \cdots & w_{s-2n+3} \\
\vdots & \vdots &  \vdots & \vdots \\
w_{s+n-1} & w_{s+n-3} &  \cdots & w_{s-n+1} 
\end{pmatrix} \equiv \underset{0\leq j,k \leq n-1}{\det} \left( w_{j-2k+s} \right) ,
\end{equation}
where the elements of the matrices are moments with respect to some measure $\dd \mu$
\begin{equation}
    w_{\ell} =  \int_{\T}\frac{\dd \mu(\ze)}{2\pi \ic \ze} \ze^{-\ell}, \quad\textrm{for}\; \ell \in \Z .
\end{equation}

\begin{definition} \normalfont
	For each offset value $r \in \Z$, define the $2j-k$ sequences of monic polynomials $\{P_{n}(z;r)\}^{\infty}_{n=0}$ and $\{Q_{n}(z;r)\}^{\infty}_{n=0}$,  $\deg P_n(z;r)=\deg Q_n(z;r) = n$, satisfying the \textit{bi-orthogonality}  condition:
	\begin{equation}\label{PQorth}
	\int_{\T} P_{m}(\ze;r) Q_{n}(\ze^{-2};r)\ze^{-r}\frac{\dd \mu(\ze)}{2\pi \ic \ze} = h^{(r)}_{n}\delta_{mn}, \qquad m,n \in \N \cup \{0\},
	\end{equation}
	and for each offset value $s \in \Z$, define the $j-2k$ sequences of monic polynomials $\{R_{n}(z;s)\}^{\infty}_{n=0}$ and $\{S_{n}(z;s)\}^{\infty}_{n=0}$, $\deg R_n(z;s)=\deg S_n(z;s) = n$, satisfying the bi-orthogonality condition:
	\begin{equation}\label{RSorth}
	\int_{\T} R_{m}(\ze^2;s) S_{n}(\ze^{-1};s) \ze^{-s}\frac{\dd \mu(\ze)}{2\pi \ic \ze} = g^{(s)}_{n}\delta_{mn}, \qquad m,n \in \N \cup \{0\},
	\end{equation}
	where $h^{(r)}_{n}$ and $g^{(s)}_{n}$ are the \textit{norms} of the polynomials squared and $\dd\mu(\ze)\equiv w(\ze)\dd \ze$ for some weight function $w(z)$.
\end{definition}

We will give representations of $h^{(r)}_{n}$ and $g^{(s)}_{n}$ as ratios of $2j-k$ and $j-2k$ determinants in Theorems \ref{P and Q exist and are unique} and \ref{R and S exist and are unique}. Notice that the bi-orthogonality condition \eqref{PQorth} is equivalent to the orthogonality relations
\begin{equation}\label{OP1}
\int_{\T} P_{n}(\ze;r) \ze^{-2m-r} \frac{\dd\mu(\ze)}{2\pi \ic \ze} = h^{(r)}_{n} \delta_{mn}, \qquad m=0,1,\cdots, n, 
\end{equation}
and
\begin{equation}\label{OP2}
\int_{\T} Q_n(\ze^{-2};r) \ze^{m-r} \frac{\dd\mu(\ze)}{2\pi \ic \ze} = h^{(r)}_{n} \delta_{mn}, \qquad m=0,1,\cdots, n.
\end{equation}
Similarly, the bi-orthogonality condition \eqref{RSorth} is equivalent to the orthogonality relations
\begin{equation}\label{OP1 R}
\int_{\T} R_{n}(\ze^2;s) \ze^{-m-s} \frac{\dd\mu(\ze)}{2\pi \ic \ze} = g^{(s)}_{n} \delta_{mn},\qquad m=0,1,\cdots, n, 
\end{equation}
and
\begin{equation}\label{OP2 S}
\int_{\T} S_n(\ze^{-1};s) \ze^{2m-s} \frac{\dd\mu(\ze)}{2\pi \ic \ze} = g^{(s)}_{n} \delta_{mn},\qquad m=0,1,\cdots, n.
\end{equation}

\begin{theorem}\label{P and Q exist and are unique}
\cite{GW}	If $D_{n}^{(r)} \neq 0$, the polynomials $P_n(z;r)$ and $Q_n(z;r)$ exist and are uniquely given by
	\begin{equation}\label{OP11}
	P_n(z;r) = \frac{1}{D_{n}^{(r)}} 
	\det \begin{pmatrix}
	w_{r} & w_{r-1}  & \cdots & w_{r-n} \\
	w_{r+2} & w_{r+1} & \cdots & w_{r-n+2} \\
	\vdots & \vdots  & \vdots & \vdots \\
	w_{r+2n-2} & w_{r+2n-3} &  \cdots & w_{r+n-2} \\
	1 & z & \cdots  & z^n
	\end{pmatrix},
	\end{equation}
	and 
	\begin{equation}\label{OP22}
	Q_n(z;r) = \frac{1}{D_{n}^{(r)}} 
	\det \begin{pmatrix}
	w_{r} & w_{r-1}  & \cdots & w_{r-n+1} & 1 \\
	w_{r+2} & w_{r+1} &  \cdots & w_{r-n+3} & z \\
	\vdots & \vdots  & \vdots & \vdots & \vdots \\
	w_{r+2n} & w_{r+2n-1} & \cdots & w_{r+n+1} & z^n
	\end{pmatrix},
	\end{equation}
	from which one can observe that $	h^{(r)}_{n}$ exists and can be written as
	\begin{equation}\label{h}
	h^{(r)}_{n} = \frac{D_{n+1}^{(r)}}{D_{n}^{(r)}}, \qquad n \in \N \cup \{0\}, \qquad D^{(r)}_0 \equiv 1.
	\end{equation}
	Therefore if all $h^{(r)}_{\ell}$ exist and are non-zero for $\ell=0, \cdots, n-1$, then
	\begin{equation}\label{Dets from norms}
	D^{(r)}_{n}=\prod_{\ell=0}^{n-1} h^{(r)}_{\ell}.
	\end{equation}
\end{theorem}

\begin{theorem}\label{R and S exist and are unique}\cite{GW}
	If $E_{n}^{(s)} \neq 0$, the polynomials $R_n(z;s)$ and $S_n(z;s)$ exist and are uniquely given by
	
\begin{equation}\label{OP11 R}
	R_n(z;s) = \frac{1}{E_{n}^{(s)}} \det \begin{pmatrix}
	w_{s} & w_{s-2} & w_{s-4} & \cdots & w_{s-2n} \\
	w_{s+1} & w_{s-1} & w_{s-3} & \cdots & w_{s-2n+1} \\
	\vdots & \vdots & \ddots & \vdots & \vdots \\
	w_{s+n-1} & w_{s+n-3} & w_{s+n-5} & \cdots & w_{s-n-1} \\
	1 & z & z^2 & \cdots  & z^n
	\end{pmatrix},
\end{equation}
and
\begin{equation}\label{OP22 S}
	S_n(z;s) = \frac{1}{E_{n}^{(s)}} \det \begin{pmatrix}
	w_{s} & w_{s-2} & w_{s-4} & \cdots & w_{s-2n+2} & 1 \\
	w_{s+1} & w_{s-1} & w_{s-3} & \cdots & w_{s-2n+3} & z \\
	\vdots & \vdots & \ddots & \vdots & \vdots & \vdots \\
	w_{s+n} & w_{s+n-2} & w_{s+n-4} & \cdots & w_{s-n+2} & z^n
	\end{pmatrix},
\end{equation}
	from which one can observe that $	g^{(s)}_{n}$ exists and can be written as
	\begin{equation}\label{gggggg}
	g^{(s)}_{n} = \frac{E_{n+1}^{(s)}}{E_{n}^{(s)}}, \qquad n \in \N \cup \{0\}, \qquad E^{(s)}_0 \equiv 1.
	\end{equation} 
	Therefore if all $g^{(s)}_{\ell}$ exist for $\ell=0, \cdots, n-1$, then
	\begin{equation}\label{Dets from norms 2}
	E^{(s)}_{n}=\prod_{\ell=0}^{n-1} g^{(s)}_{\ell}.
	\end{equation}
\end{theorem}

In some of the proofs, especially in the derivation of the four-term recurrence relations from the Riemann-Hilbert approach, we use the \textit{Dodgson Condensation Identity} \footnote{which is also known as the \textit{Desnanot–Jacobi} identity or the \textit{Sylvester determinant}  identity.} (see \cite{Abeles,Bressoud} and references therein).   Let $\boldsymbol{\mathscr{M}}$ be an $n \times n$ matrix. By \[ \boldsymbol{\mathscr{M}} \left\lbrace \begin{matrix} j_1& j_2& \cdots & j_{\ell} \\  k_1& k_2& \cdots & k_{\ell} \end{matrix} \right\rbrace, \qandq \mathscr{M} \left\lbrace \begin{matrix} j_1& j_2& \cdots & j_{\ell} \\  k_1& k_2& \cdots & k_{\ell} \end{matrix} \right\rbrace, \]
we respectively mean the $(n-\ell)\times(n-\ell)$ matrix obtained from $\boldsymbol{\mathscr{M}}$ by removing the rows $j_i$ and the columns $k_i$, $1\leq i \leq \ell$, and its corresponding determinant. Although the order of writing the row and column indices is immaterial for this definition, in this work we prefer to respect the order of indices, for example we prefer to write \[ \mathscr{M} \left\lbrace \begin{matrix} 3& 5 \\  1& 4 \end{matrix} \right\rbrace, \] and not \[ \mathscr{M} \left\lbrace \begin{matrix} 5 & 3 \\  1& 4 \end{matrix} \right\rbrace, \qquad \mbox{or} \qquad \mathscr{M} \left\lbrace \begin{matrix} 3& 5 \\  4& 1 \end{matrix} \right\rbrace \qquad \mbox{or} \qquad \mathscr{M} \left\lbrace \begin{matrix} 5 & 3 \\  4& 1 \end{matrix} \right\rbrace, \]
although all of these are the same determinant. Let $\boldsymbol{\mathscr{M}}$ be an $n \times n$ matrix and $\mathscr{M}$ denote its determinant. Also let $j_1 < j_2$ and $k_1 < k_2$. Then the Dodgson Condensation identity reads
\begin{equation}\label{DODGSON}
\mathscr{M} \cdot \mathscr{M}\left\lbrace \begin{matrix} j_1 & j_2 \\  k_1& k_2 \end{matrix} \right\rbrace = \mathscr{M}\left\lbrace \begin{matrix} j_1  \\  k_1 \end{matrix} \right\rbrace \cdot \mathscr{M}\left\lbrace \begin{matrix} j_2  \\  k_2 \end{matrix} \right\rbrace - \mathscr{M}\left\lbrace \begin{matrix} j_1  \\  k_2 \end{matrix} \right\rbrace \cdot \mathscr{M}\left\lbrace \begin{matrix} j_2  \\  k_1 \end{matrix} \right\rbrace.
\end{equation}

\subsection{Main Results}

We now summarize the main results of the paper. The first group of results
revisits the recurrence relations obtained in \cite{GW} and expresses the relevant
coefficients in terms of the norm ratios.

The second group of results concerns auxiliary determinant identities. We
prove a multiple integral formula for a bi-bordered $2j-k$ determinant and
evaluate this determinant in terms of the $Q$-polynomials. We also prove a
shifted Vandermonde identity which enters the proof of the simplified
Christoffel--Darboux formula. These identities are useful in their own
right, and they will also be used in our forthcoming application-oriented paper \cite{GW3}.

The third and main group of results gives Riemann--Hilbert
characterizations of the \(2j-k\) and \(j-2k\) systems. More precisely, we
formulate \(3\times3\) Riemann--Hilbert problems whose solutions, when they
exist, recover the \(P_n\)-polynomials and their associated Cauchy-type
transforms \(G_n\) in the \(2j-k\) case, and the \(S_n\)-polynomials and
their associated transforms \(H_n\) in the \(j-2k\) case. The size of these
problems reflects the four-term recurrence structure of the corresponding
polynomial systems. Finally, we show how the recurrence relations can be
recovered from the Riemann--Hilbert formulation.

To present these results, let us consider the following associated functions:
\begin{equation}\label{G}
	G_n(z;r) := 2z^{-2n} \int_{\T} P_n(\ze;r) w(\ze) \ze^{-r} \frac{ \ze^2}{\ze^2-z^2} \frac{\dd \ze}{2\pi \ic \ze},
\end{equation}
and 
 \begin{equation}\label{H}
	H_n(z;s) := 2\int_{\T} S_n(\ze^{-1};s) w(\ze) \ze^{-s} \frac{1}{\ze^2-z^2} \frac{\dd \ze}{2\pi \ic \ze}.
\end{equation} Let \begin{equation}\label{gg}
\mathfrak{g}^{(r)}_n :=  -2 (-1)^n \frac{D^{(r-2)}_{n+1}}{D^{(r)}_n},
\end{equation} and further assume that $w(z)$, $m$ and $r$ are such that $\mathfrak{g}^{(r)}_{2m-2} \neq 0$. Below we use the notation  for the $P$ polynomials defined in \eqref{OP11}:
\begin{equation}
P_n(z;r) = z^n + \sum_{\ell=1}^{n} p^{(r)}_{n,\ell} z^{n-\ell}.
\end{equation}

Consider the following matrix-valued function:
  \begin{equation}\label{YYYY}
 	\boldsymbol{Y}_{m}(z;r):= \boldsymbol{A}^{-1}_{m}(r) \begin{pmatrix}
 		P_{2m}(z;r) & \mathscr{D}_{2m}(z;r) & z^4G_{2m}(z;r) \\[10pt]
 		P_{2m-1}(z;r) & \mathscr{D}_{2m-1}(z;r) & z^2G_{2m-1}(z;r) \\[10pt]
 		P_{2m-2}(z;r) & \mathscr{D}_{2m-2}(z;r) & G_{2m-2}(z;r) \\	
 	\end{pmatrix},
 \end{equation}
where \begin{equation}\label{DD}
	\mathscr{D}_k(z;r):=\frac{1}{2z}\big(P_k(z;r)-P_k(-z;r)\big),
\end{equation} and
\begin{equation}\label{A A A}
	\boldsymbol{A}_{m}(r) =  \begin{pmatrix}
		1 & p^{(r)}_{2m,1} & \mathfrak{g}^{(r)}_{2m} \\
		0 & 1 & \mathfrak{g}^{(r)}_{2m-1}\\
		0 & 0 & \mathfrak{g}^{(r)}_{2m-2}\\	
	\end{pmatrix}.
\end{equation}

In \S \ref{sec 2j-k} we prove
\begin{theorem}\label{thm P} Suppose that the symbol $w$ and $r\in \Z$ are such that $D^{(r)}_n \neq 0$ for $n\in \{2m, 2m-1,2m-2\}$ and also $D^{(r-2)}_{2m-1} \neq 0$. Then  $\boldsymbol{Y}_{m}(z;r)$ is uniquely well-defined by \eqref{YYYY} and satisfies the following Riemann-Hilbert problem
	
	\begin{itemize}
		\item \textbf{RH}-$\boldsymbol{Y1}$ \quad The matrix-valued function $\boldsymbol{Y}_m(\cdot;r): \C \setminus \T \to \C^{3 \times 3}$ is holomorphic.
		\item \textbf{RH}-$\boldsymbol{Y2}$ \quad  $\boldsymbol{Y}_m(z;r)$ satisfies the following jump condition on the unit circle
		\begin{equation}\label{jump Y}
			\boldsymbol{Y}_{m,+}(z;r) = \boldsymbol{Y}_{m,-}(z;r) \begin{pmatrix}
				1 & 0 & \mathcal{j}_1(z;m,r) \\
				0 & 1 & \mathcal{j}_2(z;m,r)  \\
				0 & 0 & 1\\	
			\end{pmatrix}, \qquad z \in \T,
		\end{equation}
		where
		\begin{equation}
			\mathcal{j}_1(z;m,r) :=	z^{-4m-r+4} \left[w(z)+(-1)^rw(-z)\right], \qandq \mathcal{j}_2(z;m,r) := 2(-1)^{r+1} z^{-4m-r+5}w(-z).
		\end{equation}
		
		\item \textbf{RH}-$\boldsymbol{Y3}$ \quad  $\boldsymbol{Y}_m(z;r)$ satisfies
		\begin{equation}
			\boldsymbol{Y}_{m}(z;r)= \left(\di I+\frac{ \overset{\infty}{  \boldsymbol{Y}}_1(m,r)}{z}+\frac{\overset{\infty}{\boldsymbol{Y}}_2(m,r)}{z^2} + O(z^{-3})\right)
			\begin{pmatrix}
				z^{2m} & 0 & 0 \\
				0 & z^{2m-2}  & 0\\
				0 & 0 &  z^{-4m+2} \\	
			\end{pmatrix}, \qasq z \to \infty,
		\end{equation}
		where 	\begin{equation}\label{Y_1 is nice!}
		\overset{\infty}{  \boldsymbol{Y}}_1(m,r) \equiv  \begin{pmatrix}
			0 & 0 & 0 \\
			1 & 0  & 0\\
			0 & 0 &  0 \\	
		\end{pmatrix}, \qandq \overset{\infty}{\boldsymbol{Y}}_{2,31}(m,r)\neq0.
	\end{equation}
	\end{itemize}
\end{theorem}
A standard Liouville theorem-based argument gives:
\begin{lemma}
\label{lemma:Y uniqueness}
If \textbf{RH}-\(\boldsymbol{Y1}\)--\textbf{RH}-\(\boldsymbol{Y3}\) is
solvable, then its solution is unique. Moreover,
\[
	\det \boldsymbol{Y}_m(z;r)\equiv1.
\]
\end{lemma}
We also prove the converse of Theorem~\ref{thm P} and
thereby complete the Riemann--Hilbert characterization of the polynomials
\(P_n(z;r)\) and their associated functions $G_n(z;r)$. The following lemma shows that imposing the condition \eqref{Y_1 is nice!}
forces the second column to be the odd part of the first column divided by \(z\). This is necessary to prove the converse of Theorem \ref{thm P} and will be proven in Section \ref{sec converse thm P}.
\begin{lemma}
\label{prop:Y column symmetry}
Assume that \textbf{RH}-\(\boldsymbol{Y1}\)--\textbf{RH}-\(\boldsymbol{Y3}\)
is solvable. Then
\begin{equation}\label{Y z minus z symmetry}
	\boldsymbol{Y}_m(z;r)
	=
	\boldsymbol{Y}_m(-z;r)
	\begin{pmatrix}
		1 & 0 & 0\\
		2z & 1 & 0\\
		0 & 0 & 1
	\end{pmatrix}.
\end{equation}
In particular, for each \(j=1,2,3\),
\begin{equation}\label{Y second column from first column}
	Y_{j2}(z;r)
	=
	\frac{Y_{j1}(z;r)-Y_{j1}(-z;r)}{2z}.
\end{equation}
\end{lemma}

Here is the converse of Theorem \ref{thm P}.

\begin{theorem}
\label{thm:Y converse P entries}
Assume that \textbf{RH}-\(\boldsymbol{Y1}\)--\textbf{RH}-\(\boldsymbol{Y3}\)
is solvable. Then
\[
	D_{2m}^{(r)}D_{2m-1}^{(r)}D_{2m-2}^{(r)}\neq0.
\]
Consequently, the polynomials
\[
	P_{2m}(z;r),\qquad P_{2m-1}(z;r),\qquad P_{2m-2}(z;r)
\]
exist and are unique. Moreover, the matrix \(\boldsymbol{A}_m(r)\) defined
by \eqref{A A A} is well-defined, and the polynomials \(P_n(z;r)\) together
with the associated functions \(G_n(z;r)\), for
\(n\in\{2m,2m-1,2m-2\}\), are characterized by the solution of the
\(\boldsymbol{Y}_m\)-RHP through
\[
	\begin{pmatrix}
		P_{2m}(z;r) & \mathscr{D}_{2m}(z;r) & z^4G_{2m}(z;r) \\[8pt]
		P_{2m-1}(z;r) & \mathscr{D}_{2m-1}(z;r) & z^2G_{2m-1}(z;r) \\[8pt]
		P_{2m-2}(z;r) & \mathscr{D}_{2m-2}(z;r) & G_{2m-2}(z;r)
	\end{pmatrix}
	=
	\boldsymbol{A}_m(r)\boldsymbol{Y}_m(z;r).
\]
\end{theorem}

In Section \ref{sec rec rel from RHP} we derive a four-term recurrence relation, in $n$, for $P_n(z;r)$ and illustrate its compatibility with the four-term recurrence relations for $P_n(z;r)$ obtained via Dodgson condensation identities.

We also prove an analogous result for the $j-2k$ polynomials. Consider the following matrix-valued function constructed from the $S$ polynomials and their associated functions $H$ defined in \eqref{OP22 S} and \eqref{H}:
 
 \begin{equation}\label{XXXXXX}
 	\boldsymbol{X}_{m}(z;s):= \boldsymbol{B}^{-1}_{m}(s) \begin{pmatrix}
 		S^*_{2m}(z;s) &  \di \frac{1}{2z}\left(S^*_{2m}(z;s)-S^*_{2m}(-z;s)\right) & H_{2m}(z;s) \\[10pt]
 		z S^*_{2m-1}(z;s) & \di \frac{1}{2} \left(S^*_{2m-1}(z;s)+S^*_{2m-1}(-z;s)\right) & H_{2m-1}(z;s) \\[10pt]
 		z^2 S^*_{2m-2}(z;s) & \di \frac{z}{2}\left(S^*_{2m-2}(z;s)-S^*_{2m-2}(-z;s)\right) & H_{2m-2}(z;s) \\	
 	\end{pmatrix},
 \end{equation}
where
\begin{equation}\label{BBB}
	\boldsymbol{B}_{m}(s) := \begin{pmatrix}
		S_{2m}(0;s)  &  \di S'_{2m}(0;s)  & 0 \\[10pt]
		S_{2m-1}(0;s) & \di S'_{2m-1}(0;s) & 0 \\[10pt]
		S_{2m-2}(0;s) & \di S'_{2m-2}(0;s) & -2 g^{(s)}_{2m-2} \\	
	\end{pmatrix}, \qquad g^{(s)}_{n} = \frac{E_{n+1}^{(s)}}{E_{n}^{(s)}},
\end{equation}
and $f_n^*(z) := z^n f_n(z^{-1})$ for any polynomial $f_n$ of degree $n$, and is sometimes referred to as the \textit{reciprocal polynomial}. In order to make sure $	\boldsymbol{B}_{m}(s) $ is invertible we need $ \det \begin{pmatrix} S_{2m}(0;s) & S'_{2m}(0;s) \\ S_{2m-1}(0;s) & S'_{2m-1}(0;s) \end{pmatrix} \neq 0$. This condition can be written in terms of the original weight $w(z)$ and the determinants it generates; indeed
\begin{equation}\label{S 0 S' 0}
	S_{2m}(0;s)  =  \frac{E^{(s+1)}_{2m}}{E^{(s)}_{2m}}, \qquad S_{2m-1}(0;s)  = - \frac{E^{(s+1)}_{2m-1}}{E^{(s)}_{2m-1}}, \qquad 	S'_{2m}(0;s)  = - \frac{\widehat{E}^{(s)}_{2m}}{E^{(s)}_{2m}}, \qquad S'_{2m-1}(0;s)  =  \frac{\widehat{E}^{(s)}_{2m-1}}{E^{(s)}_{2m-1}},
\end{equation}
where $\widehat{E}^{(s)}_n$ is the following $n \times n$ \textit{bordered} $j-2k$ determinant:
\begin{equation}\label{OP22 S E}
\widehat{E}^{(s)}_n =  \det \begin{pmatrix}
		w_{s} & w_{s-2} & w_{s-4} & \cdots & w_{s-2n+2}  \\
		w_{s+2} & w_{s} & w_{s-2} & \cdots & w_{s-2n+4}  \\
				w_{s+3} & w_{s+1} & w_{s-1} & \cdots & w_{s-2n+5}  \\
				w_{s+4} & w_{s+2} & w_{s} & \cdots & w_{s-2n+6}  \\
		\vdots & \vdots & \ddots & \vdots & \vdots  \\
		w_{s+n} & w_{s+n-2} & w_{s+n-4} & \cdots & w_{s-n+2} 
	\end{pmatrix},
\end{equation}
In view of this, for a fixed choice of $s$, we consider \textit{generic weights} $w(z)$
for which
\begin{equation}\label{generic weights}
	\frac{E^{(s+1)}_{2m}\widehat{E}^{(s)}_{2m-1} - E^{(s+1)}_{2m-1}\widehat{E}^{(s)}_{2m} }{E^{(s)}_{2m}E^{(s)}_{2m-1}} \neq 0.
\end{equation}
  In \S \ref{sec j-2k} we prove

\begin{theorem}\label{thm Q} Suppose that the symbol $w$ and $s\in \Z$ are such that the nonvanishing condition \eqref{generic weights} holds and furthermore $E^{(s)}_n \neq 0$ for $n\in \{2m, 2m-1,2m-2\}$. Then 
	 $\boldsymbol{X}_{m}(z;s)$ is uniquely well-defined by \eqref{XXXXXX} and satisfies the following Riemann-Hilbert problem
	
	\begin{itemize}
		\item \textbf{RH}-$\boldsymbol{X1}$ \quad The matrix-valued function $\boldsymbol{X}_m(\cdot;s): \C \setminus \T \to \C^{3 \times 3}$ is holomorphic.
		\item \textbf{RH}-$\boldsymbol{X2}$ \quad  $\boldsymbol{X}_m(z;s)$ satisfies the following jump condition on the unit circle
		\begin{equation}
			\boldsymbol{X}_{m,+}(z;s) = \boldsymbol{X}_{m,-}(z;s) \begin{pmatrix}
				1 & 0 & \mathcal{k}_1(z;m,s) \\
				0 & 1 & \mathcal{k}_2(z;m,s)  \\
				0 & 0 & 1\\	
			\end{pmatrix}, \qquad z \in \T,
		\end{equation}
		where
		\begin{equation}
			\mathcal{k}_1(z;m,s) :=	z^{-2m-s-2} \left[w(z)+(-1)^sw(-z)\right], \qandq \mathcal{k}_2(z;m,s) := 2(-1)^{s+1} z^{-2m-s-1}w(-z)
		\end{equation}

		\item \textbf{RH}-$\boldsymbol{X3}$ \quad  $\boldsymbol{X}_m(z;s)$ satisfies
		
		\begin{equation}
			\boldsymbol{X}_{m}(z;s)= \left(\di I+\frac{ \overset{\infty}{  \boldsymbol{X}}_1(m,s)}{z}+\frac{\overset{\infty}{\boldsymbol{X}}_2(m,s)}{z^2} + O(z^{-3})\right)\begin{pmatrix}
				z^{2m} & 0 & 0 \\
				0 & z^{2m-2}  & 0\\
				0 & 0 &  z^{-4m+2} \\	
			\end{pmatrix}, \qasq z \to \infty,
		\end{equation}where 	\begin{equation}\label{X_1 is nice!}
	\overset{\infty}{  \boldsymbol{X}}_1(m,s) \equiv  \begin{pmatrix}
			0 & 0 & 0 \\
			1 & 0  & 0\\
			0 & 0 &  0 \\	
		\end{pmatrix}.
	\end{equation} 	In addition, the third column has the refined expansion
	\begin{equation}\label{X third column refined}
		\boldsymbol{X}^{(3)}_m(z;s)
		=
		z^{-4m+2}
		\left(
			\boldsymbol{e}_3+\frac{\boldsymbol{\phi}}{z^2}+O(z^{-4})
		\right),
		\qquad z\to\infty,
	\end{equation}
	where \(\boldsymbol{e}_3=(0,0,1)^T\). If
	\[
		\boldsymbol{v}(z):=\boldsymbol{X}^{(1)}_m(z;s),
		\qquad
		\boldsymbol{v}_j
		:=
		\frac{1}{j!}
		\frac{\dd^j}{\dd z^j}\boldsymbol{v}(0),
		\qquad j=0,1,2,
	\]
	then $\boldsymbol{X}_m(z;s)$ also satisfies
	\begin{equation}\label{X intrinsic nondegeneracy}
		\rank
		\begin{pmatrix}
			\boldsymbol{v}_0 & \boldsymbol{e}_3 & \boldsymbol{\phi}
		\end{pmatrix}
		=
		\rank
		\begin{pmatrix}
			\boldsymbol{v}_0 & \boldsymbol{v}_1 & \boldsymbol{e}_3
		\end{pmatrix}
		=\rank
		\begin{pmatrix}
			\boldsymbol{v}_0 & \boldsymbol{v}_1 & \boldsymbol{v}_2
		\end{pmatrix}
		=3.
	\end{equation}
        \end{itemize}
\end{theorem}

\begin{lemma}
\label{lemma:X uniqueness}
If \textbf{RH}-\(\boldsymbol{X1}\)--\textbf{RH}-\(\boldsymbol{X3}\) is
solvable, then its solution is unique. Moreover,
\[
	\det \boldsymbol{X}_m(z;s)\equiv1.
\]
\end{lemma}
\begin{lemma}
\label{lemma:X column symmetry}
Assume that \textbf{RH}-\(\boldsymbol{X1}\)--\textbf{RH}-\(\boldsymbol{X3}\)
is solvable, with the same first-order normalization structure at infinity.
Then
\begin{equation}\label{X z minus z symmetry}
	\boldsymbol{X}_m(z;s)
	=
	\boldsymbol{X}_m(-z;s)
	\begin{pmatrix}
		1 & 0 & 0\\
		2z & 1 & 0\\
		0 & 0 & 1
	\end{pmatrix}.
\end{equation}
In particular, for each \(j=1,2,3\),
\begin{equation}\label{X second column from first column}
	X_{j2}(z;s)
	=
	\frac{X_{j1}(z;s)-X_{j1}(-z;s)}{2z}.
\end{equation}
\end{lemma}
The following theorem provides the converse of Theorem~\ref{thm Q} and
thereby completes the Riemann--Hilbert characterization of the polynomials
\(S_n(z;s)\) and their associated functions \(H_n(z;s)\).
\begin{theorem}
\label{thm:X converse}
Assume that \textbf{RH}-\(\boldsymbol{X1}\)--\textbf{RH}-\(\boldsymbol{X3}\)
is solvable. Then
\[
	E_{2m}^{(s)}E_{2m-1}^{(s)}E_{2m-2}^{(s)}\neq0.
\]
Consequently, the polynomials
\[
	S_{2m}(z;s),\qquad S_{2m-1}(z;s),\qquad S_{2m-2}(z;s)
\]
exist and are unique. Moreover, the matrix \(\boldsymbol{B}_m(s)\) defined by
\eqref{BBB} is well-defined and invertible. Equivalently, the generic
condition \eqref{generic weights} holds. Finally the polynomials \(S^*_n(z;s)\) together
with the associated functions \(H_n(z;s)\), for
\(n\in\{2m,2m-1,2m-2\}\), are characterized by the solution of the
\(\boldsymbol{X}_m\)-RHP through
\[
	\begin{pmatrix}
 		S^*_{2m}(z;s) &  \di \frac{1}{2z}\left(S^*_{2m}(z;s)-S^*_{2m}(-z;s)\right) & H_{2m}(z;s) \\[10pt]
 		z S^*_{2m-1}(z;s) & \di \frac{1}{2} \left(S^*_{2m-1}(z;s)+S^*_{2m-1}(-z;s)\right) & H_{2m-1}(z;s) \\[10pt]
 		z^2 S^*_{2m-2}(z;s) & \di \frac{z}{2}\left(S^*_{2m-2}(z;s)-S^*_{2m-2}(-z;s)\right) & H_{2m-2}(z;s) \\	
 	\end{pmatrix}=\boldsymbol{B}_m(s)\boldsymbol{X}_m(z;s)
	.
\]
\end{theorem}

\begin{remark}
    The Riemann--Hilbert characterizations of the \(P\)- and
\(S\)-polynomials are enough, in principle, to recover the asymptotics of
the other two polynomial families appearing in the bi-orthogonality
relations. Indeed, if a Deift--Zhou nonlinear steepest descent analysis is
available for the Riemann--Hilbert problems associated with \(P_n\) and
\(S_n\), then the asymptotics of \(Q_n(z^{-2};r)\) and \(R_n(z^2;s)\) follow
from the identities of \cite[Theorem~5.9]{GW}. More precisely, that result
states that
\begin{equation}\label{Q in terms of P GW intro}
	Q_n(z^{2};r)
	=
	\frac{D^{(r+2)}_n}{D^{(r)}_n}
	\frac{z^{2n+1}}{2}
	\left[
	P_{n+1}(z^{-1};r+2)P_n(-z^{-1};r+2)
	-
	P_{n+1}(-z^{-1};r+2)P_n(z^{-1};r+2)
	\right],
\end{equation}
and
\begin{equation}\label{R in terms of S GW intro}
	R_n(z^{2};s)
	=
	\frac{E^{(s-2)}_n}{E^{(s)}_n}
	\frac{z^{2n+1}}{2}
	\left[
	S_{n+1}(z^{-1};s-2)S_n(-z^{-1};s-2)
	-
	S_{n+1}(-z^{-1};s-2)S_n(z^{-1};s-2)
	\right].
\end{equation}

Thus, once the \(P\)- and \(S\)-Riemann--Hilbert problems have been analyzed
asymptotically, the remaining polynomial families \(Q\) and \(R\) can be
obtained through these algebraic identities, together with the corresponding
determinant ratios.
\end{remark}

\begin{remark}
It is worth noting that, despite their inherent asymmetries and more involved
normalization requirements as \(z\to\infty\), the Riemann--Hilbert problems
formulated above are closely related in spirit to the Baik-Deift-Johansson
Riemann--Hilbert characterization of bi-orthogonal polynomials on the unit
circle \cite{BDJ99} (itself inspired by the Fokas-Its-Kitaev
Riemann--Hilbert characterization of orthogonal polynomials on the real line
\cite{FIK92}). In the standard zero-offset OPUC setting, one seeks a
\(2\times 2\) matrix \(Z\) satisfying the following three conditions:
\begin{enumerate}
	\item \(Z\) is analytic in \(\mathbb C\setminus\mathbb T\);
	\item on \(\mathbb T\), its boundary values satisfy a triangular jump of the form
	\[
		Z_+(z)
		=
		Z_-(z)
		\begin{pmatrix}
			1 & z^{-n}w(z)\\
			0 & 1
		\end{pmatrix};
	\]
	\item as \(z\to\infty\),
	\[
		Z(z)
		=
		\left(I+O(z^{-1})\right)
		\begin{pmatrix}
			z^n & 0\\
			0 & z^{-n}
		\end{pmatrix}.
	\]
\end{enumerate}
The \(2j-k\) and \(j-2k\) Riemann--Hilbert problems in this paper have the
same basic structure: analyticity away from \(\mathbb T\), triangular jump
matrices whose off-diagonal entries contain the relevant weights, and
normalization conditions at infinity encoding the degrees of the polynomial
entries. The main difference is that the slanted moment structure leads to
larger matrix problems, modified powers of \(z\) in the jumps and
normalizations, and more delicate asymptotic requirements at infinity; these
asymptotics are what ultimately enforce the conditions needed to recover
the corresponding orthogonality relations as detailed in Sections \ref{sec converse thm P} and \ref{sec converse thm Q}.
\end{remark}

Beyond the Riemann--Hilbert characterizations, in Sections \ref{sec rec rel} -- \ref{sec shifted vandermonde} of this work we prove several other results for the $2j-k$ and $j-2k$ systems. Some are of interest in their own right as they fill the gaps in the general theory, while others will be needed for our forthcoming application-oriented work \cite{GW3}. Among these results, we highlight one concerning the Christoffel-Darboux kernel.

In \cite{GW} the authors derived the following Christoffel-Darboux formula for the $2j-k$ reproducing kernel. 

\begin{theorem}[\cite{GW}]
	The Christoffel-Darboux identity for the $2j-k$ system can be written as
	\begin{equation}\label{CD}
	K_n(z^2_2,z_1;r)=\frac{1}{2}\frac{D_n^{(r+2)}}{D_{n+1}^{(r)}} \frac{z^{2n+1}_2}{z^{2}_1-z^{-2}_2} \det \begin{pmatrix}
	P_n(-z^{-1}_2;r+2) & P_{n+1}(-z^{-1}_2;r+2) & P_{n+2}(-z^{-1}_2;r+2) \\[2pt]
	P_n(z^{-1}_2;r+2) & P_{n+1}(z^{-1}_2;r+2) & P_{n+2}(z^{-1}_2;r+2) \\[2pt]
	P_n(z_1;r+2) & P_{n+1}(z_1;r+2) & P_{n+2}(z_1;r+2) \\				
	\end{pmatrix},
	\end{equation}
	where $K_{n}$ denotes the reproducing kernel:
	\begin{equation}
	K_{n}(z,\mathcal{z};r)  := \sum_{j=0}^{n} \frac{1}{h^{(r)}_{j}}Q_{j}(z;r)P_{j}(\mathcal{z};r). \label{RepKer3}
	\end{equation}
\end{theorem}
The formula \eqref{CD} expresses the Christoffel--Darboux kernel through a
determinantal identity whose expansion is cubic in the \(P\)-polynomials. In
the present paper, we recast the same identity in a more transparent mixed
quadratic \(P\)-\(Q\) form.
\begin{prop}\label{prop CD mixed}
	The Christoffel-Darboux identity for the $2j-k$ system can be written as
	\begin{equation}\label{CD1}
	\begin{split}
	K_n(z_2,z_1;r) & =\frac{1}{z_2z^2_1-1} \left\{ \frac{1}{h^{(r+2)}_n} P_n(z_1;r+2) Q_{n+1}(z_2;r) + \frac{1}{h^{(r)}_n} z_2 P_{n+2}(z_1;r+2) Q_{n}(z_2;r) \right. \\ & + \left.  \frac{1}{h^{(r+2)}_{n+1}} P_{n+1}(z_1;r+2) \left[ Q_{n+2}(z_2;r) - \left(z_2+q^{(r)}_{n+2,1}-q^{(r+2)}_{n+1,1}\right)Q_{n+1}(z_2;r)  \right]   \right\}
	\end{split}
	\end{equation}
\end{prop}

\subsection{Outline}
The paper is organized as follows. In Section~\ref{sec rec rel}, we revisit the
pure-degree recurrence relations from \cite{GW} and rewrite their coefficients
in a simpler form. We also present several identities relating polynomial
coefficients to recurrence coefficients. These quantities reappear later in the
Riemann--Hilbert derivation of the recurrence relations and in the description
of multiplication by \(z^k\) in the polynomial bases \(\{P_n(z;r)\}\) and
\(\{Q_n(z;r)\}\).

Section~\ref{sec mult z} studies the expansions of multiplication by \(z\) and
\(z^2\) in the polynomial bases generated by the \(2j-k\) and \(j-2k\) systems.
These expansion formulae serve as a useful bookkeeping tool: they describe how
the usual multiplication operator acts on the corresponding slanted polynomial
bases.

In Section~\ref{sec bi-bordered}, we prove a multiple integral formula for a
bi-bordered \(2j-k\) determinant and then evaluate this determinant in terms of
the \(Q\)-polynomials. In Section~\ref{sec shifted vandermonde}, we prove a
shifted Vandermonde determinant identity and use it to derive the quadratic
\(P\)-\(Q\) form of the Christoffel--Darboux formula in
Proposition~\ref{prop CD mixed}. These sections provide the determinant tools
needed for our forthcoming applications paper \cite{GW3}.

The final part of the paper is devoted to Riemann--Hilbert problems. In
Sections~\ref{sec 2j-k} and~\ref{sec j-2k}, we prove Riemann--Hilbert
characterizations of the \(2j-k\) and \(j-2k\) systems, respectively. The
forward directions, proved in Sections~\ref{sec forward thm P}
and~\ref{sec forward thm Q}, show that the polynomials defined by the
\(2j-k\) and \(j-2k\) orthogonality relations give rise to solutions of the
corresponding Riemann--Hilbert problems. The converse directions, proved in
Sections~\ref{sec converse thm P} and~\ref{sec converse thm Q}, show that any
solution of the corresponding Riemann--Hilbert problem recovers the relevant
orthogonal polynomials and their associated functions. Finally, in
Section~\ref{sec rec rel from RHP}, we demonstrate how the pure-degree
recurrence relations can be recovered from the Riemann--Hilbert formulation.

\section{Revisiting the Recurrence Relations}\label{sec rec rel}

The recurrence relations obtained in \cite{GW} are one of the basic structural
features of the \(2j-k\) and \(j-2k\) systems. They differ from the classical
OPUC recurrence relations: in the Toeplitz/OPUC setting one has three-term
Szeg\H{o} recurrences \cite{Szego75}, whereas the slanted moment structure of
the \(2j-k\) and \(j-2k\) systems naturally leads to four-term recurrence
relations.

The
purpose of this section is to rewrite these relations in a form in which all recurrence coefficients are expressed in terms of norm ratios. Such relations are advantageous as they allow for large-$n$ asymptotic description of the recurrence coefficients once the large-$n$ asymptotics of the $Y$-RHP is known. Indeed, as described in Section \ref{sec 2j-k}, we have the following representation of the norms of odd/even index, in terms of the $Y$-RHP data: 	\begin{align}
		h^{(r)}_{2m-1} &=
		\frac{1}{2}\lim_{z\to 0} z^{-2}
		\left(\boldsymbol{A}_m(r)\boldsymbol{Y}_m(z;r)\right)_{23}, \label{h2m-1 from Y}\\
		h^{(r)}_{2m-2} &=
		\frac{1}{2}
		\left(\boldsymbol{A}_m(r)\boldsymbol{Y}_m(0;r)\right)_{33}. \label{h2m-2 from Y}
	\end{align}

Let us recall the following four-term recurrence relations for the $2j-k$ system of bi-orthogonal polynomials.
\begin{theorem}\cite{GW}\label{THM degree rec 2j-k}
	The third-order pure-degree recurrence relations for the $2j-k$ polynomials are given by
	\begin{equation}\label{P pure n rec}
	P_{n+3}(z;r) -(\de_{n+2}^{(r)}+\de_{n+1}^{(r-1)})P_{n+2}(z;r) + (\de_{n+1}^{(r-1)}\de_{n+1}^{(r)} - z^2)P_{n+1}(z;r)+ (\de_n^{(r)}+\eta^{(r-2)}_n)z^2P_{n}(z;r)=0,
	\end{equation}
	and
	\begin{equation}\label{Q* pure n rec}
	\begin{split}
	Q^*_{n+3}(z;r) &  -  (1+\be^{(r)}_{n+2} z) Q^*_{n+2}(z;r) +  (\be^{(r)}_{n+1}+\al^{(r+1)}_{n+1}+\be^{(r+1)}_{n+1}+\al^{(r+2)}_{n+1})z Q^*_{n+1}(z;r) \\ & -(\be^{(r+1)}_{n+1}+\al^{(r+2)}_{n+1})  (\be^{(r)}_{n}+\al^{(r+1)}_{n})z^2 Q^*_{n}(z;r) = 0,
	\end{split}
	\end{equation}
	where

	\noindent\begin{minipage}{.5\linewidth}
		\begin{alignat}{2}
		&\de^{(r)}_{n} &&= -\frac{h^{(r-1)}_n}{h^{(r)}_{n}}, \label{delta} \\
		& \eta^{(r)}_n &&= \frac{D_n^{(r+2)}D_{n+1}^{(r-1)}}{D^{(r)}_{n+1}D^{(r+1)}_n}, \label{eta}
		\end{alignat}	
	\end{minipage}	
	\begin{minipage}{.5\linewidth}
		\begin{alignat}{2}
		&\be^{(r)}_{n} &&= -\frac{h^{(r+2)}_n}{h^{(r)}_{n}}, \label{beta} \\
		& \al^{(r)}_n&&= \frac{D^{(r-1)}_n D^{(r+2)}_{n+1}}{D^{(r)}_{n+1}D^{(r+1)}_n}. \label{alpha}
		\end{alignat}	
	\end{minipage}
	
\end{theorem}

In this section our objective is to simplify the coefficients appearing in the above recurrence relation. We prove:
\begin{prop}
    The third-order pure-degree recurrence relations for the $2j-k$ polynomials are given by
    \begin{equation}\label{P rec rel new}
		P_{n+3}(z;r) +\left(\frac{h^{(r-1)}_{n+2}}{h^{(r)}_{n+2}}+ \frac{h^{(r-2)}_{n+1}}{h^{(r-1)}_{n+1}} \right)P_{n+2}(z;r) + \left(\frac{h^{(r-2)}_{n+1}}{h^{(r)}_{n+1}} - z^2\right)P_{n+1}(z;r)-\frac{h_{n+1}^{(r-2)}}{h_{n}^{(r)}} z^2P_{n}(z;r)=0, 
\end{equation} and 
	\begin{equation}\label{Q* pure n rec 1}
Q^*_{n+3}(z;r)  -  \left(1-\frac{h^{(r+2)}_{n+2}}{h^{(r)}_{n+2}} z\right) Q^*_{n+2}(z;r) - \left(\frac{h_{n+2}^{(r+1)}}{h_{n+1}^{(r)}}+\frac{h_{n+2}^{(r+2)}}{h_{n+1}^{(r+1)}}\right)z Q^*_{n+1}(z;r)  -\frac{h_{n+2}^{(r+2)}}{h_{n}^{(r)}}z^2 Q^*_{n}(z;r) = 0
\end{equation}
or
\begin{equation}
	Q_{n+3}(z;r)  
	- \left( z-\frac{h^{(r+2)}_{n+2}}{h^{(r)}_{n+2}} \right) Q_{n+2}(z;r) 
	- \left( \frac{h_{n+2}^{(r+1)}}{h_{n+1}^{(r)}}+\frac{h_{n+2}^{(r+2)}}{h_{n+1}^{(r+1)}} \right)z Q_{n+1}(z;r)  
	- \frac{h_{n+2}^{(r+2)}}{h_{n}^{(r)}}z Q_{n}(z;r) = 0 .
\label{Qn-recur}
\end{equation}
\end{prop}
\begin{proof}
\begin{align}
	\delta_{n}^{(r)}+\eta_{n}^{(r-2)} & =-\frac{h_{n}^{(r-1)}}{h_{n}^{(r)}}+\frac{D_{n}^{(r)}}{D_{n}^{(r-1)} }\frac{D_{n+1}^{(r-3)}}{D_{n+1}^{(r-2)}} \nonumber \\
	& =-\frac{D_{n+1}^{(r-1)}}{D_{n}^{(r-1)}} \cdot \frac{D_{n}^{(r)}}{D_{n+1}^{(r)}}+\frac{D_{n}^{(r)} D_{n+1}^{(r-3)}}{D_{n}^{(r-1)} D_{n+1}^{(r-2)}}  \nonumber \\
	& =\frac{D_{n}^{(r)}}{D_{n}^{(r-1)}}\left\{-\frac{D_{n+1}^{(r-1)}}{D_{n+1}^{(r)}}+\frac{D_{n+1}^{(r-3)}}{D_{n+1}^{(r-2)}}\right\} \nonumber \\
	& =\frac{D_{n}^{(r)}}{D_{n}^{(r-1)}} \frac{1}{D_{n+1}^{(r)} D_{n+1}^{(r-2)}}\left(-D_{n+1}^{(r-1)} D_{n+1}^{(r-2)}+D_{n+1}^{(r)} D_{n+1}^{(r-3)}\right) . \label{delta r + eta r-2}
\end{align}
Now let us recall the Dodgson Condensation identity written in equation (6.35) of \cite{GW}
\begin{equation}
	D^{(r-2)}_{n+2} D^{(r-1)}_{n} = D^{(r-2)}_{n+1} D^{(r-1)}_{n+1} - D^{(r-3)}_{n+1} D_{n+1}.
\end{equation}
Using this, we can simplify \eqref{delta r + eta r-2} as

\begin{equation}\label{de n r + eta n r-2}
\delta_{n}^{(r)}+\eta_{n}^{(r-2)} 
 = - \frac{D_{n}^{(r)}D^{(r-2)}_{n+2}}{D_{n+1}^{(r)} D_{n+1}^{(r-2)}} = - \frac{h^{(r-2)}_{n+1}}{h^{(r)}_{n}} 
\end{equation}
Also
\begin{align} \prod_{l=0}^{n-1}\left(\delta_{l}^{(r)}+\eta_{l}^{(r-2)}\right)& =\prod_{l=0}^{n-1}(-) \frac{h_{l+1}^{(r-2)}}{h_{l}^{(r)}} \nonumber \\
	& =(-)^{n} \frac{D_{n+1}^{(r-2)}}{D_{1}^{(r-2)}} \cdot \frac{D^{(r)}_0}{D_{n}^{(r)}} \nonumber \\
	& =(-)^{n} \frac{D_{n+1}^{(r-2)}}{w_{r-2}D_{n}^{(r)}}.
\end{align}

Recalling (4.115) and (4.116) of \cite{GW}, and using \eqref{de n r + eta n r-2} we have

\begin{align}
\beta_{l}^{(r)}+\alpha_{l}^{(r+1)} & =-\frac{1}{\delta_{l}^{(r+1)} \delta_{l}^{(r+2)}}+\frac{-\eta_{l}^{(r+1)}}{\delta_{l}^{(r+1)} \delta_{l}^{(r+2)} \delta_{l}^{(r+3)}} \nonumber \\
& =-\frac{1}{\delta_{l}^{(r+1)} \delta_{l}^{(r+2)} \delta_{l}^{(r+3)}}\left(\delta_{l}^{(r+3)}+\eta_{l}^{(r+1)}\right) \nonumber \\
& =\frac{1}{\delta_{l}^{(r+1)} \delta_{l}^{(r+2)} \delta_{l}^{(r+3)}} \cdot \frac{h_{l+1}^{(r+1)}}{h_{l}^{(r+3)}} 
\end{align}
Now, recalling \eqref{delta}, we find

\begin{equation}\label{be r + al r+1}
	\beta_{l}^{(r)}+\alpha_{l}^{(r+1)} = - \frac{h^{(r+1)}_{l+1}}{h^{(r)}_{l}}
\end{equation}
Therefore, in view of equation (4.2) of \cite{GW}, we have
\begin{equation}\label{be r + al r+1 2}
	\quad\left(\beta_{l+1}^{(r+1)}+\alpha_{l+1}^{(r+2)}\right)\left(\beta_{l}^{(r)}+\alpha_{l}^{(r+1)}\right)=\frac{h_{l+2}^{(r+2)}}{h_{l}^{(r)}}
\end{equation}
Thus
\begin{align}
	&\prod_{l=0}^{n-1}\left(\beta_{l+1}^{(r+1)}+\alpha_{l+1}^{(r+2)}\right)\left(\beta_{l}^{(r)}+\alpha_{l}^{(r+1)}\right)=\prod_{l=0}^{n-1} \frac{h_{l+2}^{(r+2)}}{h_{l}^{(r)}}=\frac{D_{n+2}^{(r+2)}}{D_{2}^{(r+2)} D_{n}^{(r)}} 
\end{align}
where
\begin{equation*}
	 D_{2}^{(r+2)}=\operatorname{det}\left(\begin{array}{ll}
	w_{r+2} & w_{r+1} \\
	w_{r+4} & w_{r+3}
	\end{array}\right)=w_{r+2} w_{r+3}-w_{r+1} w_{r+4}.
\end{equation*}

Next, we show that 

\begin{equation}
	q_{n+1,1}^{(r)}-q_{n,1}^{(r+2)}=-\frac{h_{n}^{(r+2)}}{h_{n}^{(r)}} \label{q n+1,1 r - q n,1 r+2}
\end{equation}
Indeed, let us recall equation (4.24) of \cite{GW} 
\begin{equation}
	Q_{n}^{*}(z ; r+2)=Q_{n+1}^{*}(z ; r)-\beta_{n}^{(r)} z Q_{n}^{*}(z ; r)
\end{equation}
Now, using the definition of reciprocal polynomials and after some straightforward simplifications we obtain
\begin{equation}
	z^{-1} Q_{n}\left(z^{-1} ; r+2\right)=Q_{n+1}\left(z^{-1} ; r\right)-\beta_{n}^{(r)} Q_{n}\left(z^{-1} ; r\right)
\end{equation} and thus
\begin{equation}
z Q_{n}(z ; r+2)=Q_{n+1}(z ; r)-\beta_{n}^{(r)} Q_{n}(z ; r) 
\end{equation}
Now, the comparison of the large-$z$ expansions of both sides gives 
\begin{equation}
	q_{n,1}^{(r+2)}=q_{n+1,1}^{(r)}-\beta_{n}^{(r)} ,
\end{equation}
which is equivalent to \eqref{q n+1,1 r - q n,1 r+2} once we recall \eqref{beta}. Using the identities \eqref{delta} and \eqref{de n r + eta n r-2} we can write \eqref{P pure n rec} as \eqref{P rec rel new}.

Finally using the identities \eqref{beta}, \eqref{be r + al r+1}, and \eqref{be r + al r+1 2} we can write \eqref{Q* pure n rec} as
\eqref{Q* pure n rec 1}--\eqref{Qn-recur}.
\end{proof}

\subsection{Relating Polynomial Coefficients to Recurrence Coefficients}

In this section we prove several identities connecting the coefficients of
the polynomials with the coefficients appearing in the recurrence relations.
These expressions will appear later in the Riemann--Hilbert derivation of the
recurrence relations in Section~\ref{sec rec rel from RHP}, as well as in the
expansions of multiplication by \(z^k\) in the bases \(\{P_n(z;r)\}\) and
\(\{Q_n(z;r)\}\) in Section~\ref{sec mult z}.

\begin{lemma}
  The following recurrence relations hold for the coefficients of the $2j-k$-bi-orthogonal polynomials:
  \begin{align}
& p^{(r)}_{n+3,1}  - p^{(r)}_{n+1,1} =  \de^{(r)}_{n+2} + \de^{(r-1)}_{n+1} - \de^{(r)}_{n}  -   \eta^{(r-2)}_{n}, \label{z^n+2 coeffs in pure n rec rel 1} \\ & 
  p^{(r)}_{n+3,2}   - p^{(r)}_{n+1,2} =  \left( \de^{(r)}_{n+2} + \de^{(r-1)}_{n+1} \right) p^{(r)}_{n+2,1} - \de^{(r-1)}_{n+1} \de^{(r)}_{n+1}  - \left( \de^{(r)}_{n} + \eta^{(r-2)}_{n}  \right) p^{(r)}_{n,1} \label{z^n+2 coeffs in pure n rec rel p2} \\ 
&  
p^{(r)}_{n+3,3}   - p^{(r)}_{n+1,3} =  \left( \de^{(r)}_{n+2} + \de^{(r-1)}_{n+1} \right) p^{(r)}_{n+2,2} - \de^{(r-1)}_{n+1} \de^{(r)}_{n+1} p^{(r)}_{n+1,1}  - \left( \de^{(r)}_{n} + \eta^{(r-2)}_{n}  \right) p^{(r)}_{n,2} 
\label{z^n coeffs in pure n rec rel 1}
\end{align}
and 
\begin{equation}\label{q1 rec rel}
q^{(r)}_{n+1,1} - q^{(r)}_{n,1}   =  \be^{(r)}_{n} - \be^{(r)}_{n-1}-\al^{(r+1)}_{n-1}-\be^{(r+1)}_{n-1}-\al^{(r+2)}_{n-1},
\end{equation}
\begin{equation}
\begin{split}
q^{(r)}_{n+1,2} - q^{(r)}_{n,2} & =  \be^{(r)}_{n}q^{(r)}_{n,1} - \left(\be^{(r)}_{n-1}+\al^{(r+1)}_{n-1}+\be^{(r+1)}_{n-1}+\al^{(r+2)}_{n-1}\right)q^{(r)}_{n-1,1} \\ & +\left(\be^{(r+1)}_{n-1}+\al^{(r+2)}_{n-1}\right)  \left(\be^{(r)}_{n-2}+\al^{(r+1)}_{n-2}\right),
\end{split}
\end{equation}
\begin{equation}\label{q3 rec rel}
\begin{split}
q^{(r)}_{n+1,3} - q^{(r)}_{n,3}   & =  \be^{(r)}_{n}q^{(r)}_{n,2} - \left(\be^{(r)}_{n-1}+\al^{(r+1)}_{n-1}+\be^{(r+1)}_{n-1}+\al^{(r+2)}_{n-1}\right)q^{(r)}_{n-1,2} \\ & +\left(\be^{(r+1)}_{n-1}+\al^{(r+2)}_{n-1}\right)  \left(\be^{(r)}_{n-2} + \al^{(r+1)}_{n-2}\right)q^{(r)}_{n-2,1}
\end{split}
\end{equation}
\end{lemma}

\begin{proof}
    
Let 

\begin{equation}\label{P}
P_n(z;r) = z^n + \sum_{j=1}^{n} p^{(r)}_{n,j}z^{n-j}.
\end{equation}

\begin{equation}\label{Q}
Q_n(z;r) = z^n + \sum_{j=1}^{n} q^{(r)}_{n,j}z^{n-j}.
\end{equation}
From \eqref{P pure n rec} we have
\begin{equation}\label{P pure n rec*} \begin{split}
0 = & \ z^{n+3} + p^{(r)}_{n+3,1} z^{n+2} + p^{(r)}_{n+3,2} z^{n+1} + p^{(r)}_{n+3,3} z^{n} + \cdots  \\ &   -(\de_{n+2}^{(r)}+\de_{n+1}^{(r-1)}) \left[ z^{n+2} + p^{(r)}_{n+2,1} z^{n+1} + p^{(r)}_{n+2,2} z^{n} + p^{(r)}_{n+2,3} z^{n-1} + \cdots \right] \\ & + (\de_{n+1}^{(r-1)}\de_{n+1}^{(r)} )\left[ z^{n+1} + p^{(r)}_{n+1,1} z^{n} + p^{(r)}_{n+1,2} z^{n-1} + p^{(r)}_{n+1,3} z^{n-2} + \cdots \right] \\ &  - \left[ z^{n+3} + p^{(r)}_{n+1,1} z^{n+2} + p^{(r)}_{n+1,2} z^{n+1} + p^{(r)}_{n+1,3} z^{n} + \cdots \right] \\ & + (\de_n^{(r)}+\eta^{(r-2)}_n)\left[ z^{n+2} + p^{(r)}_{n,1} z^{n+1} + p^{(r)}_{n,2} z^{n} + p^{(r)}_{n,3} z^{n-1} + \cdots \right].
\end{split}
\end{equation}

Matching the coefficients of the monomials gives the desired identities relating the  polynomial coefficients to the recurrence coefficients. Doing so, respectively, for the coefficients of $z^{n+2}$, $z^{n+1}$, and $z^{n}$ yield
\begin{align}
	&  p^{(r)}_{n+3,1} - \de^{(r)}_{n+2} - \de^{(r-1)}_{n+1}+ \de^{(r)}_{n} + \eta^{(r-2)}_{n} - p^{(r)}_{n+1,1} = 0, \\ & 
	 p^{(r)}_{n+3,2} - \left( \de^{(r)}_{n+2} + \de^{(r-1)}_{n+1} \right) p^{(r)}_{n+2,1} + \de^{(r-1)}_{n+1} \de^{(r)}_{n+1}  + \left( \de^{(r)}_{n} + \eta^{(r-2)}_{n}  \right) p^{(r)}_{n,1}  - p^{(r)}_{n+1,2} = 0 \\ & 
	 p^{(r)}_{n+3,3} - \left( \de^{(r)}_{n+2} + \de^{(r-1)}_{n+1} \right) p^{(r)}_{n+2,2} + \de^{(r-1)}_{n+1} \de^{(r)}_{n+1} p^{(r)}_{n+1,1}  + \left( \de^{(r)}_{n} + \eta^{(r-2)}_{n}  \right) p^{(r)}_{n,2}  - p^{(r)}_{n+1,3} = 0 \label{z^n coeffs in pure n rec rel}
\end{align}
which gives the recurrence relations \eqref{z^n+2 coeffs in pure n rec rel 1}--\eqref{z^n coeffs in pure n rec rel 1} which can be solved recursively in order. 

Now we turn our attention to the analogous results for the coefficients $q^{(r)}_{n,j}$. The equation \eqref{Q* pure n rec} can be written as

\begin{equation}\label{Q* pure n rec 11}
\begin{split}
z^{n+3}Q_{n+3}(z^{-1};r) &  -  (1+\be^{(r)}_{n+2} z) z^{n+2}Q_{n+2}(z^{-1};r) +  (\be^{(r)}_{n+1}+\al^{(r+1)}_{n+1}+\be^{(r+1)}_{n+1}+\al^{(r+2)}_{n+1}) z^{n+2}Q_{n+1}(z^{-1};r) \\ & -(\be^{(r+1)}_{n+1}+\al^{(r+2)}_{n+1})  (\be^{(r)}_{n}+\al^{(r+1)}_{n}) z^{n+2}Q_{n}(z^{-1};r) = 0,
\end{split}
\end{equation}
or
\begin{equation}\label{Q* pure n rec 1111}
\begin{split}
Q_{n+3}(z^{-1};r) &  -  (z^{-1}+\be^{(r)}_{n+2} ) Q_{n+2}(z^{-1};r) +  (\be^{(r)}_{n+1}+\al^{(r+1)}_{n+1}+\be^{(r+1)}_{n+1}+\al^{(r+2)}_{n+1}) z^{-1}Q_{n+1}(z^{-1};r) \\ & -(\be^{(r+1)}_{n+1}+\al^{(r+2)}_{n+1})  (\be^{(r)}_{n}+\al^{(r+1)}_{n}) z^{-1}Q_{n}(z^{-1};r) = 0,
\end{split}
\end{equation}
which, after replacing $z \mapsto z^{-1}$, can be written as

\begin{equation}\label{Q* pure n rec 11111}
\begin{split}
Q_{n+3}(z;r) &  -  (z+\be^{(r)}_{n+2} ) Q_{n+2}(z;r) +  (\be^{(r)}_{n+1}+\al^{(r+1)}_{n+1}+\be^{(r+1)}_{n+1}+\al^{(r+2)}_{n+1}) zQ_{n+1}(z;r) \\ & -(\be^{(r+1)}_{n+1}+\al^{(r+2)}_{n+1})  (\be^{(r)}_{n}+\al^{(r+1)}_{n}) zQ_{n}(z;r) = 0,
\end{split}
\end{equation}
From \eqref{Q} we have
\begin{equation}\label{Q* pure n rec 1111111}
\begin{split}
0 = & \ \ z^{n+3} + q^{(r)}_{n+3,1} z^{n+2} + q^{(r)}_{n+3,2} z^{n+1} + q^{(r)}_{n+3,3} z^{n} + \cdots  \\ & -   \left[ z^{n+3} + q^{(r)}_{n+2,1} z^{n+2} + q^{(r)}_{n+2,2} z^{n+1} + q^{(r)}_{n+2,3} z^{n} + \cdots \right] \\ &  -  \be^{(r)}_{n+2}  \left[ z^{n+2} + q^{(r)}_{n+2,1} z^{n+1} + q^{(r)}_{n+2,2} z^{n} + \cdots \right] \\ &   +  \left(\be^{(r)}_{n+1}+\al^{(r+1)}_{n+1}+\be^{(r+1)}_{n+1}+\al^{(r+2)}_{n+1}\right) \left[ z^{n+2} + q^{(r)}_{n+1,1} z^{n+1} + q^{(r)}_{n+1,2} z^{n}  + \cdots \right] \\ & -\left(\be^{(r+1)}_{n+1}+\al^{(r+2)}_{n+1}\right)  \left(\be^{(r)}_{n}+\al^{(r+1)}_{n}\right) \left[ z^{n+1} + q^{(r)}_{n,1} z^{n}  + \cdots \right],
\end{split}
\end{equation}
Now we match the coefficients of $z^{n+2}$, $z^{n+1}$, and $z^{n}$ to respectively obtain

\begin{equation}
	q^{(r)}_{n+3,1} - q^{(r)}_{n+2,1} - \be^{(r)}_{n+2} + \left(\be^{(r)}_{n+1}+\al^{(r+1)}_{n+1}+\be^{(r+1)}_{n+1}+\al^{(r+2)}_{n+1}\right)  = 0,
\end{equation}
\begin{equation}
\begin{split}
	 q^{(r)}_{n+3,2} - q^{(r)}_{n+2,2} - \be^{(r)}_{n+2}q^{(r)}_{n+2,1} & + \left(\be^{(r)}_{n+1}+\al^{(r+1)}_{n+1}+\be^{(r+1)}_{n+1}+\al^{(r+2)}_{n+1}\right)q^{(r)}_{n+1,1} \\ & -\left(\be^{(r+1)}_{n+1}+\al^{(r+2)}_{n+1}\right)  \left(\be^{(r)}_{n}+\al^{(r+1)}_{n}\right)  = 0
\end{split}
\end{equation}
\begin{equation}
\begin{split}
	 q^{(r)}_{n+3,3} - q^{(r)}_{n+2,3} - \be^{(r)}_{n+2}q^{(r)}_{n+2,2} & + \left(\be^{(r)}_{n+1}+\al^{(r+1)}_{n+1}+\be^{(r+1)}_{n+1}+\al^{(r+2)}_{n+1}\right)q^{(r)}_{n+1,2} \\ & -\left(\be^{(r+1)}_{n+1}+\al^{(r+2)}_{n+1}\right)  \left(\be^{(r)}_{n}+\al^{(r+1)}_{n}\right)q^{(r)}_{n,1}  = 0
\end{split}
\end{equation}
These can be written as the following recursive recurrence relations \eqref{q1 rec rel}--\eqref{q3 rec rel}.

\end{proof}


We conclude this section with a collection of simple or ``atomic'' pure and mixed $n,r$ recurrence relations that will be extensively utilized in our applications work \cite{GW3}, and which were derived in \cite{GW}.
One has the following set of simple pure-$n$ recurrence relations for the sub-leading coefficients in the generic case
\begin{align}
    p^{(r)}_{n+2,1}-p^{(r)}_{n,1} & = -\frac{h^{(r-1)}_{n+1}}{h^{(r)}_{n+1}}-\frac{h^{(r-2)}_{n}}{h^{(r-1)}_{n}}+\frac{h^{(r-2)}_{n}}{h^{(r)}_{n-1}} , \label{p_1 step 2 difference}\\
    q^{(r)}_{n+1,1}-q^{(r)}_{n,1} & = -\frac{h^{(r+2)}_{n}}{h^{(r)}_{n}}+\frac{h^{(r+1)}_{n}}{h^{(r)}_{n-1}}+\frac{h^{(r+2)}_{n}}{h^{(r+1)}_{n-1}} ,
\label{p1q1_pure-n-shift}
\end{align}
and corresponding ones for pure-$r$ recurrence relations
\begin{align}
    p^{(r+2)}_{n,1}-p^{(r)}_{n,1} & = -\frac{h^{(r)}_{n}}{h^{(r+2)}_{n-1}} ,\\
    q^{(r+1)}_{n,1}-q^{(r)}_{n,1} & = \frac{h^{(r+1)}_{n}}{h^{(r)}_{n-1}} .
\label{p1q1_pure-r-shift}
\end{align}
In addition we list some mixed $n,r$ recurrences of low order \cite[Equations (4.24), (4.11), (4.14), and (4.27)]{GW}:
\begin{equation}
    q^{(r)}_{n+1,1}-q^{(r+2)}_{n,1} = -\frac{h^{(r+2)}_{n}}{h^{(r)}_{n}} ,
\label{q1_mixed-shift_single-h}
\end{equation}
\begin{equation}
    p^{(r)}_{n+1,1}-p^{(r-1)}_{n,1} = -\frac{h^{(r-1)}_{n}}{h^{(r)}_{n}} ,
\label{p1_mixed-shift_single-h}
\end{equation}
\begin{equation}
    p^{(r)}_{n+1,1}-p^{(r+1)}_{n,1} = -\frac{h^{(r+1)}_{n}}{h^{(r+2)}_{n}}+\frac{h^{(r)}_{n+1}}{h^{(r+2)}_{n}} ,
\label{p1_mixed-shift_double-h}
\end{equation}
\begin{equation}
    q^{(r)}_{n+1,1}-q^{(r+1)}_{n,1} = -\frac{h^{(r+1)}_{n}}{h^{(r-1)}_{n}}+\frac{h^{(r)}_{n+1}}{h^{(r-1)}_{n}} .
\label{q1_mixed-shift_double-h}
\end{equation}

\section{The Action of the Multiplication Operator in the Polynomial Basis}\label{sec mult z}
Understanding this action has both theoretical and practical significance. 
On the theoretical side one problem yet to be resolved is that of the analogue of the CMV matrix in the slanted Toeplitz systems, although we will not pursue this direction in our current work.
On the practical side of our application \cite{GW3} we will find it necessary to evaluate integrals of the form
\begin{equation}\label{Ik}
	\mathscr{I}^{(k)}_{n,m}(u;r) := \int_{\T} \frac{\dd\ze}{2\pi\ic\ze}w(\ze;u) P_n(\ze;r)Q_m(\ze^{-2};r) \ze^{k-r},
\end{equation}
where $ m,n \in \Z_{\ge 0} $, $k\in \Z$ are independent and arbitrary integers. Leveraging these integrals involves calculating them in two different ways, and both ways involve employing a multiplicative mapping of one of the polynomial factors by some power of $\ze$. 

Lemmas~\ref{prop zP zQ} and~\ref{prop z2P z2Q} below concern multiplication by \(z\) and
multiplication by \(z^2\), respectively.

\begin{lemma}\label{prop zP zQ}
   Let $\pi_n(z)$ and  $\hat{\pi}_n(z)$ denote some polynomials of degree at most $n$ in $z$. We have the following formulas for the action of the multiplication operator by $z$ in the polynomial basis:
\begin{equation}\label{zP}
\begin{split}
zP_n(z;r) & =  P_{n+1}(z;r) + \mathscr{A}^{(r)}_{n,0} P_n(z;r) + \mathscr{A}^{(r)}_{n,1} P_{n-1}(z;r) 	+ \mathscr{A}^{(r)}_{n,2} P_{n-2}(z;r)   + \pi_{n-3}(z)
\end{split}
\end{equation}
where
\begin{align}
\mathscr{A}^{(r)}_{n,0} & = p^{(r)}_{n,1}  - p^{(r)}_{n+1,1}, \label{An0} \\
\mathscr{A}^{(r)}_{n,1} & =  p^{(r)}_{n,1} p^{(r)}_{n+1,1}- \left(p^{(r)}_{n,1}\right)^2       +   p^{(r)}_{n,2}  - p^{(r)}_{n+1,2}, \label{An1}  \\
\mathscr{A}^{(r)}_{n,2} & = \left( p^{(r)}_{n,1}  - p^{(r)}_{n+1,1} \right)\left( p^{(r)}_{n,1} p^{(r)}_{n-1,1}  - p^{(r)}_{n,2} \right)   -  p^{(r)}_{n-1,1}  \left( p^{(r)}_{n,2}  - p^{(r)}_{n+1,2} \right)   + p^{(r)}_{n,3}  - p^{(r)}_{n+1,3}, \label{An2}   
\end{align}

\begin{equation}\label{zQ}
\begin{split}
zQ_n(z;r) & =  Q_{n+1}(z;r) + \mathscr{B}^{(r)}_{n,0} Q_n(z;r) + \mathscr{B}^{(r)}_{n,1} Q_{n-1}(z;r) 	+ \mathscr{B}^{(r)}_{n,2} Q_{n-2}(z;r)   + \hat{\pi}_{n-3}(z)
\end{split}
\end{equation}
where
\begin{align}
\mathscr{B}^{(r)}_{n,0} & = q^{(r)}_{n,1}  - q^{(r)}_{n+1,1}, \label{Bn0} \\
\mathscr{B}^{(r)}_{n,1} & =  q^{(r)}_{n,1} q^{(r)}_{n+1,1}- \left(q^{(r)}_{n,1}\right)^2       +   q^{(r)}_{n,2}  - q^{(r)}_{n+1,2}, \label{Bn1} \\
\mathscr{B}^{(r)}_{n,2} & = \left( q^{(r)}_{n,1}  - q^{(r)}_{n+1,1} \right)\left( q^{(r)}_{n,1} q^{(r)}_{n-1,1}  - q^{(r)}_{n,2} \right)   -  q^{(r)}_{n-1,1}  \left( q^{(r)}_{n,2}  - q^{(r)}_{n+1,2} \right)   + q^{(r)}_{n,3}  - q^{(r)}_{n+1,3}, \label{Bn2}  
\end{align}
\end{lemma}
\begin{proof} We only present the details for the proof of \eqref{zP}--\eqref{An2}, as the proof of \eqref{zQ}--\eqref{Bn2} follows in a similar way.  We first express monomials in the polynomial basis $\{P_n(z;r)\}$.
Below, $\pi_n(z)$ denotes some polynomial of degree at most $n$ in $z$. Recall the notation introduced in \eqref{P}:

\begin{equation}
P_n(z;r) = z^n + p^{(r)}_{n,1} z^{n-1} + p^{(r)}_{n,2} z^{n-2} + p^{(r)}_{n,3} z^{n-3} + \pi_{n-4}(z) 
\end{equation}

\begin{equation}
P_{n-1}(z;r) = z^{n-1} + p^{(r)}_{n-1,1} z^{n-2} + p^{(r)}_{n-1,2} z^{n-3} + \pi_{n-4}(z) 
\end{equation}

\begin{equation}
P_{n-2}(z;r) = z^{n-2} + p^{(r)}_{n-2,1} z^{n-3} + \pi_{n-4}(z) 
\end{equation}

\begin{equation}
P_{n-3}(z;r) = z^{n-3} + \pi_{n-4}(z) 
\end{equation}

Inverting this representation we find

\begin{equation}\label{zn-3}
z^{n-3} = P_{n-3}(z;r)  + \pi_{n-4}(z) 
\end{equation}

\begin{equation}\label{zn-2}
z^{n-2}  = P_{n-2}(z;r) - p^{(r)}_{n-2,1} P_{n-3}(z;r)  + \pi_{n-4}(z) 
\end{equation}

\begin{equation}\label{zn-1}
z^{n-1} = P_{n-1}(z;r) - p^{(r)}_{n-1,1} P_{n-2}(z;r) + \left( p^{(r)}_{n-1,1}p^{(r)}_{n-2,1}  - p^{(r)}_{n-1,2} \right) P_{n-3}(z;r) + \pi_{n-4}(z)
\end{equation}

\begin{equation}
\begin{split}
z^n & = P_n(z;r) -  p^{(r)}_{n,1}  P_{n-1}(z;r) + \left( p^{(r)}_{n,1} p^{(r)}_{n-1,1}  - p^{(r)}_{n,2} \right)  P_{n-2}(z;r) \\ & +\left( p^{(r)}_{n,2} p^{(r)}_{n-2,1} -  p^{(r)}_{n,1} \left( p^{(r)}_{n-1,1}p^{(r)}_{n-2,1}  - p^{(r)}_{n-1,2} \right)    - p^{(r)}_{n,3} \right) P_{n-3}(z;r) + \pi_{n-4}(z)
\end{split}
\end{equation}

Similarly, we have

\begin{equation}
z^{n-3} = Q_{n-3}(z;r)  + \pi_{n-4}(z) 
\end{equation}

\begin{equation}
z^{n-2}  = Q_{n-2}(z;r) - q^{(r)}_{n-2,1} Q_{n-3}(z;r)  + \pi_{n-4}(z) 
\end{equation}

\begin{equation}
z^{n-1} = Q_{n-1}(z;r) - q^{(r)}_{n-1,1} Q_{n-2}(z;r) + \left( q^{(r)}_{n-1,1}q^{(r)}_{n-2,1}  - q^{(r)}_{n-1,2} \right) Q_{n-3}(z;r) + \pi_{n-4}(z)
\end{equation}

\begin{equation}
\begin{split}
z^n & = Q_n(z;r) -  q^{(r)}_{n,1}  Q_{n-1}(z;r) + \left( q^{(r)}_{n,1} q^{(r)}_{n-1,1}  - q^{(r)}_{n,2} \right)  Q_{n-2}(z;r) \\ & +\left( q^{(r)}_{n,2} q^{(r)}_{n-2,1} -  q^{(r)}_{n,1} \left( q^{(r)}_{n-1,1}q^{(r)}_{n-2,1}  - q^{(r)}_{n-1,2} \right)    - q^{(r)}_{n,3} \right) Q_{n-3}(z;r) + \pi_{n-4}(z)
\end{split}
\end{equation}

Now we are ready to read the action of multiplication operator in the polynomial basis.

\begin{equation}
P_{n+1}(z;r) = z^{n+1} + p^{(r)}_{n+1,1} z^{n} + p^{(r)}_{n+1,2} z^{n-1} + p^{(r)}_{n+1,3} z^{n-2} + \pi_{n-3}(z) 
\end{equation}

\begin{equation}
zP_n(z;r) = z^{n+1} + p^{(r)}_{n,1} z^{n} + p^{(r)}_{n,2} z^{n-1} + p^{(r)}_{n,3} z^{n-2} + \pi_{n-3}(z) 
\end{equation}
Thus
\begin{equation}
zP_n(z;r)-P_{n+1}(z;r) = \left( p^{(r)}_{n,1}  - p^{(r)}_{n+1,1} \right) z^n + \left( p^{(r)}_{n,2}  - p^{(r)}_{n+1,2} \right) z^{n-1} +\left( p^{(r)}_{n,3}  - p^{(r)}_{n+1,3} \right) z^{n-2}  + \pi_{n-3}(z), 
\end{equation}
and hence
\begin{equation}
\begin{split}
zP_n(z;r) & =  P_{n+1}(z;r) + \left( p^{(r)}_{n,1}  - p^{(r)}_{n+1,1} \right) P_n(z;r) + \left( p^{(r)}_{n,1} p^{(r)}_{n+1,1}- \left(p^{(r)}_{n,1}\right)^2      +   p^{(r)}_{n,2}  - p^{(r)}_{n+1,2}  \right) P_{n-1}(z;r) 	\\ &  + \left[  \left( p^{(r)}_{n,1}  - p^{(r)}_{n+1,1} \right)\left( p^{(r)}_{n,1} p^{(r)}_{n-1,1}  - p^{(r)}_{n,2} \right)   -  p^{(r)}_{n-1,1}  \left( p^{(r)}_{n,2}  - p^{(r)}_{n+1,2} \right)   +\left( p^{(r)}_{n,3}  - p^{(r)}_{n+1,3} \right) \right] P_{n-2}(z;r)   \\ & + \pi_{n-3}(z). 
\end{split}
\end{equation}
This is the desired identity \eqref{zP}.
\end{proof}

\begin{lemma}\label{prop z2P z2Q}
   Let $\tilde{\pi}_n(z)$ and  $\pi^*_n(z)$ denote some polynomials of degree at most $n$ in $z$. We have the following formulas for the action of multiplication operator by $z^2$ in the polynomial basis:
   \begin{equation}\label{z2P}
\begin{split}
z^2P_n(z;r) & =  P_{n+2}(z;r) + 	\mathscr{C}^{(r)}_{n,0} P_{n+1}(z;r) + \mathscr{C}^{(r)}_{n,1} P_{n}(z;r)  + \mathscr{C}^{(r)}_{n,2} P_{n-1}(z;r)  + \tilde{\pi}_{n-2}(z), 
\end{split}
\end{equation}

where
\begin{align}
\mathscr{C}^{(r)}_{n,0} & = p^{(r)}_{n,1}  - p^{(r)}_{n+2,1}, \\
\mathscr{C}^{(r)}_{n,1} & =  p^{(r)}_{n+1,1} p^{(r)}_{n+2,1}         - p^{(r)}_{n+2,2}     -p^{(r)}_{n,1} p^{(r)}_{n+1,1}       +   p^{(r)}_{n,2},  \\
\mathscr{C}^{(r)}_{n,2} & = p^{(r)}_{n+1,1} \left(p^{(r)}_{n,1}\right)^2-\left(p^{(r)}_{n,2}+p^{(r)}_{n+1,2}+p^{(r)}_{n+1,1} p^{(r)}_{n+2,1}-p^{(r)}_{n+2,2}\right) p^{(r)}_{n,1}+p^{(r)}_{n,3}+p^{(r)}_{n+1,2}
p^{(r)}_{n+2,1}-p^{(r)}_{n+2,3}. \label{Cn2}
\end{align}

\begin{equation}\label{z2Q}
\begin{split}
z^2 Q_n(z;r) & =  Q_{n+2}(z;r) + 	\mathscr{D}^{(r)}_{n,0} Q_{n+1}(z;r) + \mathscr{D}^{(r)}_{n,1} Q_{n}(z;r)  + \mathscr{D}^{(r)}_{n,2} Q_{n-1}(z;r)  + \pi^*_{n-2}(z), 
\end{split}
\end{equation}

where
\begin{align}
\mathscr{D}^{(r)}_{n,0} & = q^{(r)}_{n,1}  - q^{(r)}_{n+2,1}, \\
\mathscr{D}^{(r)}_{n,1} & =  q^{(r)}_{n+1,1} q^{(r)}_{n+2,1}         - q^{(r)}_{n+2,2}     -q^{(r)}_{n,1} q^{(r)}_{n+1,1}       +   q^{(r)}_{n,2},  \\
\mathscr{D}^{(r)}_{n,2} & = q^{(r)}_{n+1,1} \left(q^{(r)}_{n,1}\right)^2-\left(q^{(r)}_{n,2}+q^{(r)}_{n+1,2}+q^{(r)}_{n+1,1} q^{(r)}_{n+2,1}-q^{(r)}_{n+2,2}\right) q^{(r)}_{n,1}+q^{(r)}_{n,3}+q^{(r)}_{n+1,2}
q^{(r)}_{n+2,1}-q^{(r)}_{n+2,3}. \label{Dn2}
\end{align}
   \end{lemma}

\begin{proof} We only present the details for the proof of \eqref{z2P}--\eqref{Cn2}, as the proof of \eqref{z2Q}--\eqref{Dn2} follows in a similar way.
 Note that 
\begin{equation}
\begin{split}
zP_{n+1}(z;r) & =  P_{n+2}(z;r) + \left( p^{(r)}_{n+1,1}  - p^{(r)}_{n+2,1} \right) P_{n+1}(z;r) \\ & + \left( p^{(r)}_{n+1,1} p^{(r)}_{n+2,1}- \left(p^{(r)}_{n+1,1}\right)^2      +   p^{(r)}_{n+1,2}  - p^{(r)}_{n+2,2}  \right) P_{n}(z;r) 	\\ &  + \left[  \left( p^{(r)}_{n+1,1}  - p^{(r)}_{n+2,1} \right)\left( p^{(r)}_{n+1,1} p^{(r)}_{n,1}  - p^{(r)}_{n+1,2} \right)   -  p^{(r)}_{n,1}  \left( p^{(r)}_{n+1,2}  - p^{(r)}_{n+2,2} \right)   + p^{(r)}_{n+1,3}  - p^{(r)}_{n+2,3}  \right] P_{n-1}(z;r)  \\ & + \pi_{n-2}(z), 
\end{split}
\end{equation}

\begin{equation}
\begin{split}
zP_{n-1}(z;r) & =  P_{n}(z;r) + \left( p^{(r)}_{n-1,1}  - p^{(r)}_{n,1} \right) P_{n-1}(z;r) + \left( p^{(r)}_{n-1,1} p^{(r)}_{n,1}- \left(p^{(r)}_{n-1,1}\right)^2       +   p^{(r)}_{n-1,2}  - p^{(r)}_{n,2}  \right) P_{n-2}(z;r) \\ & + \pi_{n-3}(z), 
\end{split}
\end{equation}
and
\begin{equation}
\begin{split}
zP_{n-2}(z;r) & =  P_{n-1}(z;r) + \left( p^{(r)}_{n-2,1}  - p^{(r)}_{n-1,1} \right) P_{n-2}(z;r) + \pi_{n-3}(z). 
\end{split}
\end{equation}
These are useful to express the right-hand side of 
\begin{equation}
\begin{split}
z^2P_n(z;r) & =  zP_{n+1}(z;r) + \left( p^{(r)}_{n,1}  - p^{(r)}_{n+1,1} \right) zP_n(z;r) + \left( p^{(r)}_{n,1} p^{(r)}_{n+1,1}- \left(p^{(r)}_{n,1}\right)^2       +   p^{(r)}_{n,2}  - p^{(r)}_{n+1,2}  \right) zP_{n-1}(z;r) 	\\ &  + \left[  \left( p^{(r)}_{n,1}  - p^{(r)}_{n+1,1} \right)\left( p^{(r)}_{n,1} p^{(r)}_{n-1,1}  - p^{(r)}_{n,2} \right)   -  p^{(r)}_{n-1,1}  \left( p^{(r)}_{n,2}  - p^{(r)}_{n+1,2} \right)   +\left( p^{(r)}_{n,3}  - p^{(r)}_{n+1,3} \right) \right] zP_{n-2}(z;r) \\ & + \pi_{n-3}(z), 
\end{split}
\end{equation}
purely in terms of the polynomials $P_j(z;r), 0\leq j \leq n+2$. Indeed,

\begin{equation}
\begin{split}
z^2P_n(z;r) & =  P_{n+2}(z;r) + \left( p^{(r)}_{n+1,1}  - p^{(r)}_{n+2,1} \right) P_{n+1}(z;r) \\ & + \left( p^{(r)}_{n+1,1} p^{(r)}_{n+2,1}- \left(p^{(r)}_{n+1,1}\right)^2      +   p^{(r)}_{n+1,2}  - p^{(r)}_{n+2,2}  \right) P_{n}(z;r) 	\\ &  + \left[  \left( p^{(r)}_{n+1,1}  - p^{(r)}_{n+2,1} \right)\left( p^{(r)}_{n+1,1} p^{(r)}_{n,1}  - p^{(r)}_{n+1,2} \right)   -  p^{(r)}_{n,1}  \left( p^{(r)}_{n+1,2}  - p^{(r)}_{n+2,2} \right)   + p^{(r)}_{n+1,3}  - p^{(r)}_{n+2,3}  \right] P_{n-1}(z;r)    \\& + \left( p^{(r)}_{n,1}  - p^{(r)}_{n+1,1} \right) \Big[P_{n+1}(z;r) + \left( p^{(r)}_{n,1}  - p^{(r)}_{n+1,1} \right) P_n(z;r) \\ & + \left( p^{(r)}_{n,1} p^{(r)}_{n+1,1}- \left(p^{(r)}_{n,1}\right)^2       +   p^{(r)}_{n,2}  - p^{(r)}_{n+1,2}  \right) P_{n-1}(z;r) \Big] \\ & + \left( p^{(r)}_{n,1} p^{(r)}_{n+1,1}- \left(p^{(r)}_{n,1}\right)^2       +   p^{(r)}_{n,2}  - p^{(r)}_{n+1,2}  \right) \left[ P_{n}(z;r) + \left( p^{(r)}_{n-1,1}  - p^{(r)}_{n,1} \right) P_{n-1}(z;r)\right]	\\ &  + \left[  \left( p^{(r)}_{n,1}  - p^{(r)}_{n+1,1} \right)\left( p^{(r)}_{n,1} p^{(r)}_{n-1,1}  - p^{(r)}_{n,2} \right)   -  p^{(r)}_{n-1,1}  \left( p^{(r)}_{n,2}  - p^{(r)}_{n+1,2} \right)   +\left( p^{(r)}_{n,3}  - p^{(r)}_{n+1,3} \right) \right] P_{n-1}(z;r) \\ &  + \pi_{n-2}(z). 
\end{split}
\end{equation}
Further simplification gives
\begin{equation}
\begin{split}
& z^2P_n(z;r)  =  P_{n+2}(z;r) + \left( p^{(r)}_{n,1}  - p^{(r)}_{n+2,1} \right) P_{n+1}(z;r) + \left(  p^{(r)}_{n+1,1} p^{(r)}_{n+2,1}         - p^{(r)}_{n+2,2}     -p^{(r)}_{n,1} p^{(r)}_{n+1,1}       +   p^{(r)}_{n,2}  \right) P_{n}(z;r)	\\ &  + \left[ p^{(r)}_{n+1,1} \left(p^{(r)}_{n,1}\right)^2-\left(p^{(r)}_{n,2}+p^{(r)}_{n+1,2}+p^{(r)}_{n+1,1} p^{(r)}_{n+2,1}-p^{(r)}_{n+2,2}\right) p^{(r)}_{n,1}+p^{(r)}_{n,3}+p^{(r)}_{n+1,2}
p^{(r)}_{n+2,1}-p^{(r)}_{n+2,3} \right] P_{n-1}(z;r) \\ & + \pi_{n-2}(z), 
\end{split}
\end{equation} which is the desired expression \eqref{z2P}.

\end{proof}

\hspace{1cm}

\section{A Multiple Integral Formula for a Bi-bordered $2j-k$ Determinant}\label{sec bi-bordered}

Multiple integral formulations are fundamental as they constitute the starting point for any applications involving
random matrix theory, matrix integrals or averages over the classical groups.
In addition they will be seen as an efficient computational tool for deriving key identities, 
as will become apparent in the subsequent sections.

We now turn to determinant identities for relevant multiple integral formulae.
Recall that \(P_n(z;r)\) and \(Q_n(z;r)\) are bordered \(2j-k\)
determinants; see \eqref{OP11} and \eqref{OP22}. Their multiple integral
representations are
\begin{align}
	P_n(z;r)
	&=
	\frac{1}{D_n^{(r)}}
	\mathcal{D}_n\!\left[w(\zeta)\zeta^{-r}(z-\zeta)\right],
	\label{P multiple integral recalled}\\
	Q_n(z;r)
	&=
	\frac{1}{D_n^{(r)}}
	\mathcal{D}_n\!\left[w(\zeta)\zeta^{-r}(z-\zeta^{-2})\right],
	\label{Q multiple integral recalled}
\end{align}
as proved in \cite[Theorem~5.3, Eqs.~(5.13)--(5.14)]{GW}. Moving beyond (singly) bordered determinants, bi-bordered determinants are also of interest. For example, in the case of a single row and a single column borders, the determinant is the reproducing kernel for the $P-Q$ system \cite[Theorem 3.4]{GW}. 
The determinant below is another important bi-bordered object: it has two spectral columns,
and its multiple integral representation therefore contains the two factors
\((z_1-\zeta^{-2})(z_2-\zeta^{-2})\).
The result below gives the corresponding multiple integral formula.

\begin{lemma}
	For the bi-bordered $2j-k$-determinant 
	\begin{equation}\label{bi-bordered 2j-k monomials}
	F^{(2)}_n(z_1,z_2):=  \frac{1}{z_1-z_2}  \det \begin{pmatrix}
	w_r & w_{r-1} & \cdots & w_{r-n+1} & 1 & 1 \\[3pt]
	w_{r+2} & w_{r+1}  & \cdots & w_{r-n+3}   & z_2 & z_1 \\[3pt]
	w_{r+4} & w_{r+3} &  \cdots & w_{r-n+5} & z^2_2 & z^2_1 \\[3pt]
	w_{r+6} & w_{r+5} &  \cdots & w_{r-n+7} & z^3_2 & z^3_1 \\[3pt]
	\vdots & \vdots  & \vdots & \vdots & \vdots & \vdots \\[3pt]
	w_{r+2n+2} & w_{r+2n+1}  & \cdots & w_{r+n+3}   & z^{n+1}_2 & z_1^{n+1}
	\end{pmatrix}
	\end{equation}
	it holds that 
	\begin{equation}
	F^{(2)}_n(z_1,z_2) = \mathcal{D}_n[w(\ze)\ze^{-r}(z_1-\ze^{-2})(z_2-\ze^{-2})],
	\end{equation}
	where for an integrable function $f$ on the unit circle we define \ \cite{GW} 
	\begin{equation}\label{DDDDDD}
	\mathcal{D}_n[f(\ze)]:= \frac{1}{n!} \int_{\T} \frac{\dd \ze_1}{2 \pi \ic \ze_1}\int_{\T} \frac{\dd \ze_2}{2 \pi \ic \ze_2} \cdots \int_{\T} \frac{\dd \ze_n}{2 \pi \ic \ze_n} \prod_{j=1}^{n}f(\ze_j) \prod_{1\leq j<k\leq n} (\ze_k-\ze_j)(\ze^{-2}_k-\ze^{-2}_j). 
	\end{equation} 	 
\end{lemma}

\begin{proof}

	Let us recall from Theorem 5.3 of \cite{GW} the multiple integral formula for $Q_n$ polynomials:
	
	\begin{equation}\label{MultIntPolysQ}
	Q_n(z;r) = \frac{1}{D^{(r)}_n} \mathcal{D}_n[w(\ze)\ze^{-r}(z-\ze^{-2})], 
	\end{equation}
	
	Let $\overset{*}{w}(z):=w(z)(z_1-z^{-2})$. We have
	\begin{equation}
	\mathcal{D}_n[w(\ze)\ze^{-r}(z_1-\ze^{-2})(z_2-\ze^{-2})]=	\mathcal{D}_n[w_*(\ze)\ze^{-r}(z_2-\ze^{-2})] =  \det \begin{pmatrix}
	\overset{*}{w}_{r} & \overset{*}{w}_{r-1}  & \cdots & \overset{*}{w}_{r-n+1} & 1 \\
	\overset{*}{w}_{r+2} & \overset{*}{w}_{r+1} &  \cdots & \overset{*}{w}_{r-n+3} & z_2 \\
	\vdots & \vdots  & \vdots & \vdots & \vdots \\
	\overset{*}{w}_{r+2n} & \overset{*}{w}_{r+2n-1} & \cdots & \overset{*}{w}_{r+n+1} & z_2^n
	\end{pmatrix} \oeq
	\end{equation}
	Since $\overset{*}{w}_r=z_1w_r-w_{r+2}$, we have
	\begin{equation}
	\begin{split}
	& \oeq \det \begin{pmatrix}
	z_1w_r-w_{r+2} & z_1w_{r-1}-w_{r+1}  & \cdots & z_1w_{r-n+1}-w_{r-n+3}   & 1 \\
	z_1w_{r+2}-w_{r+4} & z_1w_{r+1}-w_{r+3} &  \cdots & z_1w_{r-n+3}-w_{r-n+5} & z_2 \\
	\vdots & \vdots  & \vdots & \vdots & \vdots \\
	z_1w_{r+2n}-w_{r+2n+2} & z_1w_{r+2n-1}-w_{r+2n+1}  & \cdots & z_1w_{r+n+1}-w_{r+n+3}   & z_2^n
	\end{pmatrix} \\ & = (-1)^{n+1} \det \begin{pmatrix}
	w_r & w_{r-1} & \cdots & w_{r-n+1} & 0 & 1 \\
	z_1w_r-w_{r+2} & z_1w_{r-1}-w_{r+1}  & \cdots & z_1w_{r-n+1}-w_{r-n+3}   & 1 & 0 \\
	z_1w_{r+2}-w_{r+4} & z_1w_{r+1}-w_{r+3} &  \cdots & z_1w_{r-n+3}-w_{r-n+5} & z_2 & 0 \\
	\vdots & \vdots  & \vdots & \vdots & \vdots & \vdots \\
	z_1w_{r+2n}-w_{r+2n+2} & z_1w_{r+2n-1}-w_{r+2n+1}  & \cdots & z_1w_{r+n+1}-w_{r+n+3}   & z_2^n & 0
	\end{pmatrix} \\ & = \det \begin{pmatrix}
	w_r & w_{r-1} & \cdots & w_{r-n+1} & 0 & 1 \\
	-z_1w_r+w_{r+2} & -z_1w_{r-1}+w_{r+1}  & \cdots & -z_1w_{r-n+1}+w_{r-n+3}   & -1 & 0 \\
	-z_1w_{r+2}+w_{r+4} & -z_1w_{r+1}+w_{r+3} &  \cdots & -z_1w_{r-n+3}+w_{r-n+5} & -z_2 & 0 \\
	-z_1w_{r+4}+w_{r+6} & -z_1w_{r+3}+w_{r+5} &  \cdots & -z_1w_{r-n+5}+w_{r-n+7} & -z^2_2 & 0 \\
	\vdots & \vdots  & \vdots & \vdots & \vdots & \vdots \\
	-z_1w_{r+2n}+w_{r+2n+2} & -z_1w_{r+2n-1}+w_{r+2n+1}  & \cdots & -z_1w_{r+n+1}+w_{r+n+3}   & -z_2^n & 0
	\end{pmatrix} \oeq
	\end{split}
	\end{equation}
	Let $\widehat{R}_j$ denote the row $R_j$ after a row operation. We do the following row operations in order: I) $\widehat{R}_2 = z_1 R_1+R_2$, II) $\widehat{R}_3 \mapsto z_1 \widehat{R}_2+R_3$, and so on. In this calculation we use \[ -\left(\frac{z^j_1-z^j_2}{z_1-z_2}\right)z_1-z^j_2 = -\left(\frac{z^{j+1}_1-z^{j+1}_2}{z_1-z_2}\right)  \] We obtain
	\begin{equation}
	\oeq \det \begin{pmatrix}
	w_r & w_{r-1} & \cdots & w_{r-n+1} & 0 & 1 \\
	w_{r+2} & w_{r+1}  & \cdots & w_{r-n+3}   & -1 & z_1 \\
	w_{r+4} & w_{r+3} &  \cdots & w_{r-n+5} & -\left(\frac{z^2_1-z^2_2}{z_1-z_2}\right) & z^2_1 \\
	w_{r+6} & w_{r+5} &  \cdots & w_{r-n+7} & -\left(\frac{z^3_1-z^3_2}{z_1-z_2}\right) & z^3_1 \\
	\vdots & \vdots  & \vdots & \vdots & \vdots & \vdots \\
	w_{r+2n+2} & w_{r+2n+1}  & \cdots & w_{r+n+3}   & -\left(\frac{z^{n+1}_1-z^{n+1}_2}{z_1-z_2}\right) & z_1^{n+1}
	\end{pmatrix} 
	\end{equation}
	
	\begin{equation*}
	= \frac{1}{z_1-z_2}  \det \begin{pmatrix}
	w_r & w_{r-1} & \cdots & w_{r-n+1} & 0 & 1 \\[3pt]
	w_{r+2} & w_{r+1}  & \cdots & w_{r-n+3}   & z_2-z_1 & z_1 \\[3pt]
	w_{r+4} & w_{r+3} &  \cdots & w_{r-n+5} & z^2_2-z^2_1 & z^2_1 \\[3pt]
	w_{r+6} & w_{r+5} &  \cdots & w_{r-n+7} & z^3_2-z^3_1 & z^3_1 \\[3pt]
	\vdots & \vdots  & \vdots & \vdots & \vdots & \vdots \\[3pt]
	w_{r+2n+2} & w_{r+2n+1}  & \cdots & w_{r+n+3}   & z^{n+1}_2-z^{n+1}_1 & z_1^{n+1}
	\end{pmatrix}
	\end{equation*}
	
	\begin{equation*}
	= \frac{1}{z_1-z_2}  \det \begin{pmatrix}
	w_r & w_{r-1} & \cdots & w_{r-n+1} & 1 & 1 \\[3pt]
	w_{r+2} & w_{r+1}  & \cdots & w_{r-n+3}   & z_2 & z_1 \\[3pt]
	w_{r+4} & w_{r+3} &  \cdots & w_{r-n+5} & z^2_2 & z^2_1 \\[3pt]
	w_{r+6} & w_{r+5} &  \cdots & w_{r-n+7} & z^3_2 & z^3_1 \\[3pt]
	\vdots & \vdots  & \vdots & \vdots & \vdots & \vdots \\[3pt]
	w_{r+2n+2} & w_{r+2n+1}  & \cdots & w_{r+n+3}   & z^{n+1}_2 & z_1^{n+1}
	\end{pmatrix}
	\end{equation*}
\end{proof}

\subsection{Bi-bordered Determinant Evaluation}

We now evaluate the bi-bordered determinant discussed above in terms of the $Q$-polynomials.
This expression will be useful when we come to recast the Christoffel-Darboux summation formula in Proposition \ref{prop CD mixed}.
In addition these results are of interest in their own right.

\begin{lemma}
	For the bi-bordered $2j-k$-determinant \eqref{bi-bordered 2j-k monomials} we have the following representation in terms of the $Q$-polynomials:
	\begin{equation}
	\mathcal{D}_n[w(\ze)\ze^{-r}(z_1-\ze^{-2})(z_2-\ze^{-2})]	= D^{(r)}_n \frac{  Q_{n}(z_1;r) Q_{n+1}(z_2;r) - Q_{n}(z_2;r) Q_{n+1}(z_1;r)  }{z_2-z_1}. 
    \label{Q_bi-bordered}
	\end{equation}
\end{lemma}
\begin{proof}
	Let us recall the LDU decomposition for the $2j-k$ determinants from \cite{GW}. Let \begin{equation}\label{polys}
	P_n(z;r) = \sum^{n}_{\ell=0} \mathcal{p}^{(r)}_{n,\ell} z^{\ell}, \qquad Q_n(z;r) =\sum^{n}_{\ell=0} \mathcal{q}^{(r)}_{n,\ell} z^{\ell},
	\end{equation}
	with $\mathcal{p}^{(r)}_{n,n}=\mathcal{q}^{(r)}_{n,n}=1$. Let us also denote 
	\begin{equation}\label{Vectors}
	\boldsymbol{Z}_n(z) := \begin{pmatrix}
	1 \\ z \\ \vdots \\ z^{n}
	\end{pmatrix}  \qandq \boldsymbol{F}_n(z;r) := \begin{pmatrix}
	F_0(z;r) \\ F_1(z;r) \\ \vdots \\ F_n(z;r)
	\end{pmatrix}, \qquad \boldsymbol{F} \in \{\boldsymbol{P},\boldsymbol{Q}\}.
	\end{equation}
	We thus have
	\begin{equation}
	\boldsymbol{P}_n(z;r) = \boldsymbol{\mathcal{P}}_{n}^{(r)}  \boldsymbol{Z}_n(z), \qquad	\boldsymbol{Q}_n(z;r) = \boldsymbol{\mathcal{Q}}_{n}^{(r)}  \boldsymbol{Z}_n(z), 
	\end{equation}
	where $\boldsymbol{\mathcal{P}}_{n}^{(r)}$ and $\boldsymbol{\mathcal{Q}}_{n}^{(r)},$ are the following $(n+1)\times(n+1)$ lower triangular matrices
	\begin{equation}\label{A B}
	\boldsymbol{\mathcal{P}}_{n}^{(r)} := \begin{pmatrix}
	1 & 0 & \cdots & 0 \\
	\mathcal{p}^{(r)}_{1,0} & 1 & \cdots & 0 \\
	\vdots & \vdots & \ddots & \vdots \\
	\mathcal{p}^{(r)}_{n,0} & \mathcal{p}^{(r)}_{n,1} & \cdots & 1
	\end{pmatrix}, \qquad 	\boldsymbol{\mathcal{Q}}_{n}^{(r)} := \begin{pmatrix}
	1 & 0 & \cdots & 0 \\
	\mathcal{q}^{(r)}_{1,0} & 1 & \cdots & 0 \\
	\vdots & \vdots & \ddots & \vdots \\
	\mathcal{q}^{(r)}_{n,0} & \mathcal{q}^{(r)}_{n,1} & \cdots & 1
	\end{pmatrix}, 
	\end{equation}
	whose inverses are also lower diagonal with $1$'s on the main diagonal. Also let us denote the diagonal matrix of norms of polynomials by $\boldsymbol{h}^{(r)}_{n}$:
	\begin{equation}\label{h&H diag}
	\boldsymbol{h}^{(r)}_{n} := \begin{pmatrix}
	h^{(r)}_0 &  \cdots & 0 \\
	\vdots  & \ddots & \vdots \\
	0  & \cdots & h^{(r)}_n
	\end{pmatrix}.
	\end{equation}
	
	Let us now recall the LDU decomposition of the moment matrix for the $2j-k$ systems \cite[Theorem 2.1]{GW}:

		\begin{equation}\label{LDU D}
		\boldsymbol{D}^{(r)}_{n+1} = \left[\boldsymbol{\mathcal{Q}}^{(r)}_{n}\right]^{-1} 	\boldsymbol{h}^{(r)}_{n}  \left[\left(\boldsymbol{\mathcal{P}}^{(r)}_{n}\right)^T\right]^{-1}.
		\end{equation}

	Let us also define 
	\begin{equation}
	\langle f(z),g(z) \rangle_r := \int_{\T} f(\ze) g(\ze^{-2}) \ze^{-r} w(\ze) \frac{\dd \ze}{2\pi \ic \ze},
	\end{equation}
	so
	\begin{equation}
	w_{r+2j-k} = \langle z^{k},z^{j} \rangle_r.
	\end{equation}
	
	Now Let us consider the matrices
	
	\begin{equation}
	M_{n} :=	\begin{pmatrix}
	w_r & w_{r-1} & \cdots & w_{r-n+1} & 1 & 1 \\[3pt]
	w_{r+2} & w_{r+1}  & \cdots & w_{r-n+3}   & z_1 & z_2 \\[3pt]
	w_{r+4} & w_{r+3} &  \cdots & w_{r-n+5} & z^2_1 & z^2_2 \\[3pt]
	w_{r+6} & w_{r+5} &  \cdots & w_{r-n+7} & z^3_1 & z^3_2 \\[3pt]
	\vdots & \vdots  & \vdots & \vdots & \vdots & \vdots \\[3pt]
	w_{r+2n+2} & w_{r+2n+1}  & \cdots & w_{r+n+3}   & z^{n+1}_1 & z_2^{n+1}
	\end{pmatrix}, \qandq W_{n} := \begin{pmatrix}
	w_r & w_{r-1} & \cdots & w_{r-n+1}  \\
	w_{r+2} & w_{r+1}  & \cdots & w_{r-n+3}    \\
	w_{r+4} & w_{r+3} &  \cdots & w_{r-n+5}  \\
	w_{r+6} & w_{r+5} &  \cdots & w_{r-n+7}  \\
	\vdots & \vdots  & \vdots & \vdots  \\
	w_{r+2n+2} & w_{r+2n+1}  & \cdots & w_{r+n+3} 
	\end{pmatrix}
	\end{equation}
	Notice that $M_{n}$ is an $(n+2)\times(n+2)$ matrix while $W_n$ is strictly rectangular of size $(n+2)\times n$. Notice that
	\begin{equation}
	\boldsymbol{\mathcal{Q}}_{n+1}^{(r)} M_{n} = \boldsymbol{\mathcal{Q}}_{n+1}^{(r)}\left(  W_n , \boldsymbol{Z}_{n+1}(z_1), \boldsymbol{Z}_{n+1}(z_2)  \right)=  \left( \boldsymbol{\mathcal{Q}}_{n+1}^{(r)} W_n , \boldsymbol{Q}_{n+1}(z_1;r), \boldsymbol{Q}_{n+1}(z_2;r)  \right)
	\end{equation}
	Moreover,
	\begin{equation}
    \begin{split}
	\left( \boldsymbol{\mathcal{Q}}_{n+1}^{(r)} W_n\right)_{j,k} & = \left(  \boldsymbol{\mathcal{Q}}_{n+1}^{(r)} \left( w_{r+2j'-k'} \right)_{\substack{0\leq j' \leq n+1 \\ 0 \leq k' \leq n-1}} \right)_{j,k} = \sum^{n+1}_{\ell=0} \left( \boldsymbol{\mathcal{Q}}_{n+1}^{(r)}  \right)_{j,\ell} w_{r+2\ell-k} = \sum^{n+1}_{\ell=0} \mathcal{q}^{(r)}_{j,\ell} \langle z^{k},z^{\ell} \rangle_r \\ & = \langle z^{k},\sum^{n+1}_{\ell=0} \mathcal{q}^{(r)}_{j,\ell} z^{\ell} \rangle_r =  \langle z^{k}, Q_j(z;r) \rangle_r,
    \end{split}
	\end{equation}
	and therefore
	\begin{equation}
	\left( \boldsymbol{\mathcal{Q}}_{n+1}^{(r)} W_n\right)_{j,k} = \begin{cases}
	0 & k < j \\
	h^{(r)}_j & k=j \\
	* & k>j
	\end{cases}
	\end{equation}
	We thus have
	\begin{equation}
	\boldsymbol{\mathcal{Q}}_{n+1}^{(r)} M_n :=	\begin{pmatrix}
	h^{(r)}_0 & * & \cdots & * & 1 & 1 \\[3pt]
	0 & h^{(r)}_1  & \cdots & *   & Q_1(z_1;r) & Q_2(z_1;r) \\[3pt]
	\vdots & \vdots  & \ddots & \vdots & \vdots & \vdots \\[3pt]
	0 & 0  & \cdots & h^{(1)}_{n-1}   & Q_{n-1}(z_1;r) & Q_{n-1}(z_2;r) \\[3pt]
	0 & 0  & \cdots & 0   & Q_{n}(z_1;r) & Q_{n}(z_2;r) \\[3pt]
	0 & 0  & \cdots & 0   & Q_{n+1}(z_1;r) & Q_{n+1}(z_2;r)
	\end{pmatrix}.
	\end{equation}
	Therefore
	\begin{equation}
	\begin{split}
	\det M_n & \equiv	\det \left( \boldsymbol{\mathcal{Q}}_{n+1}^{(r)} M_n \right) = \left[ Q_{n}(z_1;r) Q_{n+1}(z_2;r) - Q_{n}(z_2;r) Q_{n+1}(z_1;r)  \right] \prod_{j=0}^{n-1} h^{(r)}_j \\ & = D^{(r)}_n \left[ Q_{n}(z_1;r) Q_{n+1}(z_2;r) - Q_{n}(z_2;r) Q_{n+1}(z_1;r)  \right]
	\end{split}
	\end{equation}	
	Therefore
	\begin{equation}
	\mathcal{D}_n[w(\ze)\ze^{-r}(z_1-\ze^{-2})(z_2-\ze^{-2})]	= D^{(r)}_n \frac{  Q_{n}(z_1;r) Q_{n+1}(z_2;r) - Q_{n}(z_2;r) Q_{n+1}(z_1;r)  }{z_2-z_1} .
	\end{equation}
\end{proof}

A consequence of the preceding theorem will be required for subsequent use in Lemma \ref{PP_shift_by_2}.
We require an identity which will be furnished by extracting the first two terms of the relation
\eqref{Q_bi-bordered} in a certain expansion of the free variables $z_1, z_2$.
\begin{lemma}
We have the following multiple integral evaluation
\begin{multline}
    \frac{1}{n!} \int_{\T} \frac{\dd \ze_1}{2 \pi \ic \ze_1} \cdots \int_{\T} \frac{\dd \ze_n}{2 \pi \ic \ze_n} 
            \prod_{j=1}^{n}w(\ze_j)\ze^{-r}_{j}
            \left( \sum^{n}_{l=1}\ze^{-2}_{l} \right) \prod^{n}_{m=1}\left( z-\ze^{-2}_{m} \right)
            \prod_{1\leq j<k\leq n} (\ze_k-\ze_j)(\ze^{-2}_k-\ze^{-2}_j)
\\
        = -D^{(r)}_{n} \left[ (z+q^{(r)}_{n+1,1})Q_{n}(z;r) - Q_{n+1}(z;r) \right] .
\label{extended_multiple_integral}
\end{multline}
\end{lemma}
\begin{proof}
We take both sides of \eqref{Q_bi-bordered} and perform an expansion of $ z_2\to \infty $ whilst keeping $ z_1 $ fixed
and only retaining the first two orders. 
Specifically we only require the coefficient of the order $ z_2^{n-1} $ term as the coefficient of the $ z_2^n $ term is trivially 
$ D^{(r)}_{n} Q_{n}(z_1;r) $.
\end{proof}

\begin{remark}\label{rem:m-bordered-polynomial-determinants}
In the same spirit as Propositions~1.9 and~1.10 of \cite{GL24}, the
bi-bordered identities stated above admit natural multi-bordered analogues.
More precisely, after division by the appropriate Vandermonde factor in
\(z_1,\ldots,z_m\), the \(m\)-column monomial-bordered determinant whose bulk
has the \(2j-k\) structure can be expressed as an \(m\times m\) determinant
built from the corresponding \(Q\)-polynomials. Similarly, the \(m\)-row
monomial-bordered determinant with \(2j-k\) bulk can be expressed as an
\(m\times m\) determinant built from the corresponding \(P\)-polynomials.

The same statement holds for the complementary \(j-2k\) structure. Namely,
the \(m\)-column monomial-bordered determinant with \(j-2k\) bulk can be
written as an \(m\times m\) determinant built from the corresponding
\(S\)-polynomials, while the \(m\)-row monomial-bordered determinant with
\(j-2k\) bulk can be written as an \(m\times m\) determinant built from the
corresponding \(R\)-polynomials. These identities follow by the same
Dodgson-condensation argument used in the proof of the multi-bordered
Toeplitz and Hankel identities in \cite[Propositions 1.9 and 1.10]{GL24}; see also the other derivation in \cite[Proposition A.1.]{GKM26}.
\end{remark}

\section{A Shifted Vandermonde Determinant and the Proof of Proposition \ref{prop CD mixed}}\label{sec shifted vandermonde}

In this section
we prove an elementary identity, which is important for the derivation of the mixed $P-Q$
Christoffel--Darboux formula in Proposition \ref{prop CD mixed}. The usual Vandermonde determinant appears
when one evaluates the determinant of consecutive monomials, or of a monic
polynomial basis with consecutive degrees. In the formula below, one
degree is skipped and replaced by the next one. The price of this shift is
a simple multiplicative factor involving the first nontrivial leading coefficient of the
top polynomial and the sum of the variables.

\begin{lemma}
	Let $P_n(z) = z^n + \sum_{j=1}^{n} p_{n,j}z^{n-j}.$ Then \begin{equation}\label{det identity}
	\underset{1 \leq j \leq n}{\det}	 \left[ P_0(\ze_j),P_1(\ze_j),\cdots, P_{n-2}(\ze_j),P_{n}(\ze_j) \right] = \left( p_{n,1} + \sum_{j=1}^{n} \ze_j \right)\underset{1 \leq j<k \leq n}{\prod} (\ze_k-\ze_j)
	\end{equation}
\end{lemma}
\begin{proof}
	Let us start from the left-hand side of \eqref{det identity}:
	\begin{equation}
	\begin{split}
	&	\underset{1 \leq j \leq n}{\det}	 \left[ P_0(\ze_j),P_1(\ze_j),\cdots, P_{n-2}(\ze_j),P_{n}(\ze_j) \right]  = \det \begin{pmatrix}
	P_{0}(\ze_1) & P_{1}(\ze_1)  & \cdots & P_{n-2}(\ze_1) & P_{n}(\ze_1) \\
	P_{0}(\ze_2) & P_{1}(\ze_2)  & \cdots & P_{n-2}(\ze_2) & P_{n}(\ze_2) \\
	\vdots & \vdots  & \vdots & \vdots & \vdots \\
	P_{0}(\ze_n) & P_{1}(\ze_n)  & \cdots & P_{n-2}(\ze_n) & P_{n}(\ze_n) \\
	\end{pmatrix} \\ & = \det \begin{pmatrix}
	1 & \ze_1+p_{1,1} &  \cdots & \ze^{n-2}_{1}+\cdots +p_{n-2,n-3}\ze_{1}+p_{n-2,n-2} & \ze^{n}_{1}+ p_{n,1}\ze^{n-1}_{1} +\cdots +p_{n,n-1}\ze_{1}+p_{n,n} \\
	1 & \ze_2+p_{1,1}  & \cdots & \ze^{n-2}_{2}+\cdots +p_{n-2,n-3}\ze_{2}+p_{n-2,n-2} & \ze^{n}_{2}+p_{n,1}\ze^{n-1}_{2}+\cdots +p_{n,n-1}\ze_{2}+p_{n,n} \\
	\vdots & \vdots   & \cdots & \vdots & \vdots \\
	1 & \ze_n+p_{1,1}  & \cdots & \ze^{n-2}_{n}+\cdots +p_{n-2,n-3}\ze_{n}+p_{n-2,n-2} & \ze^{n}_{n}+ p_{n,1}\ze^{n-1}_{n}+\cdots +p_{n,n-1}\ze_{n}+p_{n,n} \\
	\end{pmatrix} \oeq
	\end{split}
	\end{equation}	
	Using the first column we eliminate all the constant terms in all columns, then we use the second column to eliminate all the linear terms, and so on. Indeed, 
	\begin{equation}\label{73}
	\begin{split}
	& \oeq \det \begin{pmatrix}
	1 & \ze_1 & \ze^2_1+p_{2,1}\ze_1 & \cdots & \ze^{n-2}_{1}+\cdots +p_{n-2,n-3}\ze_{1} & \ze^{n}_{1}+p_{n,1}\ze^{n-1}_{1}+\cdots +p_{n,n-1}\ze_{1} \\
	1 & \ze_2 & \ze^2_2+p_{2,1}\ze_2 & \cdots & \ze^{n-2}_{2}+\cdots +p_{n-2,n-3}\ze_{2} & \ze^{n}_{2}+p_{n,1}\ze^{n-1}_{2}+\cdots +p_{n,n-1}\ze_{2} \\
	\vdots & \vdots  & \vdots & \vdots & \vdots \\
	1 & \ze_n & \ze^2_n+p_{2,1}\ze_n & \cdots & \ze^{n-2}_{n}+\cdots +p_{n-2,n-3}\ze_{n} & \ze^{n}_{n}+p_{n,1}\ze^{n-1}_{n}+\cdots +p_{n,n-1}\ze_{n} \\
	\end{pmatrix} \\ & = \det \begin{pmatrix}
	1 & \ze_1 & \ze^2_1 & \cdots & \ze^{n-2}_{1}+\cdots +p_{n-2,n-4}\ze^2_{1} & \ze^{n}_{1}+p_{n,1}\ze^{n-1}_{1}+\cdots +p_{n,n-2}\ze^2_{1} \\
	1 & \ze_2 & \ze^2_2 & \cdots & \ze^{n-2}_{2}+\cdots +p_{n-2,n-4}\ze^2_{2} & \ze^{n}_{2}+p_{n,1}\ze^{n-1}_{2}+\cdots +p_{n,n-2}\ze^2_{2} \\
	\vdots & \vdots  & \vdots & \vdots & \vdots \\
	1 & \ze_n & \ze^2_n & \cdots & \ze^{n-2}_{n}+\cdots +p_{n-2,n-4}\ze^2_{n} & \ze^{n}_{n}+p_{n,1}\ze^{n-1}_{n}+\cdots +p_{n,n-2}\ze^2_{n} \\
	\end{pmatrix} \\ & \vdots \\ & = \det \begin{pmatrix}
	1 & \ze_1 & \ze^2_1 & \cdots & \ze^{n-2}_{1} & \ze^{n}_{1}+p_{n,1}\ze^{n-1}_{1} \\
	1 & \ze_2 & \ze^2_2 & \cdots & \ze^{n-2}_{2} & \ze^{n}_{2}+p_{n,1}\ze^{n-1}_{2} \\
	\vdots & \vdots  & \vdots & \vdots & \vdots \\
	1 & \ze_n & \ze^2_n & \cdots & \ze^{n-2}_{n} & \ze^{n}_{n}+p_{n,1}\ze^{n-1}_{n} \\
	\end{pmatrix} \\ & = \det \begin{pmatrix}
	1 & \ze_1 & \ze^2_1 & \cdots & \ze^{n-2}_{1} & \ze^{n}_{1} \\
	1 & \ze_2 & \ze^2_2 & \cdots & \ze^{n-2}_{2} & \ze^{n}_{2} \\
	\vdots & \vdots  & \vdots & \vdots & \vdots \\
	1 & \ze_n & \ze^2_n & \cdots & \ze^{n-2}_{n} & \ze^{n}_{n} \\
	\end{pmatrix} + p_{n,1} \det \begin{pmatrix}
	1 & \ze_1 & \ze^2_1 & \cdots & \ze^{n-2}_{1} & \ze^{n-1}_{1} \\
	1 & \ze_2 & \ze^2_2 & \cdots & \ze^{n-2}_{2} & \ze^{n-1}_{2} \\
	\vdots & \vdots  & \vdots & \vdots & \vdots \\
	1 & \ze_n & \ze^2_n & \cdots & \ze^{n-2}_{n} & \ze^{n-1}_{n} \\
	\end{pmatrix} \\ & =  \left( p_{n,1} + \sum_{j=1}^{n} \ze_j \right)\underset{1 \leq j<k \leq n}{\prod} (\ze_k-\ze_j).
	\end{split}
	\end{equation}
\end{proof}

\subsection{The Christoffel-Darboux Formula Revisited}

Let us recall from \cite{GW} the Christoffel-Darboux formula for the $2j-k$ reproducing kernel.

	\begin{equation}\label{CDDDDD}
	K_n(z^2_2,z_1;r)=\frac{1}{2}\frac{D_n^{(r+2)}}{D_{n+1}^{(r)}} \frac{z^{2n+1}_2}{z^{2}_1-z^{-2}_2} \det \begin{pmatrix}
	P_n(-z^{-1}_2;r+2) & P_{n+1}(-z^{-1}_2;r+2) & P_{n+2}(-z^{-1}_2;r+2) \\[2pt]
	P_n(z^{-1}_2;r+2) & P_{n+1}(z^{-1}_2;r+2) & P_{n+2}(z^{-1}_2;r+2) \\[2pt]
	P_n(z_1;r+2) & P_{n+1}(z_1;r+2) & P_{n+2}(z_1;r+2) \\				
	\end{pmatrix},
	\end{equation}
	where $K_{n}$ denotes the reproducing kernel:
	\begin{equation}
	K_{n}(z,\mathcal{z};r)  := \sum_{j=0}^{n} \frac{1}{h^{(r)}_{j}}Q_{j}(z;r)P_{j}(\mathcal{z};r). \label{RepKer33333}
	\end{equation}

Let us also recall the following identity from \cite[Theorem~5.9]{GW} which helps in simplifying this determinant:

\begin{equation}\label{222}
Q_n(z^{2};r)=\frac{D^{(r+2)}_n}{D^{(r)}_n}\frac{z^{2n+1}}{2}\left[ P_{n+1}(z^{-1};r+2)P_n(-z^{-1};r+2) - P_{n+1}(-z^{-1};r+2)P_n(z^{-1};r+2) \right].
\end{equation}
Notice that

\begin{equation}
\begin{split}
& \det \begin{pmatrix}
P_n(-z^{-1}_2;r+2) & P_{n+1}(-z^{-1}_2;r+2) & P_{n+2}(-z^{-1}_2;r+2) \\[2pt]
P_n(z^{-1}_2;r+2) & P_{n+1}(z^{-1}_2;r+2) & P_{n+2}(z^{-1}_2;r+2) \\[2pt]
P_n(z_1;r+2) & P_{n+1}(z_1;r+2) & P_{n+2}(z_1;r+2) \\				
\end{pmatrix} = \\ & P_{n+2}(z_1;r+2)\left[P_n(-z^{-1}_2;r+2) P_{n+1}(z^{-1}_2;r+2) -  P_{n+1}(-z^{-1}_2;r+2)P_n(z^{-1}_2;r+2) \right] \\& + 	P_n(z_1;r+2) \left[ P_{n+1}(-z^{-1}_2;r+2)P_{n+2}(z^{-1}_2;r+2) - P_{n+2}(-z^{-1}_2;r+2) P_{n+1}(z^{-1}_2;r+2)   \right] \\ & - P_{n+1}(z_1;r+2) \left[ 	P_n(-z^{-1}_2;r+2) P_{n+2}(z^{-1}_2;r+2) - P_{n+2}(-z^{-1}_2;r+2) 	P_n(z^{-1}_2;r+2) \right].
\end{split}	
\end{equation}
We observe that the first two terms on the right-hand side of this equation can be further simplified in terms of $Q$-polynomials via equation \eqref{222}, while for the last term we would need an identity analogous to \eqref{222}. This is what the next lemma gives us.

\begin{lemma}\label{PP_shift_by_2}
We have
\begin{equation}\label{a 2by2 determinant of P's with a gap}
	\begin{split}
	& P_{n+1}(z^{-1};r+2)P_{n-1}(-z^{-1};r+2) - P_{n+1}(-z^{-1};r+2)P_{n-1}(z^{-1};r+2) \\ & = \frac{h^{(r+2)}_{n-1}D^{(r)}_n}{h^{(r+2)}_{n}D^{(r+2)}_n} 2z^{-2n-1} \left[    \left(z^2+q^{(r)}_{n+1,1}-q^{(r+2)}_{n,1}\right)Q_n(z^2;r) - Q_{n+1}(z^2;r)  \right]
	\end{split}
\end{equation}
\end{lemma}
\begin{proof}
We start with
\begin{equation}
	\begin{split}
	& \frac{h^{(r+2)}_n}{h^{(r+2)}_{n-1}} D^{(r+2)}_n \left[  P_{n+1}(-z^{-1};r+2)P_{n-1}(z^{-1};r+2) - P_{n+1}(z^{-1};r+2)P_{n-1}(-z^{-1};r+2) \right] \\ & =  - \det \begin{pmatrix}
	h^{(r+2)}_0  & \cdots & 0 & 0 & 0 & 0 \\
	\vdots  & \ddots & \vdots & \vdots & \vdots & \vdots \\
	0  & \cdots & h^{(r+2)}_{n-2} & 0 & 0 & 0 \\
	0  & \cdots & 0 & 0 & h^{(r+2)}_{n}  & 0 \\
	P_0(z^{-1};r+2)  & \cdots & P_{n-2}(z^{-1};r+2) & P_{n-1}(z^{-1};r+2) & P_n(z^{-1};r+2) & P_{n+1}(z^{-1};r+2) \\
	P_0(-z^{-1};r+2)  & \cdots & P_{n-2}(-z^{-1};r+2) & P_{n-1}(-z^{-1};r+2) & P_n(-z^{-1};r+2) & P_{n+1}(-z^{-1};r+2) \\	
	\end{pmatrix} \\ & = - \underset{\substack{1 \leq j \leq n-1 \\ 1 \leq k\leq n+2}}{\det} \begin{pmatrix}
	\di \int_{\T} \frac{\dd \ze}{2 \pi \ic \ze} w(\ze)\ze^{-r-2} P_{k-1}(\ze;r+2)Q_{j-1}(\ze^{-2};r+2) \\
	\di \int_{\T} \frac{\dd \ze}{2 \pi \ic \ze} w(\ze)\ze^{-r-2} P_{k-1}(\ze;r+2)Q_{n}(\ze^{-2};r+2) \\ P_{k-1}(z^{-1};r+2) \\ P_{k-1}(-z^{-1};r+2)
	\end{pmatrix} \\ & = -\int_{\T} \frac{\dd \ze_1}{2 \pi \ic \ze_1} \cdots \int_{\T} \frac{\dd \ze_n}{2 \pi \ic \ze_n} \prod_{j=1}^{n}w(\ze_j)\ze^{-r-2}_j \underset{\substack{1 \leq j \leq n \\ 1 \leq k\leq n+2}}{\det} \begin{pmatrix}
	P_{k-1}(\ze_j;r+2) \\ P_{k-1}(z^{-1};r+2) \\ P_{k-1}(-z^{-1};r+2)
	\end{pmatrix} \prod_{j=1}^{n-1} Q_{j-1}(\ze^{-2}_j;r+2) \cdot Q_{n}(\ze^{-2}_{n};r+2) \oeq 				
	\end{split}
\end{equation}
		
\noindent Notice that	
\begin{equation}
	\prod_{1\leq j<k\leq n} (\ze_k-\ze_j) = \underset{1\leq j,k\leq n}{\det} \{\ze^{k-1}_j\}= \underset{1\leq j,k\leq n}{\det} \{P_{k-1}(\ze_j)\},
\end{equation}
	
\begin{equation}
	\prod_{1\leq j<k\leq n} (\ze^{-2}_k-\ze^{-2}_j)= \underset{1\leq j,k\leq n}{\det} \{\ze^{-2(k-1)}_j\}= \underset{1\leq j,k\leq n}{\det} \{Q_{k-1}(\ze^{-2}_j)\},
\end{equation}
and
\begin{equation}
	\begin{split}
	& \frac{1}{n!} \int_{\T} \frac{\dd \ze_1}{2 \pi \ic \ze_1}\int_{\T} \frac{\dd \ze_2}{2 \pi \ic \ze_2} \cdots \int_{\T} \frac{\dd \ze_n}{2 \pi \ic \ze_n} \prod_{j=1}^{n}w(\ze_j) \ze^{-r}_j \underset{1\leq j,k\leq n}{\det} \{P_{k-1}(\ze_j)\} \underset{1\leq j,k\leq n}{\det} \{Q_{k-1}(\ze^{-2}_j)\}  \\ & = \frac{1}{n!} \int_{\T} \frac{\dd \ze_1}{2 \pi \ic \ze_1}\int_{\T} \frac{\dd \ze_2}{2 \pi \ic \ze_2} \cdots \int_{\T} \frac{\dd \ze_n}{2 \pi \ic \ze_n} \prod_{j=1}^{n}w(\ze_j) \ze^{-r}_j \underset{1\leq j,k\leq n}{\det} \{P_{k-1}(\ze_j)\} \sum_{\sigma \in S_n} \mbox{sgn}(\sigma) \prod_{j=1}^n Q_{j-1}(\ze^{-2}_{\sigma_j})  \\ & = \frac{1}{n!} \sum_{\sigma \in S_n}  \int_{\T} \frac{\dd \ze_1}{2 \pi \ic \ze_1}\int_{\T} \frac{\dd \ze_2}{2 \pi \ic \ze_2} \cdots \int_{\T} \frac{\dd \ze_n}{2 \pi \ic \ze_n} \prod_{j=1}^{n}w(\ze_j) \ze^{-r}_j \left[  \mbox{sgn}(\sigma)\underset{1\leq j,k\leq n}{\det} \{P_{k-1}(\ze_j)\} \right] \prod_{j=1}^n Q_{j-1}(\ze^{-2}_{\sigma_j})  \\ & = \frac{1}{n!} \sum_{\sigma \in S_n}  \int_{\T} \frac{\dd \ze_1}{2 \pi \ic \ze_1}\int_{\T} \frac{\dd \ze_2}{2 \pi \ic \ze_2} \cdots \int_{\T} \frac{\dd \ze_n}{2 \pi \ic \ze_n} \prod_{j=1}^{n}w(\ze_j) \ze^{-r}_j   \underset{1\leq j,k\leq n}{\det} \{P_{k-1}(\ze_{\sigma_j})\}  \prod_{j=1}^n Q_{j-1}(\ze^{-2}_{\sigma_j})  \\ & = \frac{1}{n!} \sum_{\sigma \in S_n}  \int_{\T} \frac{\dd \ze_{\sigma_1}}{2 \pi \ic \ze_{\sigma_1}}\int_{\T} \frac{\dd \ze_{\sigma_2}}{2 \pi \ic \ze_{\sigma_2}} \cdots \int_{\T} \frac{\dd \ze_{\sigma_n}}{2 \pi \ic \ze_{\sigma_n}} \prod_{j=1}^{n}w(\ze_{\sigma_j}) \ze^{-r}_{\sigma_j}   \underset{1\leq j,k\leq n}{\det} \{P_{k-1}(\ze_{\sigma_j})\}  \prod_{j=1}^n Q_{j-1}(\ze^{-2}_{\sigma_j})  \\ & =  \int_{\T} \frac{\dd \ze_1}{2 \pi \ic \ze_1}\int_{\T} \frac{\dd \ze_2}{2 \pi \ic \ze_2} \cdots \int_{\T} \frac{\dd \ze_n}{2 \pi \ic \ze_n} \prod_{j=1}^{n}w(\ze_j) \ze^{-r}_j  \underset{1\leq j,k\leq n}{\det} \{P_{k-1}(\ze_j)\}  \prod_{j=1}^n Q_{j-1}(\ze^{-2}_j)
	\end{split}
\end{equation}
So we can write
\begin{equation}
	\begin{split}
	&		\oeq \frac{-1}{n!} \int_{\T} \frac{\dd \ze_1}{2 \pi \ic \ze_1} \cdots \int_{\T} \frac{\dd \ze_n}{2 \pi \ic \ze_n} \prod_{j=1}^{n}w(\ze_j)\ze^{-r-2}_j \hspace{-.3cm} \underset{\substack{1 \leq j \leq n \\ 1 \leq k\leq n+2}}{\det} \hspace{-.2cm} \begin{pmatrix}
	P_{k-1}(\ze_j;r+2) \\ P_{k-1}(z^{-1};r+2) \\ P_{k-1}(-z^{-1};r+2)
	\end{pmatrix} \hspace{-.2cm} \underset{\substack{1 \leq j \leq n \\ 1 \leq k\leq n-1}}{\det} \hspace{-.3cm} \left( Q_{k-1}(\ze^{-2}_j;r+2) , Q_{n}(\ze^{-2}_{j};r+2) \right) \\ & = \frac{-1}{n!} \int_{\T} \frac{\dd \ze_1}{2 \pi \ic \ze_1} \cdots \int_{\T} \frac{\dd \ze_n}{2 \pi \ic \ze_n} \prod_{j=1}^{n}w(\ze_j)\ze^{-r-2}_j \prod_{1\leq j<k\leq n+2} (\ze_k-\ze_j) \underset{\substack{1 \leq j \leq n \\ 1 \leq k\leq n-1}}{\det} \left( Q_{k-1}(\ze^{-2}_j;r+2) , Q_{n}(\ze^{-2}_{j};r+2) \right) \oeq 
	\end{split}
\end{equation}
where $\ze_{n+1}:= z^{-1}$ and $\ze_{n+2}:= -z^{-1}$. Now we use the same identity as in \eqref{73} to write
\begin{equation}
	\underset{\substack{1 \leq j \leq n \\ 1 \leq k\leq n-1}}{\det} \left( Q_{k-1}(\ze^{-2}_j;r+2) , Q_{n}(\ze^{-2}_{j};r+2) \right) =  \left( q^{(r+2)}_{n,1} + \sum_{\ell=1}^{n} \ze^{-2}_{\ell} \right)\underset{1 \leq j<k \leq n}{\prod} (\ze^{-2}_k-\ze^{-2}_j)
\end{equation} 
\begin{equation}
	\begin{split}
	&		\oeq  \frac{-1}{n!} \int_{\T} \frac{\dd \ze_1}{2 \pi \ic \ze_1} \cdots \int_{\T} \frac{\dd \ze_n}{2 \pi \ic \ze_n} \prod_{j=1}^{n}w(\ze_j)\ze^{-r-2}_j (\ze_{n+2}-\ze_{n+1}) \prod_{1\leq j\leq n} (\ze_{n+2}-\ze_j) (\ze_{n+1}-\ze_j) \\ & \times \left( q^{(r+2)}_{n,1} + \sum_{\ell=1}^{n} \ze^{-2}_{\ell} \right) \prod_{1\leq j<k\leq n} (\ze_k-\ze_j)  (\ze^{-2}_k-\ze^{-2}_j) \\ & = \frac{2z^{-1}}{n!} \int_{\T} \frac{\dd \ze_1}{2 \pi \ic \ze_1} \cdots \int_{\T} \frac{\dd \ze_n}{2 \pi \ic \ze_n} \prod_{j=1}^{n}w(\ze_j)\ze^{-r-2}_j  \prod_{1\leq j\leq n} (-z^{-1}-\ze_j) (z^{-1}-\ze_j) \\ & \times \left( q^{(r+2)}_{n,1} + \sum_{\ell=1}^{n} \ze^{-2}_{\ell} \right) \prod_{1\leq j<k\leq n} (\ze_k-\ze_j)  (\ze^{-2}_k-\ze^{-2}_j) \\ & = \frac{2z^{-1}}{n!} \int_{\T} \frac{\dd \ze_1}{2 \pi \ic \ze_1} \cdots \int_{\T} \frac{\dd \ze_n}{2 \pi \ic \ze_n} \prod_{j=1}^{n}w(\ze_j)\ze^{-r-2}_j  \prod_{1\leq j\leq n} (\ze^2_j-z^{-2}) \\ & \times \left( q^{(r+2)}_{n,1} + \sum_{\ell=1}^{n} \ze^{-2}_{\ell} \right) \prod_{1\leq j<k\leq n} (\ze_k-\ze_j)  (\ze^{-2}_k-\ze^{-2}_j)
	\end{split}
\end{equation}
\begin{equation}
	\begin{split}
	& = \frac{2z^{-1-2n}}{n!} \int_{\T} \frac{\dd \ze_1}{2 \pi \ic \ze_1} \cdots \int_{\T} \frac{\dd \ze_n}{2 \pi \ic \ze_n} \prod_{j=1}^{n}w(\ze_j)\ze^{-r}_j  \prod_{1\leq j\leq n} (z^{2}-\ze_j^{-2}) \\ & \times \left( q^{(r+2)}_{n,1} + \sum_{\ell=1}^{n} \ze^{-2}_{\ell} \right) \prod_{1\leq j<k\leq n} (\ze_k-\ze_j)  (\ze^{-2}_k-\ze^{-2}_j) \\ & = 2z^{-1-2n} D^{(r)}_n \left\{ q^{(r+2)}_{n,1} Q_{n}(z^2;r) - \left[ (\ze^{2}+q^{(r)}_{n+1,1})Q_{n}(z^2;r)-Q_{n+1}(z^2;r)\right] \right\} ,
	\end{split}
\end{equation}
where we have used \eqref{extended_multiple_integral} in the final step.
Thus we have arrived at 
\begin{equation}
    \begin{split}
	& \frac{h^{(r+2)}_n}{h^{(r+2)}_{n-1}} D^{(r+2)}_n \left[  P_{n+1}(-z^{-1};r+2)P_{n-1}(z^{-1};r+2) - P_{n+1}(z^{-1};r+2)P_{n-1}(-z^{-1};r+2) \right] \\ & = 2z^{-1-2n} D^{(r)}_n \left\{ q^{(r+2)}_{n,1} Q_{n}(z^2;r) - \left[ (z^{2}+q^{(r)}_{n+1,1})Q_{n}(z^2;r)-Q_{n+1}(z^2;r)\right] \right\}
	\end{split}
\end{equation}
which is equivalent to \eqref{a 2by2 determinant of P's with a gap}.
\end{proof}

\subsubsection{Proof of Proposition \ref{prop CD mixed}}

We can use the result of Lemma \ref{PP_shift_by_2} to re-write the right-hand side of the Christoffel-Darboux identity  \eqref{CD} in terms of both $P$ and $Q$ polynomials. This gives the proof of Proposition \ref{prop CD mixed}.

	Expand \eqref{CD} along the third row:
	
	\begin{equation}
	\begin{split}
	K_{n}\left(z_{2}^{2}, z_{1} ; r\right) &=\frac{1}{2} \frac{D_{n}^{(r+2)}}{D_{n+1}^{(r)}} \frac{z_2^{2 n+1}}{z_{1}^{2}-z_{2}^{-2}} \times \\ & \Big\{ 
	P_{n}\left(z_{1} ; r+2\right)\left[P_{n+1}\left(-z_{2}^{-1} ; r+2\right) P_{n+2}\left(z_{2}^{-1} ; r+2\right)-P_{n+1}\left(z_{2}^{-1} ; r+2\right) P_{n+2}\left(-z_{2}^{-1} ; r+2\right)\right]  \\ & - P_{n+1}\left(z_{1} ; r+2\right)\left[P_{n}\left(-z_{2}^{-1} ;r+2\right) P_{n+2}\left(z_{2}^{-1} ;r+2\right)-P_{n}\left(z_{2}^{-1} ; r+2\right) P_{n+2}\left(-z_{2}^{-1} ; r+2\right)\right] \\
	&  +P_{n+2}\left(z_{1} ; r+2\right)\left[P_{n}\left(-z_{2}^{-1} ; r+2\right) P_{n+1}\left(z_{2}^{-1} ; r+2\right)-P_{n}\left(z_{2}^{-1} ; r+2\right) P_{n+1}\left(-z_{2}^{-1} ; r+2\right)\right]\Big\}
	\end{split}
	\end{equation}
	Using \eqref{222} and \eqref{a 2by2 determinant of P's with a gap} we have
	\begin{equation}
	\begin{split}
	 K_{n}\left(z_{2}^{2},  z_{1} ; r\right) &  = \frac{1}{2} \frac{D_{n}^{(r+2)}}{D_{n+1}^{(r)}} \frac{z_{2}^{2 n+1}}{z_{1}^{2}-z_{2}^{-2}} \times \\ & \bigg\{   2P_{n}\left(z_{1};r+2\right)  \frac{D_{n+1}^{(r)}}{D_{n+1}^{(r+2)}} z_{2}^{-2 n-3} Q_{n+1}\left(z_{2}^{2} ;r\right)  +P_{n+2}(z_1 ;r+2) \frac{D_{n}^{(r)}}{D_{n}^{(r+2)}} 2 z_{2}^{-2 n-1} Q_{n}\left(z_{2}^{2} ; r\right)  \\
	&  +2P_{n+1}\left(z_{1} ; r+2\right)  \frac{D_{n+1}^{(r)}}{D_{n+1}^{(r+2)}}  \frac{h_{n}^{(r+2)}}{h_{n+1}^{(r+2)}} z_{2}^{-2 n-3} \left[ -\left(z_{2}^{2}-q_{n+1,1}^{(r+2)}+q_{n+2,1}^{(r)}\right) Q_{n+1}\left(z_{2}^{2} ; r\right) +Q_{n+2}\left(z_{2}^{2} ; r\right) \right] \bigg\},
	\end{split}
	\end{equation}
	so
	\begin{equation}
	\begin{split}
	K_{n}\left(z_{2}^{2}, z_{1} ; r\right)  & =   \frac{z_{2}^{-2}}{z_{1}^{2}-z_{2}^{-2}} \left\{  \frac{D_{n}^{(r+2)}}{D_{n+1}^{(r+2)}}  P_{n}\left(z_{1};r+2\right)   Q_{n+1}\left(z_{2}^{2} ;r\right) + \frac{D_{n}^{(r)}}{D_{n+1}^{(r)}} z_{2}^{2} P_{n+2}(z_1 ;r+2)  Q_{n}\left(z_{2}^{2} ; r\right) \right. \\
	& \left. + \frac{D_{n}^{(r+2)}}{D_{n+1}^{(r+2)}}  \frac{h_{n}^{(r+2)}}{h_{n+1}^{(r+2)}}  P_{n+1}\left(z_{1} ; r+2\right)  \left[ -\left(z_{2}^{2}-q_{n+1,1}^{(r+2)}+q_{n+2,1}^{(r)}\right) Q_{n+1}\left(z_{2}^{2} ; r\right) +Q_{n+2}\left(z_{2}^{2} ; r\right) \right] \right\}.
	\end{split}
	\end{equation}
	Simplifying this using \eqref{h} and replacing $z^2_2$ by $z_2$ gives the desired formula \eqref{CD1}.

\section{The Riemann-Hilbert Characterization for the $2j-k$ Systems}\label{sec 2j-k}

This section establishes the Riemann--Hilbert characterization of the
\(2j-k\) system. We first prove the forward direction, showing that the
polynomials defined by the \(2j-k\) orthogonality relations give rise to a
solution of the corresponding Riemann--Hilbert problem. We then prove the
converse direction, showing that
any solution of the Riemann--Hilbert problem recovers the orthogonal
polynomials \(P_n(z;r)\) and their associated functions \(G_n(z;r)\).

\subsection{The Forward Direction: Proof of Theorem \ref{thm P}}\label{sec forward thm P}

In what follows we assume that the symbol $w$ and $r\in \Z$ are such that $D^{(r)}_n \neq 0$ for $n\in \{2m, 2m-1,2m-2\}$ and also $D^{(r-2)}_{2m-1} \neq 0$. Then  $\boldsymbol{Y}_{m}(z;r)$ is uniquely well-defined by \eqref{YYYY}.

First, let us tabulate the asymptotic behaviors of $G_n(z;r)$ as $z \to 0$ and $z \to \infty$.  For $|z|<1$ we can write \eqref{G} as
\begin{equation}\label{asymp G zero0}
	G_n(z;r) = 2z^{-2n} \sum^{\infty}_{\ell=0} z^{2\ell} \int_{\T} P_n(\ze;r) w(\ze) \ze^{-r-2\ell}  \frac{\dd \ze}{2\pi \ic \ze},
\end{equation}
and therefore
\begin{equation}\label{asymp G zero}
	G_n(z;r) = 2 h^{(r)}_n + O(z^2), \qasq z \to 0,
\end{equation}
by \eqref{OP1}.  Notice that due to \eqref{asymp G zero} and

\begin{equation}
	\frac{P_{n}(z;r)-P_{n}(-z;r)}{2z} = P_n'(0;r) + O(z^2), \qasq z \to 0,
\end{equation}
the function $\boldsymbol{Y}_{m}(z;r)$ has no singular behavior as $z\to 0$. This confirms \textbf{RH}-$\boldsymbol{Y1}$.

For $|z|>1$ we can write \eqref{G} as \begin{equation}\label{2}
G_n(z;r) = -2 z^{-2n-2} \sum_{\ell=0}^{\infty} z^{-2 \ell} \int_{\T} P_n(\ze;r) w(\ze) \ze^{-r} \ze^{2\ell +2} \frac{\dd \ze}{2\pi \ic \ze}.
\end{equation}
Therefore
\begin{equation}\label{3}
G_n(z;r) = \mathfrak{g}^{(r)}_n z^{-2n-2} + \mathfrak{h}^{(r)}_n z^{-2n-4} + O\left(z^{-2n-6}\right), \qasq z \to \infty,
\end{equation}
where
\begin{equation}
\mathfrak{g}^{(r)}_n = -2 \int_{\T} P_n(\ze;r) w(\ze) \ze^{2-r} \frac{\dd \ze}{2\pi \ic \ze} = -2 (-1)^n \frac{D^{(r-2)}_{n+1}}{D^{(r)}_n},
\label{g-evaluation}
\end{equation}
and
\begin{equation}
	\mathfrak{h}^{(r)}_n = -2 \int_{\T} P_n(\ze;r) w(\ze) \ze^{4-r} \frac{\dd \ze}{2\pi \ic \ze} = -2 (-1)^n \frac{\overset{\circ}{D}^{(r)}_{n+1}}{D^{(r)}_n}.
\label{h-evaluation}
\end{equation}
where $\overset{\circ}{D}^{(r)}_{n}$ is the following $n \times n$ \textit{bordered}-($2j-k$) determinant where all rows except the first one have the $2j-k$ structure:
\begin{equation}\label{D circ}
    \overset{\circ}{D}^{(r)}_{n} = \det \begin{pmatrix}
w_{r-4} & w_{r-5}   & \cdots & w_{r-n-3} \\
w_{r} & w_{r-1}   & \cdots & w_{r-n+1} \\
w_{r+2}  & w_{r+1}  & \cdots & w_{r-n+3} \\
\vdots & \vdots &  \vdots & \vdots \\
w_{r+2n-4} & w_{r+2n-5} &  \cdots & w_{r+n-3}
\end{pmatrix}  
\end{equation}
The expression for $	\mathfrak{g}^{(r)}_n $ is simply obtained by integrating \eqref{OP11} against $w(\ze) \ze^{2-r} \frac{\dd \ze}{2\pi \ic \ze}$ and swapping adjacent rows ($n$ times) in order to bring the last row to the top.

Notice that
\begin{align}
\mathscr{D}_{2m}(z;r) =	\di \frac{P_{2m}(z;r)-P_{2m}(-z;r)}{2z} \ \ \ 
	&= p^{(r)}_{2m,1} \, z^{2m-2} + O\!\left(z^{2m-4}\right), 
	&& \qasq z \to \infty, \\
\mathscr{D}_{2m-1}(z;r) =	\di \frac{P_{2m-1}(z;r)-P_{2m-1}(-z;r)}{2z} 
	&= z^{2m-2} + p_{2m-1,2}z^{2m-4} + O\!\left(z^{2m-6}\right), 
	&& \qasq z \to \infty, \\
\mathscr{D}_{2m-2}(z;r) =	\di \frac{P_{2m-2}(z;r)-P_{2m-2}(-z;r)}{2z} 
	&= p^{(r)}_{2m-2,1}\, z^{2m-4} + O\!\left(z^{2m-6}\right), 
	&& \qasq z \to \infty.
\end{align}
Using these and \eqref{3} we obtain
\begin{equation}\label{Y asymp inf}
\boldsymbol{Y}_{m}(z;r)= \boldsymbol{A}^{-1}_{m}(r) \begin{pmatrix}
z^{2m} + O(z^{2m-1}) & p^{(r)}_{2m,1} z^{2m-2} + O(z^{2m-4}) & \mathfrak{g}^{(r)}_{2m} z^{-4m+2} + O\left(z^{-4m}\right) \\[10pt]
z^{2m-1} + O(z^{2m-2}) & z^{2m-2} + O(z^{2m-4}) & \mathfrak{g}^{(r)}_{2m-1} z^{-4m+2} + O\left(z^{-4m}\right)\\[10pt]
z^{2m-2} + O(z^{2m-3}) & \di p^{(r)}_{2m-2,1} z^{2m-4} + O(z^{2m-6}) & \mathfrak{g}^{(r)}_{2m-2} z^{-4m+2} + O\left(z^{-4m}\right) \\	
\end{pmatrix},
\end{equation}
Recalling \eqref{A A A}, observe that
\begin{equation}
	\boldsymbol{A}^{-1}_{m}(r) =  \begin{pmatrix}
		1 & -p^{(r)}_{2m,1} & \frac{p^{(r)}_{2m,1}\mathfrak{g}^{(r)}_{2m-1}-\mathfrak{g}^{(r)}_{2m}}{\mathfrak{g}^{(r)}_{2m-2}}  \\[8pt]
		0 & 1 & \frac{-\mathfrak{g}^{(r)}_{2m-1}}{\mathfrak{g}^{(r)}_{2m-2}}\\[8pt]
		0 & 0 & \frac{1}{\mathfrak{g}^{(r)}_{2m-2}}\\	
	\end{pmatrix},
\end{equation} and thus the columns $\boldsymbol{Y}^{(k)}_m(z;r)$, $k=1,2,3$ of  \eqref{YYYY} can respectively be written as 
\begin{equation}
\boldsymbol{Y}^{(1)}_m(z;r)=
\begin{pmatrix}
	P_{2m}(z;r) - p^{(r)}_{2m,1} P_{2m-1}(z;r) 
	+ \dfrac{p^{(r)}_{2m,1}\,\mathfrak{g}^{(r)}_{2m-1}-\mathfrak{g}^{(r)}_{2m}}{\mathfrak{g}^{(r)}_{2m-2}}\,P_{2m-2}(z;r)
	\\[12pt]
	P_{2m-1}(z;r) - \dfrac{\mathfrak{g}^{(r)}_{2m-1}}{\mathfrak{g}^{(r)}_{2m-2}}\,P_{2m-2}(z;r)
	\\[12pt]
	\dfrac{1}{\mathfrak{g}^{(r)}_{2m-2}}\,P_{2m-2}(z;r)
\end{pmatrix},
\end{equation}
\begin{equation}
\boldsymbol{Y}^{(2)}_m(z;r)=
\begin{pmatrix}
	\mathscr{D}_{2m}(z;r) - p^{(r)}_{2m,1}\,\mathscr{D}_{2m-1}(z;r)
	+ \dfrac{p^{(r)}_{2m,1}\,\mathfrak{g}^{(r)}_{2m-1}-\mathfrak{g}^{(r)}_{2m}}{\mathfrak{g}^{(r)}_{2m-2}}\,\mathscr{D}_{2m-2}(z;r)
	\\[12pt]
	\mathscr{D}_{2m-1}(z;r) - \dfrac{\mathfrak{g}^{(r)}_{2m-1}}{\mathfrak{g}^{(r)}_{2m-2}}\,\mathscr{D}_{2m-2}(z;r)
	\\[12pt]
	\dfrac{1}{\mathfrak{g}^{(r)}_{2m-2}}\,\mathscr{D}_{2m-2}(z;r)
\end{pmatrix},
\end{equation}
and
\begin{equation}
	\boldsymbol{Y}^{(3)}_m(z;r)=
	\begin{pmatrix}
		z^{4}G_{2m}(z;r) - p^{(r)}_{2m,1}\,z^{2}G_{2m-1}(z;r) 
		+ \dfrac{p^{(r)}_{2m,1}\,\mathfrak{g}^{(r)}_{2m-1}-\mathfrak{g}^{(r)}_{2m}}{\mathfrak{g}^{(r)}_{2m-2}}\,G_{2m-2}(z;r)
		\\[12pt]
		z^{2}G_{2m-1}(z;r) - \dfrac{\mathfrak{g}^{(r)}_{2m-1}}{\mathfrak{g}^{(r)}_{2m-2}}\,G_{2m-2}(z;r)
		\\[12pt]
		\dfrac{1}{\mathfrak{g}^{(r)}_{2m-2}}\,G_{2m-2}(z;r)
	\end{pmatrix}.
\end{equation}
As $z \to \infty$ we thus have
\begin{equation}
	\boldsymbol{Y}^{(1)}_m(z;r)=
	\begin{pmatrix}
		z^{2m} + O(z^{2m-2})
		\\[12pt]
		z^{2m-1} + \left(p^{(r)}_{2m-1,1} - \dfrac{\mathfrak{g}^{(r)}_{2m-1}}{\mathfrak{g}^{(r)}_{2m-2}} \right) z^{2m-2} + O(z^{2m-3})
		\\[12pt]
	\dfrac{1}{\mathfrak{g}^{(r)}_{2m-2}}	z^{2m-2} + O(z^{2m-3})
	\end{pmatrix},
\end{equation}
\begin{equation}
	\boldsymbol{Y}^{(2)}_m(z;r)=
	\begin{pmatrix}
		 O\!\left(z^{2m-4}\right)
		\\[12pt]
z^{2m-2} + \left( p^{(r)}_{2m-1,2} - \dfrac{\mathfrak{g}^{(r)}_{2m-1}}{\mathfrak{g}^{(r)}_{2m-2}}p^{(r)}_{2m-2,1} \right)z^{2m-4} + O\!\left(z^{2m-6}\right)
		\\[12pt]
	\dfrac{p^{(r)}_{2m-2,1}}{\mathfrak{g}^{(r)}_{2m-2}}\, z^{2m-4} + O\!\left(z^{2m-6}\right)
	\end{pmatrix}
\end{equation}
and
\begin{equation}
	\boldsymbol{Y}^{(3)}_m(z;r)=
	\begin{pmatrix}
	O(z^{-4m})
		\\[12pt]
 \left( \mathfrak{h}^{(r)}_{2m-1} - \dfrac{\mathfrak{g}^{(r)}_{2m-1}}{\mathfrak{g}^{(r)}_{2m-2}}\mathfrak{h}^{(r)}_{2m-2} \right)z^{-4m} +	O(z^{-4m-2})
		\\[12pt]
 z^{-4m+2} +	O(z^{-4m})
	\end{pmatrix}.
\end{equation}

The asymptotics \eqref{Y asymp inf} can be written as 

\begin{equation}
\boldsymbol{Y}_{m}(z;r)= \left( I + \begin{pmatrix}
	 O(z^{-2}) & O(z^{-2}) & O(z^{-2}) \\
	z^{-1}+O(z^{-2}) & O(z^{-2})  & O(z^{-2})\\
	\dfrac{1}{\mathfrak{g}^{(r)}_{2m-2}}	z^{-2} + O(z^{-3}) & O(z^{-2}) &  O(z^{-2}) \\	
\end{pmatrix} \right) 
\begin{pmatrix}
z^{2m} & 0 & 0 \\
0 & z^{2m-2}  & 0\\
0 & 0 &  z^{-4m+2} \\	
\end{pmatrix}, \qasq z \to \infty,
\end{equation} where $I$ is the $3\times 3 $ identity matrix. This suggests that indeed
\begin{equation}\label{*Y_1 is nice!}
	\overset{\infty}{  \boldsymbol{Y}}_1(m,r) \equiv  \begin{pmatrix}
		0 & 0 & 0 \\
		1 & 0  & 0\\
		0 & 0 &  0 \\	
	\end{pmatrix}, \qandq \overset{\infty}{\boldsymbol{Y}}_{2,31}(m,r) = \dfrac{1}{\mathfrak{g}^{(r)}_{2m-2}} = -\frac{1}{2}\,
\frac{D^{(r)}_{2m-2}}{D^{(r-2)}_{2m-1}} \neq 0;
\end{equation}
and thus confirms \textbf{RH}-$\boldsymbol{Y3}$. The structure of $	\overset{\infty}{  \boldsymbol{Y}}_1(m,r)$ is going to be fundamental in derivation of recurrence formulae for the polynomials $P_n(z;r)$ as detailed in Section \ref{sec rec rel from RHP}. 

For the purposes of showing the compatibility of the four-term recurrence relation obtained from the Riemann-Hilbert problem in Section \ref{sec rec rel from RHP}, and the one obtained in \cite{GW} we need to keep track of $	\overset{\infty}{  \boldsymbol{Y}}_2(m,r)$. In particular:
\begin{align}
\overset{\infty}{  \boldsymbol{Y}}_{2,21}(m,r) & =  p^{(r)}_{2m-1,1} - \dfrac{\mathfrak{g}^{(r)}_{2m-1}}{\mathfrak{g}^{(r)}_{2m-2}} \label{Y2,21} \\ \overset{\infty}{  \boldsymbol{Y}}_{2,22}(m,r) & =   p^{(r)}_{2m-1,2} - \dfrac{\mathfrak{g}^{(r)}_{2m-1}}{\mathfrak{g}^{(r)}_{2m-2}}p^{(r)}_{2m-2,1} \label{Y2,22} \\ \overset{\infty}{  \boldsymbol{Y}}_{2,23}(m,r) & =   \mathfrak{h}^{(r)}_{2m-1} - \dfrac{\mathfrak{g}^{(r)}_{2m-1}}{\mathfrak{g}^{(r)}_{2m-2}}\mathfrak{h}^{(r)}_{2m-2} \label{Y2,23}
\end{align}

Finally we turn our attention to \textbf{RH}-$\boldsymbol{Y2}$. Let us investigate the additive jump condition satisfied by $G_n(z;r)$. Using
\begin{equation}\label{**}
	\frac{2 \ze^2}{\ze^2-z^2} = 2 + \frac{z}{\ze-z} - \frac{z}{\ze+z},
\end{equation} 
we rewrite \eqref{G} as 
\begin{equation}\label{rewrite G}
	G_n(z;r) = z^{-2n} \int_{\T} P_n(\ze;r) w(\ze) \ze^{-r} \left( 2 + \frac{z}{\ze-z} - \frac{z}{\ze+z} \right) \frac{\dd \ze}{2\pi \ic \ze},
\end{equation}
which can be simplfied using \eqref{OP1} as
\begin{equation}\label{RewritE G}
	G_n(z;r) = z^{-2n+1} \int_{\T}   \frac{P_n(\ze;r) w(\ze) \ze^{-r}}{\ze-z}  \frac{\dd \ze}{2\pi \ic \ze} - z^{-2n+1} \int_{\T}   \frac{P_n(\ze;r) w(\ze) \ze^{-r}}{\ze+z}  \frac{\dd \ze}{2\pi \ic \ze}.
\end{equation} This is the difference of two Cauchy transforms
\begin{equation}\label{RRewritE G}
	G_n(z;r) = z^{-2n+1} \left\{ \mathcal{C}\left[P_n(\ze;r)w(\ze)\ze^{-r-1};z\right] - \mathcal{C}\left[P_n(\ze;r)w(\ze)\ze^{-r-1};-z\right] \right\}, 
\end{equation}
where
\begin{equation}
	\mathcal{C}[f;z]:= 	\frac{1}{2 \pi \ic} \int_{\T} \frac{f(\ze)}{\ze-z} \dd \ze
\end{equation} denotes the Cauchy transform of $f$ and its boundary values satisfy the so-called Plemelj-Sokhotskii formula: 
\begin{equation}
	\mathcal{C}_{+}[f;\tau]-  \mathcal{C}_{-}[f;\tau] = f(\tau), \qquad \mbox{for} \qquad \tau \in \T.
\end{equation}
Therefore, for $z \in \T$, we have
\begin{equation}\label{G addidtive jump}
	G_{n,+}(z;r)-G_{n,-}(z;r) = z^{-2n+1}  \left[  P_n(z;r) w(z) z^{-r-1} - P_n(-z;r) w(-z) (-z)^{-r-1} \right].
\end{equation}
Let $\boldsymbol{Y}_{m}$ be given by \eqref{YYYY} and let $\boldsymbol{\widehat{Y}}_{m}(z;r):=\boldsymbol{A}_m(r)\boldsymbol{Y}_{m}(z;r)$. Let $\boldsymbol{\widehat{Y}}_{m,j}$ denote the $j$-th row of $\boldsymbol{\widehat{Y}}_{m}$, $j=1,2,3$.  Using \eqref{G addidtive jump} it is straightforward to see that for $j=1,2,3$:
\begin{equation}
	\boldsymbol{\widehat{Y}}_{m,j,+}(z;r) = \boldsymbol{\widehat{Y}}_{m,j,-}(z;r)   \begin{pmatrix}
		1 & 0 & \mathcal{j}_1(z;m,r) \\
		0 & 1 & \mathcal{j}_2(z;m,r)  \\
		0 & 0 & 1\\	
	\end{pmatrix}, \qquad z \in \T,
\end{equation}
where
\begin{equation}
	\mathcal{j}_1(z;m,r) :=	z^{-4m-r+4} \left[w(z)+(-1)^rw(-z)\right], \qandq \mathcal{j}_2(z;m,r) := 2(-1)^{r+1} z^{-4m-r+5}w(-z).
\end{equation}
Therefore for the whole matrix $\boldsymbol{\widehat{Y}}_{m}$ we have the jump condition
\begin{equation}
	\boldsymbol{\widehat{Y}}_{m,+}(z;r) = \boldsymbol{\widehat{Y}}_{m,-}(z;r)   \begin{pmatrix}
		1 & 0 & \mathcal{j}_1(z;m,r) \\
		0 & 1 & \mathcal{j}_2(z;m,r)  \\
		0 & 0 & 1\\	
	\end{pmatrix}, \qquad z \in \T.
\end{equation}
Since $\boldsymbol{Y}_{m}(z;r) = \boldsymbol{A}^{-1}_m(r)\boldsymbol{\widehat{Y}}_m(z;r)$, and multiplication on the left by an analytic factor (here the constant matrix $\boldsymbol{A}^{-1}_m(r)$) does not change the jump condition, we conclude that $\boldsymbol{Y}_m(z;r)$ indeed satisfies the jump condition \eqref{jump Y} as stated in \textbf{RH}-$\boldsymbol{Y2}$.

We have thus proven Theorem \ref{thm P}.

\subsection{The Converse Direction: Proof of Theorem \ref{thm:Y converse P entries}}\label{sec converse thm P}

We first provide the proof of Lemma \ref{prop:Y column symmetry}. Set
\[
	\boldsymbol{L}(z)
	:=
	\begin{pmatrix}
		1 & 0 & 0\\
		2z & 1 & 0\\
		0 & 0 & 1
	\end{pmatrix}
\]
and define
\[
	\widetilde{\boldsymbol{Y}}_m(z;r)
	:=
	\boldsymbol{Y}_m(-z;r)\boldsymbol{L}(z).
\]
We show that \(\widetilde{\boldsymbol{Y}}_m\) solves the same Riemann--
Hilbert problem as \(\boldsymbol{Y}_m\). First, \(\widetilde{\boldsymbol{Y}}_m\) is analytic in \(\C\setminus\T\).
Write the jump matrix as
\[
	\boldsymbol{J}_{Y}(z)
	=
	\begin{pmatrix}
		1 & 0 & \mathcal{j}_1(z;m,r)\\
		0 & 1 & \mathcal{j}_2(z;m,r)\\
		0 & 0 & 1
	\end{pmatrix}.
\]
From the definitions
\[
	\mathcal{j}_1(z;m,r)
	=
	z^{-4m-r+4}\left[w(z)+(-1)^r w(-z)\right],
\]
and
\[
	\mathcal{j}_2(z;m,r)
	=
	2(-1)^{r+1}z^{-4m-r+5}w(-z),
\]
one checks directly that
\begin{equation}\label{j symmetry identities}
	\mathcal{j}_1(-z;m,r)=\mathcal{j}_1(z;m,r),
	\qquad
	\mathcal{j}_2(-z;m,r)
	=
	\mathcal{j}_2(z;m,r)+2z\mathcal{j}_1(z;m,r).
\end{equation}
Equivalently,
\begin{equation}\label{jump intertwining Y}
	\boldsymbol{J}_{Y}(-z)\boldsymbol{L}(z)
	=
	\boldsymbol{L}(z)\boldsymbol{J}_{Y}(z).
\end{equation}
Therefore, for \(z\in\T\),
\[
	\widetilde{\boldsymbol{Y}}_{m,+}(z;r)
	=
	\boldsymbol{Y}_{m,+}(-z;r)\boldsymbol{L}(z)
	=
	\boldsymbol{Y}_{m,-}(-z;r)\boldsymbol{J}_{Y}(-z)\boldsymbol{L}(z).
\]
Using \eqref{jump intertwining Y}, this becomes
\[
	\widetilde{\boldsymbol{Y}}_{m,+}(z;r)
	=
	\boldsymbol{Y}_{m,-}(-z;r)\boldsymbol{L}(z)\boldsymbol{J}_{Y}(z)
	=
	\widetilde{\boldsymbol{Y}}_{m,-}(z;r)\boldsymbol{J}_{Y}(z).
\]
So \(\widetilde{\boldsymbol{Y}}_m\) has the same jump as
\(\boldsymbol{Y}_m\). It remains to check the normalization at infinity. Let
\begin{equation}\label{Lambda}
    \boldsymbol{\Lambda}_m(z)
	:=
	\begin{pmatrix}
		z^{2m} & 0 & 0\\
		0 & z^{2m-2} & 0\\
		0 & 0 & z^{-4m+2}
	\end{pmatrix}.
\end{equation}
Since all three powers are even,
\(
	\boldsymbol{\Lambda}_m(-z)=\boldsymbol{\Lambda}_m(z).
\)
Also,
\begin{equation}\label{Lambda L Lambda^{-1}}
\boldsymbol{\Lambda}_m(z)\boldsymbol{L}(z)\boldsymbol{\Lambda}_m(z)^{-1}
	=
	I+\frac{2E_{21}}{z}
\end{equation}
where \(E_{21}\) is the elementary matrix with a single \(1\) in the
\((2,1)\)-entry. By \textbf{RH}-\(\boldsymbol{Y3}\),
\[
	\boldsymbol{Y}_m(z;r)
	=
	\left(I+\frac{E_{21}}{z}+O(z^{-2})\right)\boldsymbol{\Lambda}_m(z),
	\qquad z\to\infty.
\]
Thus
\[
	\boldsymbol{Y}_m(-z;r)
	=
	\left(I-\frac{E_{21}}{z}+O(z^{-2})\right)\boldsymbol{\Lambda}_m(z),
	\qquad z\to\infty.
\]
Therefore
\begin{equation*}
    \begin{split}
\widetilde{\boldsymbol{Y}}_m(z;r)
	& =
	\left(I-\frac{E_{21}}{z}+O(z^{-2})\right) \boldsymbol{\Lambda}_m(z) \boldsymbol{L}(z) \\ & = 	\left(I-\frac{E_{21}}{z}+O(z^{-2})\right) \left( I+\frac{2E_{21}}{z} \right)\boldsymbol{\Lambda}_m(z) \\ & = \left(I+\frac{E_{21}}{z}+O(z^{-2})\right)\boldsymbol{\Lambda}_m(z),
	\qquad z\to\infty.
    \end{split}
\end{equation*}

So \(\widetilde{\boldsymbol{Y}}_m\) satisfies the same normalization as
\(\boldsymbol{Y}_m\). By Lemma~\ref{lemma:Y uniqueness},
\[
	\widetilde{\boldsymbol{Y}}_m(z;r)=\boldsymbol{Y}_m(z;r).
\]
This proves \eqref{Y z minus z symmetry}. Now consider the first two columns of \eqref{Y z minus z symmetry}, we obtain
\[
	\boldsymbol{Y}^{(1)}_m(z;r)
	=
	\boldsymbol{Y}^{(1)}_m(-z;r)
	+
	2z\boldsymbol{Y}^{(2)}_m(-z;r),
\]
and
\[
	\boldsymbol{Y}^{(2)}_m(z;r)
	=
	\boldsymbol{Y}^{(2)}_m(-z;r).
\]
Hence
\[
	\boldsymbol{Y}^{(1)}_m(z;r)
	-
	\boldsymbol{Y}^{(1)}_m(-z;r)
	=
	2z\boldsymbol{Y}^{(2)}_m(z;r).
\]
This is exactly \eqref{Y second column from first column} and thus completes the proof of Lemma \ref{lemma:Y uniqueness}.

\begin{lemma}
\label{lemma:large z gives P orthogonality}
Let \(n\geq1\), and let \(F\) be a polynomial of degree at most \(n\). Suppose
that a function \(\mathcal{G}\), analytic for \(|z|>1\), has boundary values on
\(\T\) satisfying
\begin{equation}\label{general unshifted G jump}
	\mathcal{G}_{+}(z)-\mathcal{G}_{-}(z)
	=
	z^{-2n-r}
	\left[
		F(z)w(z)+(-1)^rF(-z)w(-z)
	\right],
	\qquad z\in\T,
\end{equation}
and suppose that
\begin{equation}\label{general unshifted G decay}
	\mathcal{G}(z)=O(z^{-2n-1}),\qquad z\to\infty.
\end{equation}
Then
\begin{equation}\label{P orthogonality from large z lemma}
	\int_{\T}
	F(\zeta)\zeta^{-2k-r}w(\zeta)
	\frac{\dd\zeta}{2\pi\ic\zeta}
	=0,
	\qquad k=0,1,\ldots,n-1.
\end{equation}
\end{lemma}

\begin{proof}
Set
\[
	\Delta_F(\zeta)
	:=
	\zeta^{-2n-r}
	\left[
		F(\zeta)w(\zeta)+(-1)^rF(-\zeta)w(-\zeta)
	\right].
\]
By the Plemelj--Sokhotskii formula
\[
	\mathcal{G}(z)=
	\frac{1}{2\pi\ic}
	\int_{\T}
	\frac{\Delta_F(\zeta)}{\zeta-z}\,\dd\zeta.
\]

For \(|z|\) large, we expand
\[
	\frac{1}{\zeta-z}
	=
	-\sum_{\ell=0}^{N-1}\frac{\zeta^\ell}{z^{\ell+1}}
	+
	\frac{\zeta^N}{z^N(\zeta-z)}.
\]
Thus
\begin{equation}\label{general Cauchy large z expansion}
	\mathcal{G}(z)
	=
	-\sum_{\ell=0}^{N-1}
	\frac{1}{z^{\ell+1}}
	\frac{1}{2\pi\ic}
	\int_{\T}\zeta^\ell\Delta_F(\zeta)\,\dd\zeta
	+
	O(z^{-N-1}).
\end{equation}
Since \(\mathcal{G}(z)=O(z^{-2n-1})\), 
\begin{equation}\label{raw moment vanishing from G}
	\frac{1}{2\pi\ic}
	\int_{\T}\zeta^\ell\Delta_F(\zeta)\,\dd\zeta=0,
	\qquad \ell=0,1,\ldots,2n-1.
\end{equation}

We now rewrite these coefficients. By definition of \(\Delta_F\),
\[
	\frac{1}{2\pi\ic}
	\int_{\T}\zeta^\ell\Delta_F(\zeta)\,\dd\zeta
	=
	I_\ell^{(1)}+I_\ell^{(2)},
\]
where
\[
	I_\ell^{(1)}
	=
	\frac{1}{2\pi\ic}
	\int_{\T}
	F(\zeta)w(\zeta)\zeta^{\ell-2n-r}\,\dd\zeta
\qandq
	I_\ell^{(2)}
	=
	\frac{(-1)^r}{2\pi\ic}
	\int_{\T}
	F(-\zeta)w(-\zeta)\zeta^{\ell-2n-r}\,\dd\zeta.
\]
In \(I_\ell^{(2)}\), make the change of variables \(\eta=-\zeta\). Then
\[
	I_\ell^{(2)}
	=
	(-1)^{\ell+1}
	\frac{1}{2\pi\ic}
	\int_{\T}
	F(\eta)w(\eta)\eta^{\ell-2n-r}\,\dd\eta.
\]
Hence
\begin{equation}\label{parity coefficient computation}
	\frac{1}{2\pi\ic}
	\int_{\T}\zeta^\ell\Delta_F(\zeta)\,\dd\zeta
	=
	\left(1+(-1)^{\ell+1}\right)
	\frac{1}{2\pi\ic}
	\int_{\T}
	F(\zeta)w(\zeta)\zeta^{\ell-2n-r}\,\dd\zeta.
\end{equation}
If \(\ell\) is even, the factor \(1+(-1)^{\ell+1}\) is zero. If \(\ell\) is
odd, the factor is \(2\). Now take
\[
	\ell=2n-2k-1,
	\qquad k=0,1,\ldots,n-1.
\]
For these values of \(\ell\), equations \eqref{raw moment vanishing from G}
and \eqref{parity coefficient computation} yield
\[
	\int_{\T}
	F(\zeta)\zeta^{-2k-r}w(\zeta)
	\frac{\dd\zeta}{2\pi\ic\zeta}
	=0,
	\qquad k=0,1,\ldots,n-1.
\]
This proves the claim.
\end{proof}
\subsubsection{Completion of the Proof of Theorem \ref{thm:Y converse P entries}}
The first column has no jump across \(\T\). Hence \(Y_{11}(z;r)\) extends to
an entire function. From \textbf{RH}-\(\boldsymbol{Y3}\),
\[
	Y_{11}(z;r)=z^{2m}+O(z^{2m-1}),\qquad z\to\infty.
\]
Thus
\[
	F(z):=Y_{11}(z;r)
\]
is a monic polynomial of degree \(2m\). By Lemma~\ref{prop:Y column symmetry},
\[
	Y_{12}(z;r)=\frac{F(z)-F(-z)}{2z}.
\]
The jump of the \((1,3)\)-entry is therefore
\[
	Y_{13,+}(z;r)-Y_{13,-}(z;r)
	=
	F(z)\mathcal{j}_1(z;m,r)
	+
	\frac{F(z)-F(-z)}{2z}\mathcal{j}_2(z;m,r).
\]
Substituting the definitions of \(\mathcal{j}_1\) and \(\mathcal{j}_2\), we
obtain
\begin{equation}\label{Y13 jump in terms of F detailed}
	Y_{13,+}(z;r)-Y_{13,-}(z;r)
	=
	z^{-4m-r+4}
	\left[
		F(z)w(z)+(-1)^rF(-z)w(-z)
	\right].
\end{equation}

To get the orthogonality conditions, we use 
\[
	\mathcal{G}_F(z):=z^{-4}Y_{13}(z;r),
	\qquad |z|>1.
\]
From \eqref{Y13 jump in terms of F detailed},
\[
	\mathcal{G}_{F,+}(z)-\mathcal{G}_{F,-}(z)
	=
	z^{-4m-r}
	\left[
		F(z)w(z)+(-1)^rF(-z)w(-z)
	\right],
	\qquad z\in\T.
\]
The normalization at infinity gives
\[
	\mathcal{G}_F(z)=O(z^{-4m-1}),
	\qquad z\to\infty.
\]
Therefore Lemma~\ref{lemma:large z gives P orthogonality}, with
\(n=2m\), gives
\begin{equation}\label{F top P orthogonality}
 	\int_{\T}
	F(\zeta)\zeta^{-2k-r}w(\zeta)
	\frac{\dd\zeta}{2\pi\ic\zeta}
	=0,
	\qquad k=0,1,\ldots,2m-1.   
\end{equation}
Thus \(F\) satisfies the defining \(2j-k\) orthogonality conditions in
degree \(2m\).

We now prove that \(D_{2m}^{(r)}\neq0\). If \(D_{2m}^{(r)}=0\), then the
linear system for the coefficients of a monic degree \(2m\) polynomial
satisfying \eqref{F top P orthogonality} is not uniquely solvable. Hence one
can find a second monic degree \(2m\) polynomial satisfying the same
orthogonality conditions. Repeating the same Cauchy-transform construction
with this second polynomial gives another solution of the same
\(\boldsymbol{Y}\)-RHP. This contradicts
Lemma~\ref{lemma:Y uniqueness}. Therefore
\[
	D_{2m}^{(r)}\neq0.
\]
By Theorem~\ref{P and Q exist and are unique}, the monic polynomial
satisfying \eqref{F top P orthogonality} is unique. Hence
\[
	F(z)=P_{2m}(z;r).
\]

Now we prove the statment about $P_{2m-1}$. The functions \(Y_{21}\) and \(Y_{31}\)
are entire functions. For the second row, the normalization gives
\[
	Y_{21}(z;r)=z^{2m-1}+O(z^{2m-2}),
	\qquad z\to\infty.
\]
Thus \(Y_{21}\) is a monic polynomial of degree \(2m-1\). By
Proposition~\ref{prop:Y column symmetry},
\[
	Y_{22}(z;r)
	=
	\frac{Y_{21}(z;r)-Y_{21}(-z;r)}{2z}.
\]
The jump of \(Y_{23}\) gives
\[
	Y_{23,+}(z;r)-Y_{23,-}(z;r)
	=
	z^{-4m-r+4}
	\left[
		Y_{21}(z;r)w(z)+(-1)^rY_{21}(-z;r)w(-z)
	\right].
\]
Now define,
\[
	\mathcal{G}_{21}(z):=z^{-2}Y_{23}(z;r),
	\qquad |z|>1.
\]
Then
\[
	\mathcal{G}_{21,+}(z)-\mathcal{G}_{21,-}(z)
	=
	z^{-4m-r+2}
	\left[
		Y_{21}(z;r)w(z)+(-1)^rY_{21}(-z;r)w(-z)
	\right].
\]
This is the form in Lemma~\ref{lemma:large z gives P orthogonality} with
\[
	n=2m-1.
\]
The normalization at infinity gives
\[
	\mathcal{G}_{21}(z)=O(z^{-4m+1}),
	\qquad z\to\infty.
\]
Therefore
\[
	\int_{\T}
	Y_{21}(\zeta;r)\zeta^{-2k-r}w(\zeta)
	\frac{\dd\zeta}{2\pi\ic\zeta}
	=0,
	\qquad k=0,1,\ldots,2m-2.
\]
Thus \(Y_{21}\) satisfies the defining \(2j-k\) orthogonality conditions in
degree \(2m-1\). The same uniqueness argument as above gives
\[
	D_{2m-1}^{(r)}\neq0.
\]
Hence
\[
	Y_{21}(z;r)=P_{2m-1}(z;r).
\]

Finally, we prove the statement about \(P_{2m-2}\). The function
\(Y_{31}\) is entire. From the normalization at infinity,
\[
	Y_{31}(z;r)=O(z^{2m-2}),
	\qquad z\to\infty.
\]
More precisely, since
\[
	\overset{\infty}{\boldsymbol{Y}}_{1,31}(m,r)=0,
\]
we have
\[
	Y_{31}(z;r)
	=
	\overset{\infty}{\boldsymbol{Y}}_{2,31}(m,r)z^{2m-2}
	+
	O(z^{2m-3}),
	\qquad z\to\infty.
\]
By the additional condition in \textbf{RH}-\(\boldsymbol{Y3}\),
\[
	\overset{\infty}{\boldsymbol{Y}}_{2,31}(m,r)\neq0.
\]
Hence
\[
	\widetilde P_{2m-2}(z;r)
	:=
	\frac{1}{\overset{\infty}{\boldsymbol{Y}}_{2,31}(m,r)}
	Y_{31}(z;r)
\]
is a monic polynomial of degree \(2m-2\).

By Lemma~\ref{prop:Y column symmetry},
\[
	Y_{32}(z;r)
	=
	\frac{Y_{31}(z;r)-Y_{31}(-z;r)}{2z}.
\]
Therefore the jump of \(Y_{33}\) is
\[
	Y_{33,+}(z;r)-Y_{33,-}(z;r)
	=
	Y_{31}(z;r)\mathcal{j}_1(z;m,r)
	+
	\frac{Y_{31}(z;r)-Y_{31}(-z;r)}{2z}
	\mathcal{j}_2(z;m,r).
\]
Substituting the definitions of \(\mathcal{j}_1\) and \(\mathcal{j}_2\), we
obtain
\[
	Y_{33,+}(z;r)-Y_{33,-}(z;r)
	=
	z^{-4m-r+4}
	\left[
		Y_{31}(z;r)w(z)+(-1)^rY_{31}(-z;r)w(-z)
	\right].
\]
This is the form in Lemma~\ref{lemma:large z gives P orthogonality} with
\[
	n=2m-2.
\]
The normalization at infinity gives
\[
	Y_{33}(z;r)=O(z^{-4m+2}),
	\qquad z\to\infty,
\]
which is stronger than the decay required in
Lemma~\ref{lemma:large z gives P orthogonality}. Therefore
\[
	\int_{\T}
	Y_{31}(\zeta;r)\zeta^{-2k-r}w(\zeta)
	\frac{\dd\zeta}{2\pi\ic\zeta}
	=0,
	\qquad k=0,1,\ldots,2m-3.
\]
Dividing by \(\overset{\infty}{\boldsymbol{Y}}_{2,31}(m,r)\), we get
\[
	\int_{\T}
	\widetilde P_{2m-2}(\zeta;r)\zeta^{-2k-r}w(\zeta)
	\frac{\dd\zeta}{2\pi\ic\zeta}
	=0,
	\qquad k=0,1,\ldots,2m-3.
\]
Thus \(\widetilde P_{2m-2}\) is a monic polynomial of degree \(2m-2\)
satisfying the defining \(2j-k\) orthogonality conditions.

We now prove that \(D_{2m-2}^{(r)}\neq0\). If \(D_{2m-2}^{(r)}=0\), then
the linear system for the coefficients of a monic degree \(2m-2\) polynomial
satisfying the above orthogonality conditions would not have a unique
solution. Hence one could find a second monic polynomial of degree
\(2m-2\) satisfying the same orthogonality conditions. Repeating the same
Cauchy-transform construction with this second polynomial gives another
solution of the same \(\boldsymbol{Y}\)-RHP, contradicting
Lemma~\ref{lemma:Y uniqueness}. Therefore
\[
	D_{2m-2}^{(r)}\neq0.
\]
By Theorem~\ref{P and Q exist and are unique}, the monic polynomial
satisfying the defining orthogonality conditions in degree \(2m-2\) is
unique. Hence
\[
	\widetilde P_{2m-2}(z;r)=P_{2m-2}(z;r).
\]

We have now proved
\[
	D_{2m}^{(r)}D_{2m-1}^{(r)}D_{2m-2}^{(r)}\neq0.
\]
Consequently, the polynomials
\[
	P_{2m}(z;r),\qquad P_{2m-1}(z;r),\qquad P_{2m-2}(z;r)
\]
exist and are unique. In particular, the matrix \(\boldsymbol{A}_m(r)\)
defined by \eqref{A A A} is well-defined.

Set
\[
	\widehat{\boldsymbol{Y}}_m(z;r)
	:=
	\boldsymbol{A}_m(r)\boldsymbol{Y}_m(z;r).
\]
We will show that
\[
	\widehat{\boldsymbol{Y}}_m(z;r)
	=
	\boldsymbol{\Xi}_m(z;r),
\]
where
\[
	\boldsymbol{\Xi}_m(z;r)
	:=
	\begin{pmatrix}
		P_{2m}(z;r) & \mathscr{D}_{2m}(z;r) & z^4G_{2m}(z;r) \\[8pt]
		P_{2m-1}(z;r) & \mathscr{D}_{2m-1}(z;r) & z^2G_{2m-1}(z;r) \\[8pt]
		P_{2m-2}(z;r) & \mathscr{D}_{2m-2}(z;r) & G_{2m-2}(z;r)
	\end{pmatrix}.
\]

First compare the first two columns. The first two columns of
\(\widehat{\boldsymbol{Y}}_m\) have no jump across \(\T\), because the jump
matrix differs from the identity only in its third column. The first two
columns of \(\boldsymbol{\Xi}_m\) are polynomial columns. By the definition of
\(\boldsymbol{A}_m(r)\), together with the normalization of
\(\boldsymbol{Y}_m\), the first two columns of
\(\widehat{\boldsymbol{Y}}_m\) and \(\boldsymbol{\Xi}_m\) have the same
large-\(z\) expansions. Hence their difference is an entire vector-valued
function which vanishes at infinity. By Liouville's theorem, the first two
columns agree:
\[
	\widehat{\boldsymbol{Y}}_m^{(1)}(z;r)
	=
	\boldsymbol{\Xi}_m^{(1)}(z;r),
	\qquad
	\widehat{\boldsymbol{Y}}_m^{(2)}(z;r)
	=
	\boldsymbol{\Xi}_m^{(2)}(z;r).
\]

It remains to compare the third columns. Since both
\(\widehat{\boldsymbol{Y}}_m\) and \(\boldsymbol{\Xi}_m\) satisfy the same jump
relation, their difference
\[
	\boldsymbol{d}_m(z;r)
	:=
	\widehat{\boldsymbol{Y}}_m(z;r)-\boldsymbol{\Xi}_m(z;r)
\]
satisfies
\[
	\boldsymbol{d}_{m,+}(z;r)
	=
	\boldsymbol{d}_{m,-}(z;r)
	\begin{pmatrix}
		1 & 0 & \mathcal{j}_1(z;m,r)\\
		0 & 1 & \mathcal{j}_2(z;m,r)\\
		0 & 0 & 1
	\end{pmatrix}.
\]
The first two columns of \(\boldsymbol{d}_m\) are zero, as just proved.
Therefore the jump of the third column is trivial:
\[
	\boldsymbol{d}_{m,+}^{(3)}(z;r)
	=
	\boldsymbol{d}_{m,-}^{(3)}(z;r),
	\qquad z\in\T.
\]
Thus \(\boldsymbol{d}_m^{(3)}\) extends to an entire vector-valued function.
Again by the definition of \(\boldsymbol{A}_m(r)\), the normalization of
\(\boldsymbol{Y}_m\), and the large-\(z\) expansions of
\(G_{2m}\), \(G_{2m-1}\), and \(G_{2m-2}\), we have
\[
	\boldsymbol{d}_m^{(3)}(z;r)=O(z^{-1}),
	\qquad z\to\infty.
\]
Liouville's theorem gives
\[
	\boldsymbol{d}_m^{(3)}(z;r)\equiv0.
\]
Therefore
\[
	\widehat{\boldsymbol{Y}}_m(z;r)=\boldsymbol{\Xi}_m(z;r).
\]
Equivalently,
\[
	\begin{pmatrix}
		P_{2m}(z;r) & \mathscr{D}_{2m}(z;r) & z^4G_{2m}(z;r) \\[8pt]
		P_{2m-1}(z;r) & \mathscr{D}_{2m-1}(z;r) & z^2G_{2m-1}(z;r) \\[8pt]
		P_{2m-2}(z;r) & \mathscr{D}_{2m-2}(z;r) & G_{2m-2}(z;r)
	\end{pmatrix}
	=
	\boldsymbol{A}_m(r)\boldsymbol{Y}_m(z;r).
\]
This concludes the proof of Theorem~\ref{thm:Y converse P entries}.

\section{The Riemann-Hilbert Characterization for the $j-2k$ Systems}\label{sec j-2k}

This section establishes the Riemann--Hilbert characterization of the
\(j-2k\) system. We first prove the forward direction, showing that the
polynomials defined by the \(2j-k\) orthogonality relations give rise to a
solution of the corresponding Riemann--Hilbert problem, including all the requirements of \textbf{RH-X3}. We then prove the
converse direction, showing that
any solution of the Riemann--Hilbert problem recovers the orthogonal
polynomials \(S_n(z;s)\) and their associated functions \(H_n(z;s)\).

\subsection{The Forward Direction: Proof of Theorem \ref{thm Q}}\label{sec forward thm Q}

In what follows we assume that the symbol \(w\) and \(s\in\Z\) are such that
\[
	E_n^{(s)}\neq0,
	\qquad n\in\{2m,2m-1,2m-2\},
\]
and that the generic non-vanishing condition
\[
	\frac{
	E^{(s+1)}_{2m}\widehat{E}^{(s)}_{2m-1}
	-
	E^{(s+1)}_{2m-1}\widehat{E}^{(s)}_{2m}
	}{
	E^{(s)}_{2m}E^{(s)}_{2m-1}
	}
	\neq0
\]
holds. Then \(\boldsymbol{B}_m(s)\) is invertible, and hence
\(\boldsymbol{X}_m(z;s)\) is uniquely well-defined by
\[
	\boldsymbol{X}_m(z;s)
	:=
	\boldsymbol{B}_m(s)^{-1}\boldsymbol{\Phi}_m(z;s),
\]
where
\begin{equation}\label{Phi S H}
    \boldsymbol{\Phi}_m(z;s)
	:=
	\begin{pmatrix}
		S^*_{2m}(z;s) &
		\di \frac{S^*_{2m}(z;s)-S^*_{2m}(-z;s)}{2z} &
		H_{2m}(z;s) \\[10pt]
		zS^*_{2m-1}(z;s) &
		\di \frac{S^*_{2m-1}(z;s)+S^*_{2m-1}(-z;s)}{2} &
		H_{2m-1}(z;s) \\[10pt]
		z^2S^*_{2m-2}(z;s) &
		\di \frac{z\left(S^*_{2m-2}(z;s)-S^*_{2m-2}(-z;s)\right)}{2} &
		H_{2m-2}(z;s)
	\end{pmatrix}.
\end{equation}

First, let us tabulate the asymptotic behaviors of \(H_n(z;s)\) as
\(z\to0\) and \(z\to\infty\). Recall that
\[
	H_n(z;s)
	=
	\int_{\T}
	S_n(\zeta^{-1};s)w(\zeta)\zeta^{-s}
	\frac{2}{\zeta^2-z^2}
	\frac{\dd\zeta}{2\pi\ic\zeta}.
\]
For \(|z|<1\), we can write
\[
	\frac{2}{\zeta^2-z^2}
	=
	2\sum_{\ell=0}^{\infty}z^{2\ell}\zeta^{-2\ell-2}.
\]
Therefore
\begin{equation}\label{asymp H zero0}
	H_n(z;s)
	=
	2\sum_{\ell=0}^{\infty}z^{2\ell}
	\int_{\T}
	S_n(\zeta^{-1};s)w(\zeta)\zeta^{-s-2\ell-2}
	\frac{\dd\zeta}{2\pi\ic\zeta}.
\end{equation}
In particular,
\begin{equation}\label{asymp H zero}
	H_n(z;s)=\widehat H_n(s)+O(z^2),
	\qquad z\to0,
\end{equation}
where
\[
	\widehat H_n(s)
	:=
	2
	\int_{\T}
	S_n(\zeta^{-1};s)w(\zeta)\zeta^{-s-2}
	\frac{\dd\zeta}{2\pi\ic\zeta}.
\]
Notice that due to \eqref{asymp H zero} and the polynomial nature of the
first two columns of \(\boldsymbol{\Phi}_m(z;s)\), the function
\(\boldsymbol{X}_m(z;s)\) has no singular behavior as \(z\to0\). This
confirms \textbf{RH}-\(\boldsymbol{X1}\).

For \(|z|>1\), we can write
\[
	\frac{2}{\zeta^2-z^2}
	=
	-2\sum_{\ell=0}^{\infty}z^{-2\ell-2}\zeta^{2\ell}.
\]
Hence
\begin{equation}\label{asymp H inf0}
	H_n(z;s)
	=
	-2\sum_{\ell=0}^{\infty}z^{-2\ell-2}
	\int_{\T}
	S_n(\zeta^{-1};s)w(\zeta)\zeta^{2\ell-s}
	\frac{\dd\zeta}{2\pi\ic\zeta}.
\end{equation}
By the defining \(j-2k\) orthogonality conditions,
\[
	\int_{\T}
	S_n(\zeta^{-1};s)w(\zeta)\zeta^{2\ell-s}
	\frac{\dd\zeta}{2\pi\ic\zeta}
	=0,
	\qquad \ell=0,1,\ldots,n-1.
\]
Therefore
\begin{equation}\label{asymp H inf}
	H_n(z;s)
	=
	-2g_n^{(s)}z^{-2n-2}
	+
	\mathfrak{u}_n^{(s)}z^{-2n-4}
	+
	O(z^{-2n-6}),
	\qquad z\to\infty,
\end{equation}
where 
\begin{equation}\label{h X evaluation}
	g_n^{(s)}
	=
	\frac{E_{n+1}^{(s)}}{E_n^{(s)}},
\end{equation}

\begin{equation}\label{u X coefficient}
	\mathfrak{u}_n^{(s)}
	:=
	-2
	\int_{\T}
	S_n(\zeta^{-1};s)w(\zeta)\zeta^{2n+2-s}
	\frac{\dd\zeta}{2\pi\ic\zeta}.
\end{equation}

Notice that
\begin{align}
	S^*_{2m}(z;s)
	&=
	S_{2m}(0;s)z^{2m}
	+
	S'_{2m}(0;s)z^{2m-1}
	+
	O(z^{2m-2}),
	&& z\to\infty,
	\\
	zS^*_{2m-1}(z;s)
	&=
	S_{2m-1}(0;s)z^{2m}
	+
	S'_{2m-1}(0;s)z^{2m-1}
	+
	O(z^{2m-2}),
	&& z\to\infty,
	\\
	z^2S^*_{2m-2}(z;s)
	&=
	S_{2m-2}(0;s)z^{2m}
	+
	S'_{2m-2}(0;s)z^{2m-1}
	+
	O(z^{2m-2}),
	&& z\to\infty.
\end{align}
Also,
\begin{align}
	\frac{S^*_{2m}(z;s)-S^*_{2m}(-z;s)}{2z}
	&=
	S'_{2m}(0;s)z^{2m-2}
	+
	O(z^{2m-4}),
	&& z\to\infty,
	\\
	\frac{S^*_{2m-1}(z;s)+S^*_{2m-1}(-z;s)}{2}
	&=
	S'_{2m-1}(0;s)z^{2m-2}
	+
	O(z^{2m-4}),
	&& z\to\infty,
	\\
	\frac{z\left(S^*_{2m-2}(z;s)-S^*_{2m-2}(-z;s)\right)}{2}
	&=
	S'_{2m-2}(0;s)z^{2m-2}
	+
	O(z^{2m-4}),
	&& z\to\infty.
\end{align}
Using these expansions and \eqref{asymp H inf}, we obtain
\begin{equation}\label{M X asymp inf}
\boldsymbol{\Phi}_m(z;s)
=
\left(
	\boldsymbol{B}_m(s)
	+
	\frac{\boldsymbol{B}_m(s)E_{21}}{z}
	+
	O(z^{-2})
\right)
\begin{pmatrix}
	z^{2m} & 0 & 0\\
	0 & z^{2m-2} & 0\\
	0 & 0 & z^{-4m+2}
\end{pmatrix},
\qquad z\to\infty.
\end{equation}
Indeed, the leading coefficient matrix is precisely
\[
	\boldsymbol{B}_{m}(s)
	=
	\begin{pmatrix}
		S_{2m}(0;s)  &  S'_{2m}(0;s)  & 0 \\[6pt]
		S_{2m-1}(0;s) & S'_{2m-1}(0;s) & 0 \\[6pt]
		S_{2m-2}(0;s) & S'_{2m-2}(0;s) & -2g^{(s)}_{2m-2}
	\end{pmatrix}.
\]
Multiplying \eqref{M X asymp inf} on the left by
\(\boldsymbol{B}_m(s)^{-1}\), we get
\begin{equation}\label{X asymp inf}
\boldsymbol{X}_m(z;s)
=
\left(
	I+\frac{E_{21}}{z}+O(z^{-2})
\right)
\begin{pmatrix}
	z^{2m} & 0 & 0\\
	0 & z^{2m-2} & 0\\
	0 & 0 & z^{-4m+2}
\end{pmatrix},
\qquad z\to\infty.
\end{equation}
This confirms the first part of \textbf{RH}-\(\boldsymbol{X3}\), namely
\[
	\overset{\infty}{\boldsymbol{X}}_1(m,s)
	=
	\begin{pmatrix}
		0 & 0 & 0\\
		1 & 0 & 0\\
		0 & 0 & 0
	\end{pmatrix}.
\]

We now check the additional non-degeneracy condition in
\textbf{RH}-\(\boldsymbol{X3}\). Let
\[
	\boldsymbol{v}(z):=\boldsymbol{X}^{(1)}_m(z;s),
	\qandq
	\boldsymbol{v}_j
	:=
	\frac{1}{j!}
	\frac{\dd^j}{\dd z^j}\boldsymbol{v}(0),
	\qquad j=0,1,2.
\]
Since
\begin{equation}\label{v B S^*}
    \boldsymbol{v}(z)
	=
	\boldsymbol{B}_m(s)^{-1}
	\begin{pmatrix}
		S^*_{2m}(z;s)\\
		zS^*_{2m-1}(z;s)\\
		z^2S^*_{2m-2}(z;s)
	\end{pmatrix},
\end{equation}
and since \(S_n^*(0;s)=1\), we have
\[
	\begin{pmatrix}
		S^*_{2m}(z;s)\\
		zS^*_{2m-1}(z;s)\\
		z^2S^*_{2m-2}(z;s)
	\end{pmatrix}
	=
	\begin{pmatrix}
		1\\0\\0
	\end{pmatrix}
	+
	z
	\begin{pmatrix}
		*\\1\\0
	\end{pmatrix}
	+
	z^2
	\begin{pmatrix}
		*\\ *\\1
	\end{pmatrix}
	+
	O(z^3),
	\qquad z\to0.
\]
Therefore matching the coefficients of $z^0$, $z^1$, and $z^2$ in \eqref{v B S^*} yields
\begin{equation}\label{v0 v1 v2 and B}
    \boldsymbol{v}_0
	= \boldsymbol{B}^{-1}_m(s)
	\begin{pmatrix}
		1\\0\\0
	\end{pmatrix},
	\qquad
	\boldsymbol{v}_1
	=
\boldsymbol{B}^{-1}_m(s)\begin{pmatrix}
		*\\1\\0
	\end{pmatrix},
	\qquad
	\boldsymbol{v}_2
	= \boldsymbol{B}^{-1}_m(s)
	\begin{pmatrix}
		*\\ *\\1
	\end{pmatrix},
\end{equation}
and therefore
\begin{equation}\label{X rank v0v1v2 forward}
	\rank
	\begin{pmatrix}
		\boldsymbol{v}_0 & \boldsymbol{v}_1 & \boldsymbol{v}_2
	\end{pmatrix}
	=3.
\end{equation}

Next, the expansion \eqref{X asymp inf} gives
\[
	\boldsymbol{X}^{(3)}_m(z;s)
	=
	z^{-4m+2}\boldsymbol{e}_3
	+
	z^{-4m}\boldsymbol{\phi}
	+
	O(z^{-4m-2}),
	\qquad z\to\infty,
\]
for some vector \(\boldsymbol{\phi}\in\C^3\). More precisely, after
multiplication by \(\boldsymbol{B}_m(s)\), the leading coefficient and the
first subleading coefficient of the third column are
\begin{equation}\label{B e3}
\boldsymbol{B}_m(s)\boldsymbol{e}_3
	=
	\begin{pmatrix}
		0\\
		0\\
		-2g^{(s)}_{2m-2}
	\end{pmatrix},    
\end{equation}

and
\[
	\boldsymbol{B}_m(s)\boldsymbol{\phi}
	=
	\begin{pmatrix}
		0\\
		-2g^{(s)}_{2m-1}\\
		\mathfrak{u}^{(s)}_{2m-2}
	\end{pmatrix}.
\]
Since
\[
	g^{(s)}_{2m-2}
	=
	\frac{E^{(s)}_{2m-1}}{E^{(s)}_{2m-2}}\neq0,
	\qquad
	g^{(s)}_{2m-1}
	=
	\frac{E^{(s)}_{2m}}{E^{(s)}_{2m-1}}\neq0,
\]
and recalling the form of $\boldsymbol{B}_m(s)\boldsymbol{v}_0$ in \eqref{v0 v1 v2 and B}, we have
\[
	\rank
	\begin{pmatrix}
		\boldsymbol{B}_m(s)\boldsymbol{v}_0 &
		\boldsymbol{B}_m(s)\boldsymbol{e}_3 &
		\boldsymbol{B}_m(s)\boldsymbol{\phi}
	\end{pmatrix}
	=3.
\]
Since \(\boldsymbol{B}_m(s)\) is invertible, this gives
\begin{equation}\label{X rank v0 e3 a2 forward}
	\rank
	\begin{pmatrix}
		\boldsymbol{v}_0 & \boldsymbol{e}_3 & \boldsymbol{\phi}
	\end{pmatrix}
	=3.
\end{equation}
Similarly, recalling \eqref{v0 v1 v2 and B} and \eqref{B e3}
\[
	\rank
	\begin{pmatrix}
		\boldsymbol{B}_m(s)\boldsymbol{v}_0 &
		\boldsymbol{B}_m(s)\boldsymbol{v}_1 &
		\boldsymbol{B}_m(s)\boldsymbol{e}_3
	\end{pmatrix}
	=3.
\]
 Hence
\begin{equation}\label{X rank v0 v1 e3 forward}
	\rank
	\begin{pmatrix}
		\boldsymbol{v}_0 & \boldsymbol{v}_1 & \boldsymbol{e}_3
	\end{pmatrix}
	=3.
\end{equation}
The rank conditions
\eqref{X rank v0v1v2 forward}, \eqref{X rank v0 e3 a2 forward}, and
\eqref{X rank v0 v1 e3 forward} confirm the additional requirements of
\textbf{RH}-\(\boldsymbol{X3}\).

Finally, we turn our attention to \textbf{RH}-\(\boldsymbol{X2}\). Let us
investigate the additive jump condition satisfied by \(H_n(z;s)\). Using
\[
	\frac{2}{\zeta^2-z^2}
	=
	\frac{z^{-1}}{\zeta-z}
	-
	\frac{z^{-1}}{\zeta+z},
\]
we rewrite \(H_n\) as
\[
	H_n(z;s)
	=
	z^{-1}
	\int_{\T}
	\frac{S_n(\zeta^{-1};s)w(\zeta)\zeta^{-s}}
	{\zeta-z}
	\frac{\dd\zeta}{2\pi\ic\zeta}
	-
	z^{-1}
	\int_{\T}
	\frac{S_n(\zeta^{-1};s)w(\zeta)\zeta^{-s}}
	{\zeta+z}
	\frac{\dd\zeta}{2\pi\ic\zeta}.
\]
Equivalently,
\[
	H_n(z;s)
	=
	z^{-1}
	\left\{
	\mathcal{C}\left[
		S_n(\zeta^{-1};s)w(\zeta)\zeta^{-s-1};z
	\right]
	-
	\mathcal{C}\left[
		S_n(\zeta^{-1};s)w(\zeta)\zeta^{-s-1};-z
	\right]
	\right\},
\]
where
\[
	\mathcal{C}[f;z]
	:=
	\frac{1}{2\pi\ic}
	\int_{\T}\frac{f(\zeta)}{\zeta-z}\,\dd\zeta
\]
denotes the Cauchy transform. By the Plemelj--Sokhotskii formula, for
\(z\in\T\),
\begin{equation}\label{H additive jump}
	H_{n,+}(z;s)-H_{n,-}(z;s)
	=
	z^{-s-2}
	\left[
		S_n(z^{-1};s)w(z)
		+
		(-1)^sS_n(-z^{-1};s)w(-z)
	\right].
\end{equation}

For \(t=0,1,2\), set
\[
	n=2m-t,
\]
and define
\[
	F_t(z;s):=z^tS^*_{2m-t}(z;s).
\]
Then
\[
	F_t(z;s)=z^{2m}S_{2m-t}(z^{-1};s),
\qquad
	F_t(-z;s)=z^{2m}S_{2m-t}(-z^{-1};s).
\]
Hence \eqref{H additive jump} gives
\[
	H_{2m-t,+}(z;s)-H_{2m-t,-}(z;s)
	=
	z^{-2m-s-2}
	\left[
		F_t(z;s)w(z)+(-1)^sF_t(-z;s)w(-z)
	\right].
\]
On the other hand, the corresponding second-column entry in the (\(t+1\))-st
row of \(\boldsymbol{\Phi}_m\) is
\[
	\frac{F_t(z;s)-F_t(-z;s)}{2z}.
\]
Therefore
\[
	H_{2m-t,+}(z;s)-H_{2m-t,-}(z;s)
	=
	F_t(z;s)\mathcal{k}_1(z;m,s)
	+
	\frac{F_t(z;s)-F_t(-z;s)}{2z}\mathcal{k}_2(z;m,s),
\]
where
\[
	\mathcal{k}_1(z;m,s)
	:=
	z^{-2m-s-2}\left[w(z)+(-1)^sw(-z)\right],
\]
and
\[
	\mathcal{k}_2(z;m,s)
	:=
	2(-1)^{s+1}z^{-2m-s-1}w(-z).
\]
It follows that each row of \(\boldsymbol{\Phi}_m\) satisfies
\[
	\boldsymbol{\Phi}_{m,j,+}(z;s)
	=
	\boldsymbol{\Phi}_{m,j,-}(z;s)
	\begin{pmatrix}
		1 & 0 & \mathcal{k}_1(z;m,s)\\
		0 & 1 & \mathcal{k}_2(z;m,s)\\
		0 & 0 & 1
	\end{pmatrix},
	\qquad z\in\T,
\]
for \(j=1,2,3\). Therefore the whole matrix satisfies
\[
	\boldsymbol{\Phi}_{m,+}(z;s)
	=
	\boldsymbol{\Phi}_{m,-}(z;s)
	\begin{pmatrix}
		1 & 0 & \mathcal{k}_1(z;m,s)\\
		0 & 1 & \mathcal{k}_2(z;m,s)\\
		0 & 0 & 1
	\end{pmatrix},
	\qquad z\in\T.
\]
Since
\[
	\boldsymbol{X}_m(z;s)
	=
	\boldsymbol{B}_m(s)^{-1}\boldsymbol{\Phi}_m(z;s),
\]
and multiplication on the left by an analytic factor, in this case the constant
matrix \(\boldsymbol{B}_m(s)^{-1}\), does not change the jump condition, we
conclude that \(\boldsymbol{X}_m(z;s)\) satisfies
\textbf{RH}-\(\boldsymbol{X2}\).

We have thus proven Theorem~\ref{thm Q}.

\subsection{The Converse Direction: Proof of Theorem \ref{thm:X converse}}\label{sec converse thm Q}
We first prove an auxiliary lemma.
 \begin{lemma}
\label{lemma:X row reconstruction}
Assume that the \(\boldsymbol{X}\)-RHP is solvable. Let
\[
	\boldsymbol{v}(z):=\boldsymbol{X}^{(1)}_m(z;s)
\]
be the first column of \(\boldsymbol{X}_m\), and write its Taylor expansion
at the origin as
\[
	\boldsymbol{v}(z)
	=
	\boldsymbol{v}_0+z\boldsymbol{v}_1+z^2\boldsymbol{v}_2+O(z^3),
	\qquad z\to0.
\]
Also write the refined expansion of the third column as
\[
	\boldsymbol{X}^{(3)}_m(z;s)
	=
	z^{-4m+2}
	\left(
		\boldsymbol{e}_3+\frac{\boldsymbol{\phi}}{z^2}+O(z^{-4})
	\right),
	\qquad z\to\infty,
\]
where \(\boldsymbol{e}_3=(0,0,1)^T\). Then the following two statements are equivalent.

\begin{enumerate}
	\item The following non-degeneracy conditions hold:
	\begin{equation}\label{X intrinsic nondegeneracy iff}
		\rank
		\begin{pmatrix}
			\boldsymbol{v}_0 & \boldsymbol{e}_3 & \boldsymbol{\phi}
		\end{pmatrix}
		=3,
		\qquad
		\rank
		\begin{pmatrix}
			\boldsymbol{v}_0 & \boldsymbol{v}_1 & \boldsymbol{e}_3
		\end{pmatrix}
		=3,
		\qquad
		\rank
		\begin{pmatrix}
			\boldsymbol{v}_0 & \boldsymbol{v}_1 & \boldsymbol{v}_2
		\end{pmatrix}
		=3.
	\end{equation}

	\item There exists a unique constant matrix
	\[
		\boldsymbol{L}_m
		=
		\begin{pmatrix}
			\ell_0\\
			\ell_1\\
			\ell_2
		\end{pmatrix},
		\qquad
		\ell_0,\ell_1,\ell_2\in\C^{1\times3},
	\]
	such that, with
	\[
		\widehat{\boldsymbol{X}}_m(z;s)
		:=
		\boldsymbol{L}_m\boldsymbol{X}_m(z;s),
	\]
	the first column has the structure
	\begin{equation}\label{X unnormalized first column structure}
		\widehat X_{11}(z;s)=q_{2m}(z),
		\qquad
		\widehat X_{21}(z;s)=zq_{2m-1}(z),
		\qquad
		\widehat X_{31}(z;s)=z^2q_{2m-2}(z),
	\end{equation}
	where
	\[
		\deg q_n\leq n,
		\qquad
		q_n(0)=1,
	\]
	and the third column satisfies
	\begin{equation}\label{X unnormalized third column decay}
		\widehat X_{13}(z;s)=O(z^{-4m-2}),
		\qquad
		\widehat X_{23}(z;s)=O(z^{-4m}),
		\qquad
		\widehat X_{33}(z;s)=O(z^{-4m+2}),
		\qquad z\to\infty.
	\end{equation}
\end{enumerate}

When these equivalent conditions hold, the row vectors are characterized by
\[
	\ell_0\boldsymbol{v}_0=1,
	\qquad
	\ell_0\boldsymbol{e}_3=0,
	\qquad
	\ell_0\boldsymbol{\phi}=0,
\]
\[
	\ell_1\boldsymbol{v}_0=0,
	\qquad
	\ell_1\boldsymbol{v}_1=1,
	\qquad
	\ell_1\boldsymbol{e}_3=0,
\]
and
\[
	\ell_2\boldsymbol{v}_0=0,
	\qquad
	\ell_2\boldsymbol{v}_1=0,
	\qquad
	\ell_2\boldsymbol{v}_2=1.
\]
Moreover, \(\boldsymbol{L}_m\) is invertible.
\end{lemma}

\begin{proof}
We first show that the structural conditions in item \(2\) force the stated
linear conditions on the row vectors. Since the first column of
\(\widehat{\boldsymbol{X}}_m\) is
\[
	\begin{pmatrix}
		\widehat X_{11}(z;s)\\
		\widehat X_{21}(z;s)\\
		\widehat X_{31}(z;s)
	\end{pmatrix}
	=
	\begin{pmatrix}
		\ell_0\boldsymbol{v}(z)\\
		\ell_1\boldsymbol{v}(z)\\
		\ell_2\boldsymbol{v}(z)
	\end{pmatrix},
\]
the Taylor expansion of \(\boldsymbol{v}\) gives
\[
	\ell_j\boldsymbol{v}(z)
	=
	\ell_j\boldsymbol{v}_0
	+
	z\ell_j\boldsymbol{v}_1
	+
	z^2\ell_j\boldsymbol{v}_2
	+
	O(z^3).
\]

The condition
\[
	\widehat X_{11}(z;s)=q_{2m}(z),
	\qquad
	q_{2m}(0)=1,
\]
forces
\[
	\ell_0\boldsymbol{v}_0=1.
\]
Also,
\[
	\widehat X_{13}(z;s)
	=
	\ell_0\boldsymbol{X}^{(3)}_m(z;s)
	=
	z^{-4m+2}\ell_0\boldsymbol{e}_3
	+
	z^{-4m}\ell_0\boldsymbol{\phi}
	+
	O(z^{-4m-2}).
\]
Thus the decay
\[
	\widehat X_{13}(z;s)=O(z^{-4m-2})
\]
forces
\[
	\ell_0\boldsymbol{e}_3=0,
	\qandq
	\ell_0\boldsymbol{\phi}=0.
\]

Similarly, the condition
\[
	\widehat X_{21}(z;s)=zq_{2m-1}(z),
	\qquad
	q_{2m-1}(0)=1,
\]
forces
\[
	\ell_1\boldsymbol{v}_0=0,
	\qandq
	\ell_1\boldsymbol{v}_1=1.
\]
Furthermore,
\[
	\widehat X_{23}(z;s)
	=
	\ell_1\boldsymbol{X}^{(3)}_m(z;s)
	=
	z^{-4m+2}\ell_1\boldsymbol{e}_3
	+
	z^{-4m}\ell_1\boldsymbol{\phi}
	+
	O(z^{-4m-2}).
\]
The decay
\[
	\widehat X_{23}(z;s)=O(z^{-4m})
\]
forces
\[
	\ell_1\boldsymbol{e}_3=0.
\]

Finally, the condition
\[
	\widehat X_{31}(z;s)=z^2q_{2m-2}(z),
	\qquad
	q_{2m-2}(0)=1,
\]
forces
\[
	\ell_2\boldsymbol{v}_0=0,
	\qquad
	\ell_2\boldsymbol{v}_1=0,
	\qquad
	\ell_2\boldsymbol{v}_2=1.
\]
The decay condition
\[
	\widehat X_{33}(z;s)=O(z^{-4m+2})
\]
imposes no additional cancellation, since
\[
	\widehat X_{33}(z;s)
	=
	z^{-4m+2}\ell_2\boldsymbol{e}_3
	+
	z^{-4m}\ell_2\boldsymbol{\phi}
	+
	O(z^{-4m-2}).
\]

We now prove \(1\Rightarrow2\). If the rank conditions
\eqref{X intrinsic nondegeneracy iff} hold, then the three systems
\[
	\ell_0
	\begin{pmatrix}
		\boldsymbol{v}_0 & \boldsymbol{e}_3 & \boldsymbol{\phi}
	\end{pmatrix}
	=
	\begin{pmatrix}
		1 & 0 & 0
	\end{pmatrix},
\]
\[
	\ell_1
	\begin{pmatrix}
		\boldsymbol{v}_0 & \boldsymbol{v}_1 & \boldsymbol{e}_3
	\end{pmatrix}
	=
	\begin{pmatrix}
		0 & 1 & 0
	\end{pmatrix},
\]
and
\[
	\ell_2
	\begin{pmatrix}
		\boldsymbol{v}_0 & \boldsymbol{v}_1 & \boldsymbol{v}_2
	\end{pmatrix}
	=
	\begin{pmatrix}
		0 & 0 & 1
	\end{pmatrix}
\]
have unique solutions. Let these solutions be
\[
	\ell_0,\qquad \ell_1,\qquad \ell_2.
\]
We know that each entry of \(\boldsymbol{X}^{(1)}_m = \boldsymbol{v}\) is a
polynomial of degree at most \(2m\). Consequently, for any constant row vector \(\ell\in\C^{1\times3}\), the
scalar function
\[
	\ell\boldsymbol{v}(z)
	=
	\ell\boldsymbol{X}^{(1)}_m(z;s)
\]
is a polynomial of degree at most \(2m\).

Therefore, if
\[
	\ell_0\boldsymbol{v}_0=1,
\]
then
\[
	\ell_0\boldsymbol{v}(z)=q_{2m}(z),
	\qquad
	\deg q_{2m}\le 2m,
	\qquad
	q_{2m}(0)=1.
\]
If
\[
	\ell_1\boldsymbol{v}_0=0,
	\qquad
	\ell_1\boldsymbol{v}_1=1,
\]
then \(\ell_1\boldsymbol{v}(z)\) has a simple zero at \(z=0\), and hence
\[
	\ell_1\boldsymbol{v}(z)=z q_{2m-1}(z),
	\qquad
	\deg q_{2m-1}\le 2m-1,
	\qquad
	q_{2m-1}(0)=1.
\]
Finally, if
\[
	\ell_2\boldsymbol{v}_0=0,
	\qquad
	\ell_2\boldsymbol{v}_1=0,
	\qquad
	\ell_2\boldsymbol{v}_2=1,
\]
then \(\ell_2\boldsymbol{v}(z)\) has a double zero at \(z=0\), and hence
\[
	\ell_2\boldsymbol{v}(z)=z^2 q_{2m-2}(z),
	\qquad
	\deg q_{2m-2}\le 2m-2,
	\qquad
	q_{2m-2}(0)=1.
\]

The same linear conditions, together with the refined expansion of the third
column, give
\begin{equation}\label{decay X hat}
    \widehat X_{13}(z;s)=O(z^{-4m-2}),
	\qquad
	\widehat X_{23}(z;s)=O(z^{-4m}),
	\qquad
	\widehat X_{33}(z;s)=O(z^{-4m+2}).
\end{equation}
Thus item \(2\) holds.

Next we prove \(2\Rightarrow1\). Suppose that there exists a unique constant
matrix \(\boldsymbol{L}_m\) satisfying the structural conditions in item
\(2\). As shown above, those structural conditions force
\[
	\ell_0
	\begin{pmatrix}
		\boldsymbol{v}_0 & \boldsymbol{e}_3 & \boldsymbol{\phi}
	\end{pmatrix}
	=
	\begin{pmatrix}
		1 & 0 & 0
	\end{pmatrix},
\]
\[
	\ell_1
	\begin{pmatrix}
		\boldsymbol{v}_0 & \boldsymbol{v}_1 & \boldsymbol{e}_3
	\end{pmatrix}
	=
	\begin{pmatrix}
		0 & 1 & 0
	\end{pmatrix},
\]
and
\[
	\ell_2
	\begin{pmatrix}
		\boldsymbol{v}_0 & \boldsymbol{v}_1 & \boldsymbol{v}_2
	\end{pmatrix}
	=
	\begin{pmatrix}
		0 & 0 & 1
	\end{pmatrix}.
\]
If any one of the three matrices
\[
	\begin{pmatrix}
		\boldsymbol{v}_0 & \boldsymbol{e}_3 & \boldsymbol{\phi}
	\end{pmatrix},
	\qquad
	\begin{pmatrix}
		\boldsymbol{v}_0 & \boldsymbol{v}_1 & \boldsymbol{e}_3
	\end{pmatrix},
	\qquad
	\begin{pmatrix}
		\boldsymbol{v}_0 & \boldsymbol{v}_1 & \boldsymbol{v}_2
	\end{pmatrix}
\]
had rank strictly less than \(3\), then the corresponding linear system
would either have no solution or infinitely many solutions. Since a solution
exists, it would have infinitely many solutions. Replacing the corresponding
row vector by another solution would produce another constant matrix with
the same first-column structure and the same third-column decay. This would
contradict the uniqueness in item \(2\). Therefore all three matrices have
rank \(3\), which proves \eqref{X intrinsic nondegeneracy iff}.

It remains to show that \(\boldsymbol{L}_m\) is invertible. We have
\[
	\boldsymbol{L}_m
	\begin{pmatrix}
		\boldsymbol{v}_0 & \boldsymbol{v}_1 & \boldsymbol{v}_2
	\end{pmatrix}
	=
	\begin{pmatrix}
		1 & \ell_0\boldsymbol{v}_1 & \ell_0\boldsymbol{v}_2\\
		0 & 1 & \ell_1\boldsymbol{v}_2\\
		0 & 0 & 1
	\end{pmatrix}.
\]
The matrix on the right is upper triangular with determinant one. Since
\[
	\rank
	\begin{pmatrix}
		\boldsymbol{v}_0 & \boldsymbol{v}_1 & \boldsymbol{v}_2
	\end{pmatrix}
	=3,
\]
we conclude that \(\boldsymbol{L}_m\) is invertible.
\end{proof}
The following Lemma is the analogue of Lemma \ref{lemma:large z gives P orthogonality} for $j-2k$ systems. The proof follows a similar line hence we omit the details.
\begin{lemma}
\label{lemma:X large z gives S orthogonality}
Let \(t\in\{0,1,2\}\). Set
\(
	n:=2m-t,
\)
and let
\(
	F(z)=z^tq_n(z),
\)
where \(q_n\) is a polynomial satisfying
\(
	\deg q_n\leq n\), and
\(
	q_n(0)=1.
\)
Suppose that a function \(\mathcal{H}\), analytic for \(|z|>1\), has jump
\[
	\mathcal{H}_{+}(z)-\mathcal{H}_{-}(z)
	=
	z^{-2m-s-2}
	\left[
		F(z)w(z)+(-1)^sF(-z)w(-z)
	\right],
	\qquad z\in\T,
\]
and satisfies
\[
	\mathcal{H}(z)=O(z^{-2n-2}),
	\qquad z\to\infty.
\]
Define
\[
	S_n(z):=z^n q_n(z^{-1}).
\]
Then \(S_n\) is monic of degree \(n\), and
\[
	\int_{\T}
	S_n(\zeta^{-1})\zeta^{2k-s}w(\zeta)
	\frac{\dd\zeta}{2\pi\ic\zeta}
	=0,
	\qquad k=0,1,\ldots,n-1.
\]
\end{lemma}

\subsubsection{Completion of the Proof of Theorem \ref{thm:X converse}}

Let \(\boldsymbol{X}_m(z;s)\) be the solution. By
Lemma~\ref{lemma:X uniqueness}, the solution is unique, has
determinant one, and by Lemma \ref{lemma:X column symmetry} it satisfies
\[
	X_{j2}(z;s)
	=
	\frac{X_{j1}(z;s)-X_{j1}(-z;s)}{2z},
	\qquad j=1,2,3.
\]

By Lemma~\ref{lemma:X row reconstruction}, there exists an invertible
constant matrix
\[
	\boldsymbol{L}_m
	=
	\begin{pmatrix}
		\ell_0\\
		\ell_1\\
		\ell_2
	\end{pmatrix}
\]
such that
\[
	\widehat{\boldsymbol{X}}_m(z;s)
	:=
	\boldsymbol{L}_m\boldsymbol{X}_m(z;s)
\]
has first column
\[
	\widehat X_{11}(z;s)=q_{2m}(z),
	\qquad
	\widehat X_{21}(z;s)=zq_{2m-1}(z),
	\qquad
	\widehat X_{31}(z;s)=z^2q_{2m-2}(z),
\]
with
\[
	\deg q_n\leq n,
	\qquad
	q_n(0)=1,
\]
and third-column decay
\[
	\widehat X_{13}(z;s)=O(z^{-4m-2}),
	\qquad
	\widehat X_{23}(z;s)=O(z^{-4m}),
	\qquad
	\widehat X_{33}(z;s)=O(z^{-4m+2}).
\]

Since \(\boldsymbol{L}_m\) is constant, the first two columns of
\(\widehat{\boldsymbol{X}}_m\) satisfy the same symmetry:
\[
	\widehat X_{j2}(z;s)
	=
	\frac{\widehat X_{j1}(z;s)-\widehat X_{j1}(-z;s)}{2z},
	\qquad j=1,2,3.
\]
Fix \(t\in\{0,1,2\}\), and set
\[
	F_t(z):=\widehat X_{t+1,1}(z;s)=z^t q_{2m-t}(z).
\]
Notice that $\widehat{\boldsymbol{X}}_m(z;s)
	$ satisfies the same jump condition as $
	\boldsymbol{X}_m(z;s)$. The jump of the third column gives
\[
	\widehat X_{t+1,3,+}(z;s)-\widehat X_{t+1,3,-}(z;s)
	=
	F_t(z)\mathcal{k}_1(z;m,s)
	+
	\frac{F_t(z)-F_t(-z)}{2z}\mathcal{k}_2(z;m,s).
\]
Using the definitions of \(\mathcal{k}_1\) and \(\mathcal{k}_2\), this is
\[
	\widehat X_{t+1,3,+}(z;s)-\widehat X_{t+1,3,-}(z;s)
	=
	z^{-2m-s-2}
	\left[
		F_t(z)w(z)+(-1)^sF_t(-z)w(-z)
	\right].
\]
The decay estimates from Lemma~\ref{lemma:X row reconstruction} (see equation \eqref{decay X hat}) are exactly
the estimates required in Lemma~\ref{lemma:X large z gives S orthogonality}.
Therefore, defining
\[
	S_n(z;s):=z^n q_n(z^{-1}),
	\qquad n=2m,2m-1,2m-2,
\]
we obtain
\[
	\int_{\T}
	S_n(\zeta^{-1};s)\zeta^{2k-s}w(\zeta)
	\frac{\dd\zeta}{2\pi\ic\zeta}
	=0,
	\qquad k=0,1,\ldots,n-1.
\]
Thus the three polynomials
\[
	S_{2m}(z;s),\qquad S_{2m-1}(z;s),\qquad S_{2m-2}(z;s)
\]
satisfy the defining \(j-2k\) orthogonality conditions.

We now prove that the determinants are non-zero. Suppose, for one of the
degrees \(n\in\{2m,2m-1,2m-2\}\), that \(E_n^{(s)}=0\). Since a monic
solution of the degree \(n\) orthogonality problem has just been constructed,
the corresponding finite linear system would not have a unique monic
solution. Hence there would be another monic polynomial of degree \(n\)
satisfying the same orthogonality conditions. Repeating the same
Cauchy-transform construction with this second polynomial gives another
solution of the same \(\boldsymbol{X}\)-RHP, contradicting uniqueness.
Therefore
\[
	E_{2m}^{(s)}E_{2m-1}^{(s)}E_{2m-2}^{(s)}\neq0.
\]
Consequently, the above polynomials are the unique polynomials
\[
	S_{2m}(z;s),\qquad S_{2m-1}(z;s),\qquad S_{2m-2}(z;s).
\]
In particular, \(\boldsymbol{B}_m(s)\) is now well-defined.

Let \(\boldsymbol{\Phi}_m(z;s)\) be the canonical matrix \eqref{Phi S H} built from
these recovered polynomials and their associated functions \(H_n\). We claim that
\[
	\widehat{\boldsymbol{X}}_m(z;s)=\boldsymbol{\Phi}_m(z;s).
\]
The first columns agree by construction. The second columns agree by the
symmetry relation. The third columns have the same jumps, because their
jumps are determined by the first two columns. Their differences have no
jump across \(\T\), are analytic at \(z=0\), and vanish at infinity by the
decay estimates above and by \eqref{asymp H inf}. Liouville's theorem
therefore gives equality of the third columns. Hence
\[
	\widehat{\boldsymbol{X}}_m(z;s)=\boldsymbol{\Phi}_m(z;s), \qquad \mbox{in other words} \qquad \boldsymbol{L}_m\boldsymbol{X}_m(z;s)=\boldsymbol{\Phi}_m(z;s).
\]

Since \(\boldsymbol{L}_m\) is invertible and
\(\det\boldsymbol{X}_m\equiv1\), we have
\[
	\det\boldsymbol{\Phi}_m(z;s)=\det\boldsymbol{L}_m\neq0.
\]
On the other hand, recall \eqref{M X asymp inf} and \eqref{Lambda}
\[
	\boldsymbol{\Phi}_m(z;s)
	=
	\left(
		\boldsymbol{B}_m(s)+O(z^{-1})
	\right)
	\boldsymbol{\Lambda}_m(z),
	\qquad z\to\infty,
\]
where \(\det\boldsymbol{\Lambda}_m(z)=1\). Therefore
\[
	\det\boldsymbol{B}_m(s)=\det\boldsymbol{L}_m\neq0.
\]
So \(\boldsymbol{B}_m(s)\) is invertible. Finally,
\[
	\det\boldsymbol{B}_m(s)
	=
	-2g_{2m-2}^{(s)}
	\det
	\begin{pmatrix}
		S_{2m}(0;s) & S'_{2m}(0;s)\\
		S_{2m-1}(0;s) & S'_{2m-1}(0;s)
	\end{pmatrix}.
\]
Since
\[
	g_{2m-2}^{(s)}
	=
	\frac{E_{2m-1}^{(s)}}{E_{2m-2}^{(s)}}\neq0,
\]
we obtain
\[
	\det
	\begin{pmatrix}
		S_{2m}(0;s) & S'_{2m}(0;s)\\
		S_{2m-1}(0;s) & S'_{2m-1}(0;s)
	\end{pmatrix}
	\neq0.
\]
In view of \eqref{S 0 S' 0}
we obtain exactly the generic condition \eqref{generic weights}.

Since
\[
	\boldsymbol{L}_m\boldsymbol{X}_m(z;s)=\boldsymbol{\Phi}_m(z;s)
\qandq
	\det\boldsymbol{B}_m(s)=\det\boldsymbol{L}_m\neq0,
\]
the normalization at infinity yields \(\boldsymbol{L}_m=\boldsymbol{B}_m(s)\).
Therefore
\[
	\boldsymbol{B}_m(s)\boldsymbol{X}_m(z;s)=\boldsymbol{\Phi}_m(z;s).
\]
This completes the proof of Theorem \ref{thm:X converse}.

\section{Recurrence Relations from the Riemann--Hilbert Formulation}\label{sec rec rel from RHP}

One of the useful checks on the Riemann--Hilbert formulation is that it
should recover the recurrence relations of the polynomial system. This is
the analogue of deriving three-term recurrence relations from the
Riemann--Hilbert problem in the usual orthogonal polynomial theory. In the
present setting, the relevant comparison matrix is obtained by multiplying
the solution indexed by $m+1$ by the inverse of the solution indexed by $m$,
with the appropriate diagonal matrix in between. With the appropriate diagonal matrix, the jump on the unit circle disappears, and the
resulting matrix is meromorphic with controlled behavior at zero and
infinity. Reading off the polynomial entries gives the recurrence
relations. This is in the same spirit as the recurrence relations derived for the multiple orthogonal polynomials on the unit circle in \cite{MinguezCenicerosVanAssche2008}.

Let us first recall the four-term recurrence relations for the $P$ polynomials from \cite{GW}.

	\begin{equation}\label{P pure n rec 1}
		P_{n+3}(z;r) -(\de_{n+2}^{(r)}+\de_{n+1}^{(r-1)})P_{n+2}(z;r) + (\de_{n+1}^{(r-1)}\de_{n+1}^{(r)} - z^2)P_{n+1}(z;r)+ (\de_n^{(r)}+\eta^{(r-2)}_n)z^2P_{n}(z;r)=0,
	\end{equation}
	where
    \begin{equation}
        \de^{(r)}_{n} = -\frac{h^{(r-1)}_n}{h^{(r)}_{n}}, \qandq \eta^{(r)}_n = \frac{D_n^{(r+2)}D_{n+1}^{(r-1)}}{D^{(r)}_{n+1}D^{(r+1)}_n}.
    \end{equation}

Below, following \cite{MinguezCenicerosVanAssche2008} we show how we could derive this four-term recurrence relation from the Riemann-Hilbert characterization for the $2j-k$ systems. We drop the dependence of all objects on the offset $r$ for simplicity of our exposition. Consider 
\begin{equation}\label{Psi def}
	\boldsymbol{\Psi}_m(z):=   \boldsymbol{Y}_{m+1}(z) \begin{pmatrix}
		1 & 0 & 0 \\
		0 & 1 & 0 \\
		0 & 0 & z^4
	\end{pmatrix}\boldsymbol{Y}^{-1}_{m}(z),
\end{equation}
Notice that for $z \in \T$:
\begin{equation}
	\begin{split}
		\boldsymbol{\Psi}^{-1}_{m,-}(z)\boldsymbol{\Psi}_{m,+}(z) & = \boldsymbol{Y}_{m,-}(z) \begin{pmatrix}
			1 & 0 & 0 \\
			0 & 1 & 0 \\
			0 & 0 & z^{-4}
		\end{pmatrix}\boldsymbol{Y}^{-1}_{m+1,-}(z)\boldsymbol{Y}_{m+1,+}(z)\begin{pmatrix}
		1 & 0 & 0 \\
		0 & 1 & 0 \\
		0 & 0 & z^{4}
	\end{pmatrix}\boldsymbol{Y}^{-1}_{m,+}(z) \\ & = \boldsymbol{Y}_{m,-}(z) \begin{pmatrix}
	1 & 0 & 0 \\
	0 & 1 & 0 \\
	0 & 0 & z^{-4}
\end{pmatrix}\begin{pmatrix}
1 & 0 & \mathcal{j}_1(z;m+1,r) \\
0 & 1 & \mathcal{j}_2(z;m+1,r)  \\
0 & 0 & 1\\	
\end{pmatrix}\begin{pmatrix}
1 & 0 & 0 \\
0 & 1 & 0 \\
0 & 0 & z^{4}
\end{pmatrix}\boldsymbol{Y}^{-1}_{m,+}(z) \\ & = \boldsymbol{Y}_{m,-}(z) \begin{pmatrix}
1 & 0 & \mathcal{j}_1(z;m,r) \\
0 & 1 & \mathcal{j}_2(z;m,r)  \\
0 & 0 & 1\\	
\end{pmatrix}\boldsymbol{Y}^{-1}_{m,+}(z)=\boldsymbol{Y}_{m,+}(z)\boldsymbol{Y}^{-1}_{m,+}(z)=I.
	\end{split}
\end{equation}
So $\boldsymbol{\Psi}_m(z)$ has no jumps on the unit circle and is thus an entire function.  Let us consider the behavior of $\boldsymbol{\Psi}_m(z)$ as $z \to \infty$. We have

\begin{equation}
	\begin{split}
& 	\boldsymbol{\Psi}_m(z)=  \left(\di I+\frac{ \overset{\infty}{  \boldsymbol{Y}}_1(m+1)}{z}+\frac{\overset{\infty}{\boldsymbol{Y}}_2(m+1)}{z^2} + O(z^{-3})\right) \begin{pmatrix}
	z^{2m+2} & 0 & 0 \\
	0 & z^{2m}  & 0\\
	0 & 0 &  z^{-4m-2} \\	
	\end{pmatrix}  \begin{pmatrix}
	1 & 0 & 0 \\
	0 & 1 & 0 \\
	0 & 0 & z^{4}
	\end{pmatrix}
 \\ & \times \begin{pmatrix}
z^{-2m} & 0 & 0 \\
0 & z^{-2m+2}  & 0\\
0 & 0 &  z^{4m-2} \\	
\end{pmatrix} \left(\di I-\frac{ \overset{\infty}{  \boldsymbol{Y}}_1(m)}{z} +\frac{ \left(\overset{\infty}{  \boldsymbol{Y}}_1(m)\right)^2 - \overset{\infty}{  \boldsymbol{Y}}_2(m)}{z^2}+ O(z^{-3})\right) = \\ &  \left(\di I+\frac{ \overset{\infty}{  \boldsymbol{Y}}_1(m+1)}{z}+\frac{\overset{\infty}{\boldsymbol{Y}}_2(m+1)}{z^2} + O(z^{-3})\right)
 \begin{pmatrix}
z^{2} & 0 & 0 \\
0 & z^{2} & 0 \\
0 & 0 & 1
\end{pmatrix} \\ & \times  \left(\di I-\frac{ \overset{\infty}{  \boldsymbol{Y}}_1(m)}{z} +\frac{ \left(\overset{\infty}{  \boldsymbol{Y}}_1(m)\right)^2 - \overset{\infty}{  \boldsymbol{Y}}_2(m)}{z^2}+ O(z^{-3})\right)
\end{split}
\end{equation}
We will extract the polynomial part now. We have

\begin{equation}
	\begin{split}
		\boldsymbol{\Psi}_m(z) & = z^2  \begin{pmatrix}
	1 & 0 & 0 \\
	0 & 1 & 0 \\
	0 & 0  & 0 
	\end{pmatrix} + z \left[ \overset{\infty}{  \boldsymbol{Y}}_1(m+1) \begin{pmatrix}
	1 & 0 & 0 \\
	0 & 1 & 0 \\
	0 & 0  & 0 
	\end{pmatrix} - \begin{pmatrix}
	1 & 0 & 0 \\
	0 & 1 & 0 \\
	0 & 0  & 0 
	\end{pmatrix} \overset{\infty}{  \boldsymbol{Y}}_1(m)  \right] \\ & +\begin{pmatrix}
	0 & 0 & 0 \\
	0 & 0 & 0 \\
	0 & 0  & 1 
	\end{pmatrix}     - \overset{\infty}{  \boldsymbol{Y}}_1(m+1) \begin{pmatrix}
	1 & 0 & 0 \\
	0 & 1 & 0 \\
	0 & 0  & 0 
	\end{pmatrix} \overset{\infty}{  \boldsymbol{Y}}_1(m) + \overset{\infty}{  \boldsymbol{Y}}_2(m+1) \begin{pmatrix}
	1 & 0 & 0 \\
	0 & 1 & 0 \\
	0 & 0  & 0 
	\end{pmatrix}   \\ & +  \begin{pmatrix}
	1 & 0 & 0 \\
	0 & 1 & 0 \\
	0 & 0  & 0 
	\end{pmatrix} \left( \left(\overset{\infty}{  \boldsymbol{Y}}_1(m)\right)^2 - \overset{\infty}{  \boldsymbol{Y}}_2(m) \right) = \\ &  z^2  \begin{pmatrix}
	1 & 0 & 0 \\
	0 & 1 & 0 \\
	0 & 0  & 0 
	\end{pmatrix} +  \overset{\infty}{  \boldsymbol{Y}}_2(m+1) \begin{pmatrix}
	1 & 0 & 0 \\
	0 & 1 & 0 \\
	0 & 0  & 0 
	\end{pmatrix} - \begin{pmatrix}
	1 & 0 & 0 \\
	0 & 1 & 0 \\
	0 & 0  & 0 
	\end{pmatrix} \overset{\infty}{  \boldsymbol{Y}}_2(m)  +\begin{pmatrix}
	0 & 0 & 0 \\
	0 & 0 & 0 \\
	0 & 0  & 1 
	\end{pmatrix} 
	\end{split}
\end{equation}
where we have used  \textbf{RH}-$\boldsymbol{Y3}$, particularly  the structure of $\overset{\infty}{  \boldsymbol{Y}}_1(m)$ to observe that
\begin{align}
&	\overset{\infty}{  \boldsymbol{Y}}_1(m+1) \begin{pmatrix}
		1 & 0 & 0 \\
		0 & 1 & 0 \\
		0 & 0  & 0 
	\end{pmatrix} - \begin{pmatrix}
		1 & 0 & 0 \\
		0 & 1 & 0 \\
		0 & 0  & 0 
	\end{pmatrix} \overset{\infty}{  \boldsymbol{Y}}_1(m) =  \begin{pmatrix}
	0 & 0 & 0 \\
	0 & 0 & 0 \\
	0 & 0  & 0 
	\end{pmatrix}, \\  & \overset{\infty}{  \boldsymbol{Y}}_1(m+1) \begin{pmatrix}
	1 & 0 & 0 \\
	0 & 1 & 0 \\
	0 & 0  & 0 
	\end{pmatrix} \overset{\infty}{  \boldsymbol{Y}}_1(m) = \begin{pmatrix}
	0 & 0 & 0 \\
	0 & 0 & 0 \\
	0 & 0  & 0 
	\end{pmatrix}, \\ & \left(\overset{\infty}{  \boldsymbol{Y}}_1(m)\right)^2 = \begin{pmatrix}
	0 & 0 & 0 \\
	0 & 0 & 0 \\
	0 & 0  & 0 
	\end{pmatrix}.
\end{align} 
Let us introduce the notations
\begin{equation}\label{Y2jk}
	\overset{\infty}{  \boldsymbol{Y}}_2(m) \equiv \left( 	\mathscr{Y}_{jk}(m)  \right)^3_{j,k=1}, \qandq \triangle_{jk}(m) := \mathscr{Y}_{jk}(m+1) - \mathscr{Y}_{jk}(m).
\end{equation}
In these notations we can write $	\boldsymbol{\Psi}_m(z)$ as
\begin{equation}
		\boldsymbol{\Psi}_m(z) =  \begin{pmatrix}
		z^2 +	\triangle_{11}(m) & \triangle_{12}(m) & -\mathscr{Y}_{13}(m) \\
			\triangle_{21}(m) & z^2 + \triangle_{22}(m) & -\mathscr{Y}_{23}(m) \\
			\mathscr{Y}_{31}(m+1) & \mathscr{Y}_{32}(m+1)  & 1 
		\end{pmatrix}.
\end{equation}
Recalling \eqref{Psi def} we have
\begin{equation}
	 \boldsymbol{Y}_{m+1}(z) =	\boldsymbol{\Psi}_m(z)\boldsymbol{Y}_{m}(z)\begin{pmatrix}
			1 & 0 & 0 \\
			0 & 1 & 0 \\
			0 & 0 & z^{-4}
		\end{pmatrix}, 
\end{equation}
and therefore recalling \eqref{YYYY} we arrive at the equation which is going to produce the desired recurrence relation for the polynomials
\begin{equation}\label{a source of equations}\begin{split}
		&	\boldsymbol{A}^{-1}_{m+1}(r) \begin{pmatrix}
		P_{2m+2}(z;r) & \mathscr{D}_{2m+2}(z;r) & z^4G_{2m+2}(z;r) \\[10pt]
		P_{2m+1}(z;r) & \mathscr{D}_{2m+1}(z;r) & z^2G_{2m+1}(z;r) \\[10pt]
		P_{2m}(z;r) & \mathscr{D}_{2m}(z;r) & G_{2m}(z;r) \\	
	\end{pmatrix} \\ & = 	\boldsymbol{\Psi}_m(z)\boldsymbol{A}^{-1}_{m}(r) \begin{pmatrix}
	P_{2m}(z;r) & \mathscr{D}_{2m}(z;r) & z^4G_{2m}(z;r) \\[10pt]
	P_{2m-1}(z;r) & \mathscr{D}_{2m-1}(z;r) & z^2G_{2m-1}(z;r) \\[10pt]
	P_{2m-2}(z;r) & \mathscr{D}_{2m-2}(z;r) & G_{2m-2}(z;r) \\	
	\end{pmatrix}\begin{pmatrix}
		1 & 0 & 0 \\
		0 & 1 & 0 \\
		0 & 0 & z^{-4}
	\end{pmatrix}. 
\end{split}
\end{equation}
We now write the equality for $21$-entry of the above equation to arrive at a 4-term recurrence relation for the $P$ polynomials:

\begin{equation}\label{rec rel from RHP}
	\begin{split}
		P_{2m+1}(z;r)-\left(\frac{\mathfrak{g}^{(r)}_{2m+1}}{\mathfrak{g}^{(r)}_{2m}}+ \triangle_{21}(m)\right)P_{2m}(z;r)
		+\Big(\!p^{(r)}_{2m,1}\,\triangle_{21}(m)-\triangle_{22}(m)-z^2\Big)P_{2m-1}(z;r)
		\\
		+\frac{\triangle_{21}(m)\!\left( - p^{(r)}_{2m,1}\mathfrak{g}^{(r)}_{2m-1}+\mathfrak{g}^{(r)}_{2m}\right)
			+\mathfrak{g}^{(r)}_{2m-1} \triangle_{22}(m)+\mathscr{Y}_{23}(m)}
		{\mathfrak{g}^{(r)}_{2m-2}}\,P_{2m-2}(z;r) + \frac{\mathfrak{g}^{(r)}_{2m-1}}{\mathfrak{g}^{(r)}_{2m-2}} z^2 P_{2m-2}(z;r) = 0
	\end{split}
\end{equation}

In the following discussion we show that the coefficients of this recurrence relation obtained purely from the Riemann-Hilbert characterization, are identical to the recurrence coefficients in \eqref{P pure n rec 1}.

Firstly with regard to the coefficient of $z^2 P_{2m-2}(z;r)$ we observe that this is equal to
\begin{equation}
     \frac{\mathfrak{g}^{(r)}_{2m-1}}{\mathfrak{g}^{(r)}_{2m-2}} = \de_{2m}^{(r)}+\eta^{(r-2)}_{2m} ,
\label{g-ratio}
\end{equation}
as an immediate consequence of \eqref{de n r + eta n r-2}.

The coefficient of $P_{2m}(z;r)$ requires an evaluation of $\triangle_{21}(m)$. Starting with its definition \eqref{Y2jk} and employing the formula \eqref{Y2,21} evaluated at $n=2m-2$ we have
\begin{equation}
    \triangle_{21}(m) = p^{(r)}_{2m+1,1} - \dfrac{\mathfrak{g}^{(r)}_{2m+1}}{\mathfrak{g}^{(r)}_{2m}} - p^{(r)}_{2m-1,1} + \dfrac{\mathfrak{g}^{(r)}_{2m-1}}{\mathfrak{g}^{(r)}_{2m-2}} .
\end{equation}
From \eqref{z^n+2 coeffs in pure n rec rel 1}, and \eqref{g-ratio} again, we can simplify this to
\begin{equation}\label{Delta21}
    \triangle_{21}(m) = \de_{2m-1}^{(r-1)} - \eta^{(r-2)}_{2m} .
\end{equation}
Now using \eqref{g-ratio} again with $m\to m+1$ we deduce 
\begin{equation}
    \frac{\mathfrak{g}^{(r)}_{2m+1}}{\mathfrak{g}^{(r)}_{2m}} + \triangle_{21}(m) = \de_{2m}^{(r)}+\de_{2m-1}^{(r-1)} .
\end{equation}

Next we evaluate the coefficient of $P_{2m-1}(z;r)$ and for this we require $\Delta_{22}(m)$.
In this case we start with the definition \eqref{Y2jk} and the formula \eqref{Y2,22}, which yields
\begin{equation}
    \triangle_{22}(m) = p^{(r)}_{2m+1,2} - \dfrac{\mathfrak{g}^{(r)}_{2m+1}}{\mathfrak{g}^{(r)}_{2m}}p^{(r)}_{2m,1} - p^{(r)}_{2m-1,2} + \dfrac{\mathfrak{g}^{(r)}_{2m-1}}{\mathfrak{g}^{(r)}_{2m-2}}p^{(r)}_{2m-2,1} .
\end{equation}
This can be simplified by employing the next-to-leading polynomial coefficients \eqref{z^n+2 coeffs in pure n rec rel p2} along with the leading order \eqref{z^n+2 coeffs in pure n rec rel 1} and we find
\begin{equation}\label{Delta22}
	\triangle_{22}(m) = p^{(r)}_{2m,1}\,\left(\de_{2m-1}^{(r-1)} - \eta^{(r-2)}_{2m}\right) - \de_{2m-1}^{(r-1)}\de_{2m-1}^{(r)}.
\end{equation}
Now by combining this latter result and that of the preceding one \eqref{Delta21} we deduce
\begin{equation}
		p^{(r)}_{2m,1}\,\triangle_{21}(m)-\triangle_{22}(m) = \de_{2m-1}^{(r-1)}\de_{2m-1}^{(r)},
\end{equation}
which is our desired result.

Lastly, to ensure full structural consistency between \eqref{rec rel from RHP} and \eqref{P pure n rec} we need to prove that the coefficient of $P_{2m-2}(z;r)$ vanishes, namely,
\begin{equation}\label{y23}
	\mathscr{Y}_{23}(m)+\triangle_{21}(m)\!\left( - p^{(r)}_{2m,1}\mathfrak{g}^{(r)}_{2m-1}+\mathfrak{g}^{(r)}_{2m}\right)
	+\mathfrak{g}^{(r)}_{2m-1} \triangle_{22}(m) = 0.
\end{equation}
By employing our earlier results for $\triangle_{21}(m)$ \eqref{Delta21} and $\triangle_{22}(m)$ \eqref{Delta22} along with those of $\mathscr{Y}_{23}(m)$, as given in \eqref{Y2jk} and \eqref{Y2,23},
we arrive at the following expression
\begin{equation}
    \mathfrak{h}^{(r)}_{2m-1} - \dfrac{\mathfrak{g}^{(r)}_{2m-1}}{\mathfrak{g}^{(r)}_{2m-2}}\mathfrak{h}^{(r)}_{2m-2} - \mathfrak{g}^{(r)}_{2m-1}\de_{2m-1}^{(r-1)}\de_{2m-1}^{(r)} + \mathfrak{g}^{(r)}_{2m} 	\left( \de_{2m-1}^{(r-1)} - \eta^{(r-2)}_{2m} \right) .
\end{equation}
Now we utilize our evaluations of $\mathfrak{h}^{(r)}_{n}$, $\mathfrak{g}^{(r)}_{n}$, $\delta^{(r)}_{n}$ and $\eta_{n}^{(r)}$ in terms of determinants, referring to \eqref{g-evaluation}, \eqref{h-evaluation}, \eqref{delta} and \eqref{eta}.
After simplifying the ensuing expression, we now conclude that the coefficient of $P_{2m-2}(z;r)$ is proportional to
\begin{equation}
    \frac{\cD{2m}{r}}{\D{2m-1}{r}} - \frac{\D{2m}{r-2}}{\D{2m-1}{r-2}}\frac{\cD{2m-1}{r}}{\D{2m-1}{r}}
    + \frac{\D{2m+1}{r-2}\D{2m}{r-2}\D{2m-1}{r-1}}{\D{2m}{r}\D{2m-1}{r-2}\D{2m}{r-1}}
    + \frac{\D{2m+1}{r-3}}{\D{2m}{r-1}}
    - \frac{\left( \D{2m}{r-2} \right)^2}{\D{2m-1}{r-2}\D{2m}{r}} .
\end{equation}
Then using a DCI in two variations, this expression can be shown to significantly simplify to the special case $n=2m$ of
\begin{equation}
    \cD{n}{r} \D{n-1}{r-2} - \cD{n-1}{r} \D{n}{r-2} - \D{n}{r-4} \D{n-1}{r} ,
\label{P2m-2Coeff}
\end{equation}
and our next task is to derive an identity for this.
Such an identity is proved starting from a relation for $\cD{n}{r}$ in terms of the tail of the $Q_n$ polynomials
\begin{equation}
    \cD{n}{r}=(-1)^{n+1}\D{n}{r-4}Q'_n(0;r-4),
\end{equation}
which follows from \eqref{OP22}.
Next one differentiates the $n$-recurrence relation for $Q_n(z;r)$ as given in \eqref{Qn-recur} and setting $ z\to 0$, yields
\begin{equation*}
    Q'_{n+3}(0;r) + \frac{h^{(r+2)}_{n+2}}{h^{(r)}_{n+2}}Q'_{n+2}(0;r) = Q_{n+2}(0;r) 
    + \left( \frac{h^{(r+1)}_{n+2}}{h^{(r)}_{n+1}}+\frac{h^{(r+2)}_{n+2}}{h^{(r+1)}_{n+1}} \right)Q_{n+1}(0;r) + \frac{h^{(r+2)}_{n+2}}{h^{(r)}_{n}}Q_{n}(0;r) .
\end{equation*}
Furthermore we have the evaluation, from \cite[Equation (4.135)]{GW},  
\begin{equation*}
    Q_{n}(0;r) = (-1)^n \frac{\D{n}{r+2}}{\D{n}{r}} ,
\end{equation*}
which leads to the following result
\begin{multline*}
    \frac{\cD{n}{r}}{\D{n-1}{r}} - \frac{\D{n}{r-2}}{\D{n-1}{r-2}}\frac{\cD{n-1}{r}}{\D{n-1}{r}}
    - \frac{\D{n-1}{r-2}}{\D{n-1}{r}}\frac{\D{n}{r-4}}{\D{n-1}{r-4}} + \frac{\D{n}{r-3}}{\D{n-1}{r-3}}\frac{\D{n}{r-4}}{\D{n-1}{r-4}}\frac{\D{n-2}{r-2}}{\D{n-1}{r}}
\\
    + \frac{\D{n-2}{r-3}}{\D{n-1}{r-3}}\frac{\D{n}{r-2}}{\D{n-1}{r-2}}\frac{\D{n}{r-4}}{\D{n-2}{r-4}}\frac{\D{n-2}{r-2}}{\D{n-1}{r}}
    - \frac{\D{n}{r-4}}{\D{n-2}{r-4}}\frac{\D{n}{r-2}}{\D{n-1}{r-2}}\frac{\D{n-3}{r-2}}{\D{n-1}{r}} = 0 .
\end{multline*}
This result drastically simplifies with the application of three variations of the DCI,
\begin{equation*}
    \D{n}{r-2}\D{n-2}{r-1} = \D{n-1}{r-1}\D{n-1}{r-2} - \D{n-1}{r-3}\D{n-1}{r}
\end{equation*}
and gives the final result 
\begin{equation}
    \cD{n}{r} \D{n-1}{r-2} - \cD{n-1}{r} \D{n}{r-2} - \D{n}{r-4} \D{n-1}{r} = 0 .
\end{equation}
The left-hand side of this identity is precisely that of expression \eqref{P2m-2Coeff},
and thus the coefficient of $P_{2m-2}$ vanishes.

\begin{remark} If we write the equality for $31$-entry of the equation \eqref{a source of equations} we arrive at
\begin{equation}\label{aux eqn} \begin{split}
\frac{1}{\mathfrak{g}^{(r)}_{2m}}\,P_{2m}(z;r)	=\mathscr{Y}_{31}(m+1)\,P_{2m}(z;r)
	+\big(\mathscr{Y}_{32}(m+1)-p^{(r)}_{2m,1}\,\mathscr{Y}_{31}(m+1)\big)P_{2m-1}(z;r)
	\\
	+\frac{\mathscr{Y}_{31}(m+1)\!\left(p^{(r)}_{2m,1}\mathfrak{g}^{(r)}_{2m-1}-\mathfrak{g}^{(r)}_{2m}\right)
		-\mathscr{Y}_{32}(m+1)\,\mathfrak{g}^{(r)}_{2m-1}+1}
	{\mathfrak{g}^{(r)}_{2m-2}}\,P_{2m-2}(z;r).
	\end{split}
\end{equation}
Matching the coefficients of $z^{2m}$ and $z^{2m-1}$ in this equation, implies the equalities:
\begin{align}
	\mathscr{Y}_{31}(m+1) & = \frac{1}{\mathfrak{g}^{(r)}_{2m}}, \\
	\mathscr{Y}_{32}(m+1)& = p^{(r)}_{2m,1}\,\mathscr{Y}_{31}(m+1) = \frac{p^{(r)}_{2m,1}}{\mathfrak{g}^{(r)}_{2m}}.
\end{align}
These two identities, force the coefficient of $P_{2m-2}$ in \eqref{aux eqn}:
\[ \frac{\mathscr{Y}_{31}(m+1)\!\left(p^{(r)}_{2m,1}\mathfrak{g}^{(r)}_{2m-1}-\mathfrak{g}^{(r)}_{2m}\right)
	-\mathscr{Y}_{32}(m+1)\,\mathfrak{g}^{(r)}_{2m-1}+1}
{\mathfrak{g}^{(r)}_{2m-2}} \] equal to zero, as expected. So the equality for $31$-entry of the equation \eqref{a source of equations} does not produce a genuine three-term recurrence relation.\end{remark}



\section*{List of Symbols}\label{Sec list of symbols}

\footnotesize
\renewcommand{\arraystretch}{1.18}
\begin{longtable}{>{\raggedright\arraybackslash}p{0.24\textwidth}
                  >{\raggedright\arraybackslash}p{0.53\textwidth}
                  >{\raggedright\arraybackslash}p{0.16\textwidth}}
	\textbf{Symbol} & \textbf{Description} & \textbf{Definition} \\
	\hline
	\endfirsthead

	\textbf{Symbol} & \textbf{Description} & \textbf{Definition} \\
	\hline
	\endhead

	\multicolumn{3}{l}{\textbf{Basic notation}}\\[2pt]

	$\T$ & The unit circle $\{z\in\C: |z|=1\}$ & \\

	$\Z$, $\N$, $\C$ & The integers, positive integers, and complex plane & \\

	$w(z)$ & The weight function on the unit circle, with $\dd\mu(z)=w(z)\dd z$ & \\

	$w_{\ell}$ & The $\ell$-th moment of the weight, $\ell\in\Z$ & \eqref{Det} \\

	$r$ & The integer offset parameter for the \(2j-k\) system & \\

	$s$ & The integer offset parameter for the \(j-2k\) system & \\[2pt]

	\multicolumn{3}{l}{\textbf{Slanted determinants and related determinant functionals}}\\[2pt]

	$\boldsymbol{D}^{(r)}_{n}$ & The \(n\times n\) moment matrix with \(2j-k\) structure and offset \(r\) & \eqref{Det} \\

	$D^{(r)}_{n}$ & The determinant of \(\boldsymbol{D}^{(r)}_{n}\), with \(D^{(r)}_0=1\) & \eqref{Det}, \eqref{h} \\

	$\boldsymbol{E}^{(s)}_{n}$ & The \(n\times n\) moment matrix with \(j-2k\) structure and offset \(s\) & \eqref{Det E} \\

	$E^{(s)}_{n}$ & The determinant of \(\boldsymbol{E}^{(s)}_{n}\), with \(E^{(s)}_0=1\) & \eqref{Det E} \\

	$\widehat{E}^{(s)}_{n}$ & The bordered \(j-2k\) determinant entering the invertibility condition for \(\boldsymbol{B}_m(s)\) & \eqref{OP22 S E}, \eqref{generic weights} \\

	$\overset{\circ}{D}{}^{(r)}_{n}$ & The bordered \(2j-k\) determinant appearing in the second coefficient of \(G_n(z;r)\) at infinity & \eqref{D circ}, \eqref{h-evaluation} \\

	$\mathcal{D}_n[f]$ & The \(2j-k\) multiple integral functional with integrand \(f\) & \eqref{P multiple integral recalled}, \eqref{Q multiple integral recalled} \\

	$F^{(2)}_n(z_1,z_2)$ & The bi-bordered \(2j-k\) determinant with two spectral columns & \eqref{bi-bordered 2j-k monomials} \\[2pt]

	\multicolumn{3}{l}{\textbf{Bi-orthogonal polynomials and norms}}\\[2pt]

	$P_n(z;r)$ & The monic polynomial in the \(2j-k\) system associated with the first variable & \eqref{OP11}, \eqref{OP1} \\

	$Q_n(z;r)$ & The monic polynomial in the \(2j-k\) system associated with the second variable & \eqref{OP22}, \eqref{OP2} \\

	$R_n(z;s)$ & The monic polynomial in the \(j-2k\) system associated with the first variable & \eqref{OP11 R}, \eqref{OP1 R} \\

	$S_n(z;s)$ & The monic polynomial in the \(j-2k\) system associated with the second variable & \eqref{OP22 S}, \eqref{OP2 S} \\

	$f_n^*(z)$ & The reciprocal polynomial \(z^n f_n(z^{-1})\) associated with a degree \(n\) polynomial \(f_n\) & \eqref{XXXXXX} \\

	$Q_n^*(z;r)$, $S_n^*(z;s)$ & Reciprocal polynomials associated with \(Q_n(z;r)\) and \(S_n(z;s)\) & \eqref{Q* pure n rec}, \eqref{XXXXXX} \\

	$h^{(r)}_n$ & The norm of the \(n\)-th \(2j-k\) bi-orthogonal pair & \eqref{PQorth}, \eqref{h} \\

	$g^{(s)}_n$ & The norm of the \(n\)-th \(j-2k\) bi-orthogonal pair & \eqref{RSorth}, \eqref{Dets from norms 2} \\

	$p^{(r)}_{n,\ell}$ & The coefficient of \(z^{n-\ell}\) in \(P_n(z;r)\) & \eqref{P} \\

	$q^{(r)}_{n,\ell}$ & The coefficient of \(z^{n-\ell}\) in \(Q_n(z;r)\) & \eqref{Q} \\

	$\mathcal{p}^{(r)}_{n,\ell}$, $\mathcal{q}^{(r)}_{n,\ell}$ & The coefficients of \(z^\ell\) in the increasing-power expansions of \(P_n(z;r)\) and \(Q_n(z;r)\) & \eqref{polys} \\

	$\boldsymbol{Z}_n(z)$ & The monomial column vector \((1,z,\ldots,z^n)^T\) & \eqref{Vectors} \\

	$\boldsymbol{P}_n(z;r)$, $\boldsymbol{Q}_n(z;r)$ & The column vectors of \(P\)- and \(Q\)-polynomials of degrees \(0,\ldots,n\) & \eqref{Vectors} \\

	$\boldsymbol{h}^{(r)}_n$ & The diagonal matrix of \(2j-k\) norms \(h^{(r)}_0,\ldots,h^{(r)}_n\) & \eqref{h&H diag} \\

	$\boldsymbol{\mathcal{P}}^{(r)}_n$, $\boldsymbol{\mathcal{Q}}^{(r)}_n$ & The lower triangular coefficient matrices for the \(P\)- and \(Q\)-polynomial bases & \eqref{A B} \\[2pt]

	\multicolumn{3}{l}{\textbf{Associated functions, kernels, and Cauchy transforms}}\\[2pt]

	$G_n(z;r)$ & The associated Cauchy-type transform for \(P_n(z;r)\) used in the \(\boldsymbol{Y}\)-RHP & \eqref{G}, \eqref{RRewritE G} \\

	$H_n(z;s)$ & The associated Cauchy-type transform for \(S_n(z;s)\) used in the \(\boldsymbol{X}\)-RHP & \eqref{H}, \eqref{H additive jump} \\

	$\mathcal{C}[f;z]$ & The Cauchy transform of \(f\) on the unit circle & \eqref{RRewritE G} \\

	$\mathfrak{g}^{(r)}_n$ & The leading normalization coefficient in the large-\(z\) expansion of \(G_n(z;r)\) & \eqref{gg}, \eqref{g-evaluation} \\

	$\mathfrak{h}^{(r)}_n$ & The next normalization coefficient in the large-\(z\) expansion of \(G_n(z;r)\) & \eqref{h-evaluation} \\

    	$\mathfrak{u}^{(s)}_n$ & The first subleading coefficient in the large-\(z\) expansion of \(H_n(z;s)\) & \eqref{asymp H inf}, \eqref{u X coefficient} \\

	$\mathscr{D}_k(z;r)$ & The odd part of \(P_k(z;r)\) divided by \(z\), used in the second column of \(\boldsymbol{Y}_m\) & \eqref{YYYY} \\

	$K_n(z,\mathcal{z};r)$ & The reproducing kernel of the \(2j-k\) bi-orthogonal system & \eqref{RepKer3}, \eqref{CD1} \\

	$\mathscr{I}^{(k)}_{n,m}(u;r)$ & The mixed moment integral used to study multiplication by powers of \(z\) & \eqref{Ik} \\[2pt]

	\multicolumn{3}{l}{\textbf{Riemann--Hilbert data}}\\[2pt]

	$\boldsymbol{A}_m(r)$ & The constant normalization matrix in the \(2j-k\) Riemann--Hilbert problem & \eqref{A A A} \\

	$\boldsymbol{B}_m(s)$ & The constant normalization matrix in the \(j-2k\) Riemann--Hilbert problem & \eqref{BBB}, \eqref{S 0 S' 0} \\

	$\boldsymbol{Y}_m(z;r)$ & The \(3\times3\) Riemann--Hilbert solution associated with the \(2j-k\) system & \eqref{YYYY}, Theorem~\ref{thm P} \\

	$\boldsymbol{X}_m(z;s)$ & The \(3\times3\) Riemann--Hilbert solution associated with the \(j-2k\) system & \eqref{XXXXXX}, Theorem~\ref{thm Q} \\

	$\boldsymbol{\Phi}_m(z;s)$ & The unnormalized \(j-2k\) matrix built from \(S_n^*(z;s)\) and \(H_n(z;s)\) & \eqref{Phi S H} \\

	$\widehat{\boldsymbol{Y}}_m(z;r)$ & The unnormalized \(2j-k\) matrix built from \(P_n(z;r)\) and \(G_n(z;r)\) & Theorem~\ref{thm:Y converse P entries} \\

	$\widehat{\boldsymbol{X}}_m(z;s)$ & The row-reconstructed \(j-2k\) matrix \(\boldsymbol{L}_m\boldsymbol{X}_m(z;s)\) & Lemma~\ref{lemma:X row reconstruction} \\

	$\boldsymbol{L}_m$ & The constant row-reconstruction matrix with rows \(\ell_0,\ell_1,\ell_2\) & Lemma~\ref{lemma:X row reconstruction} \\

	$\boldsymbol{\Lambda}_m(z)$ & The diagonal normalization matrix in the asymptotics of $\boldsymbol{Y}_m(z;r)$ and $\boldsymbol{X}_m(z;s)$ as  $z\to\infty$ & \eqref{Lambda} \\

	$\overset{\infty}{\boldsymbol{Y}}_1(m,r)$, $\overset{\infty}{\boldsymbol{Y}}_2(m,r)$ & The first two coefficients in the expansion of \(\boldsymbol{Y}_m(z;r)\) at infinity & \eqref{Y_1 is nice!}, \eqref{Y2,21}--\eqref{Y2,23} \\

	$\overset{\infty}{\boldsymbol{X}}_1(m,s)$, $\overset{\infty}{\boldsymbol{X}}_2(m,s)$ & The first two coefficients in the expansion of \(\boldsymbol{X}_m(z;s)\) at infinity & \eqref{X_1 is nice!}, \eqref{X asymp inf} \\

	$\mathcal{j}_1(z;m,r)$, $\mathcal{j}_2(z;m,r)$ & The nontrivial jump entries in the \(\boldsymbol{Y}\)-RHP & \eqref{jump Y} \\

	$\mathcal{k}_1(z;m,s)$, $\mathcal{k}_2(z;m,s)$ & The nontrivial jump entries in the \(\boldsymbol{X}\)-RHP & Theorem~\ref{thm Q} \\

	$\boldsymbol{e}_3$ & The vector \((0,0,1)^T\) appearing in the refined \(\boldsymbol{X}\)-normalization & \eqref{X third column refined} \\

	$\boldsymbol{\phi}$ & The coefficient vector in the refined expansion of the third column of \(\boldsymbol{X}_m\) & \eqref{X third column refined} \\

	$\boldsymbol{v}(z)$, $\boldsymbol{v}_j$ & The first column of \(\boldsymbol{X}_m\) and its Taylor coefficients at the origin & \eqref{X intrinsic nondegeneracy}, Lemma~\ref{lemma:X row reconstruction} \\[2pt]

	\multicolumn{3}{l}{\textbf{Recurrence and multiplication coefficients}}\\[2pt]

	$\delta^{(r)}_n$ & A pure-degree recurrence coefficient for the \(P\)-polynomials & \eqref{delta} \\

	$\eta^{(r)}_n$ & A pure-degree recurrence coefficient for the \(P\)-polynomials & \eqref{eta} \\

	$\beta^{(r)}_n$ & A pure-degree recurrence coefficient for the \(Q\)-polynomials & \eqref{beta} \\

	$\alpha^{(r)}_n$ & A pure-degree recurrence coefficient for the \(Q\)-polynomials & \eqref{alpha} \\

	$\mathscr{A}^{(r)}_{n,j}$ & The coefficients in the expansion of \(zP_n(z;r)\) in the \(P\)-basis & \eqref{zP}, \eqref{An0}--\eqref{An2} \\

	$\mathscr{B}^{(r)}_{n,j}$ & The coefficients in the expansion of \(zQ_n(z;r)\) in the \(Q\)-basis & \eqref{zQ}, \eqref{Bn0}--\eqref{Bn2} \\

	$\mathscr{C}^{(r)}_{n,j}$ & The coefficients in the expansion of \(z^2P_n(z;r)\) in the \(P\)-basis & \eqref{z2P}, \eqref{Cn2} \\

	$\mathscr{D}^{(r)}_{n,j}$ & The coefficients in the expansion of \(z^2Q_n(z;r)\) in the \(Q\)-basis & \eqref{z2Q}, \eqref{Dn2} \\

	$\boldsymbol{\Psi}_m(z)$ & The comparison matrix used to recover recurrence relations from the \(\boldsymbol{Y}\)-RHP & \eqref{Psi def} \\

	$\mathscr{Y}_{jk}(m)$, $\triangle_{jk}(m)$ & The entries of \(\overset{\infty}{\boldsymbol{Y}}_2(m)\) and their forward \(m\)-differences & \eqref{Y2jk} \\[2pt]

	\multicolumn{3}{l}{\textbf{Dodgson condensation notation}}\\[2pt]

	$\boldsymbol{\mathscr{M}}$ & A generic square matrix used in the Dodgson condensation identity & \eqref{DODGSON} \\

	$\mathscr{M}\left\lbrace \begin{matrix} j_1&\cdots&j_\ell \\ k_1&\cdots&k_\ell \end{matrix}\right\rbrace$ &
	The determinant obtained from \(\boldsymbol{\mathscr{M}}\) by deleting rows \(j_1,\ldots,j_\ell\) and columns \(k_1,\ldots,k_\ell\) & \eqref{DODGSON} \\

\end{longtable}
\normalsize

\end{itemize}

\newpage

\bibliographystyle{plain}
\bibliography{all}

\end{document}